\documentclass[a4paper, 11pt]{article}

\usepackage[latin1]{inputenc}
\usepackage{graphics}
\usepackage{rotating,multirow}
\usepackage{slashbox}
\usepackage{amsfonts}
\usepackage{url} 
\usepackage{comment}
\usepackage{color}
\usepackage{amsthm} 
\usepackage{fullpage} 

\newtheorem{assumption}{Assumption} 
\newtheorem{theorem}{Theorem} 
\newtheorem{lemma}{Lemma} 
\newtheorem{proposition}{Proposition} 
\newtheorem{definition}{Definition} 
 
\input alphabet
\input jidef

\newcommand{\minimize}[2]{\ensuremath{\underset{\substack{{#1}}}%
{\mathrm{minimize}}\;\;#2 }}
\newtheorem{example}[theorem]{Example}
\newtheorem{remark}[theorem]{Remark}

\renewcommand{\sb}{{\boldsymbol u}}

\newcommand {\newtext[1]}{#1}

\title{A Majorize-Minimize subspace approach for
$\ell_2-\ell_0$ image regularization \thanks{A preliminary version of this work has appeared in \cite{Chouzenoux11icip}.}}
\author{Emilie Chouzenoux, Anna Jezierska, Jean-Christophe Pesquet and Hugues Talbot}

\begin{document}
\maketitle

\begin{abstract}
  In this work, we consider a class of differentiable criteria for sparse
  image computing problems, where a nonconvex regularization is applied to an
  arbitrary linear transform of the target image. As special cases, it includes
  edge preserving measures or frame-analysis potentials commonly used in image
  processing. As shown by our asymptotic results, the $\ell_2-\ell_0$ penalties
  we consider may be employed to provide approximate solutions to
  $\ell_0$-penalized optimization problems. One of the advantages of the
  proposed approach is that it allows us to derive an efficient
  Majorize-Minimize subspace algorithm. The convergence of the algorithm is
  investigated by using recent results in nonconvex optimization.  The fast
  convergence properties of the proposed optimization method are illustrated
  through image processing examples. In particular, its effectiveness is
  demonstrated on several data recovery problems.
\end{abstract}


\pagestyle{myheadings}
\thispagestyle{plain}
\markboth{ }{ }

\newpage

\section{Introduction}

 \label{sec:intro}
 The objective of this paper is to show that, for a wide range of variational
 problems in image processing, an estimation $\widehat{\xb} \in \eR^N$ of the
 target image can be efficiently obtained by using a class of nonconvex,
 regularizing criteria that promote sparsity.  More specifically, we focus on
 the following penalized optimization problem:
\begin{equation}
\minimize{\xb \in \eR^N} 
{\big(F_\delta(\xb) = \Phi(\Hb \xb - \yb) + \Psi_\delta(\xb)\big)} 
\label{Eq_CritJ},
\end{equation}
where $\Hb\neq \zerob$ is a matrix in $\eR^{Q\times N}$, \newtext[$\yb$
is a vector in  $\eR^Q$, $\Phi\colon \eR^Q \to \eR$ and $\Psi_\delta\colon \eR^N\to \eR$ are functions, and $\delta$ is a positive scalar.]
We are mainly interested in the case when $\Phi$ is a differentiable
function. This includes 
the classical squared Euclidean norm. The problem then reduces to a
penalized least squares (PLS) problem \cite{Tikhonov77,Titterington85a}.
Another case of interest is when $\Phi$ is the separable Huber
function \cite[Example 5.4]{Huber81} which is useful for limiting the influence
of outliers in some observed data. Other examples shall be mentioned subsequently.

Note that the considered optimization problem is frequently encountered in the field of inverse
problems. \newtext[Then, $\yb$ is some vector of observations related to the original
image  $\overline{\xb}\in \eR^N$ through a linear model of the form 
\begin{equation}\label{e:modellin}
\yb = \Hb
\overline{\xb} + \wb,
\end{equation}
where $\Hb$  models
the measurement process (e.g. a convolution operator or a projection operator), $\wb$ is an additive
noise vector, $\Phi$ is a data-fidelity term and $\Psi_\delta$ is a regularization term.] 

An efficient strategy to promote images formed by smooth regions separated by sharp edges, is to use regularization functions of the form
\begin{equation}\label{e:Psid}
(\forall \xb \in \eR^N)\qquad
\Psi_\delta(\xb) = \sum_{s=1}^S \psi_{s,\delta}(\|\Vb_s \xb-\cb_s\|) + 
\| \Vb_0 \xb\|^2,
\end{equation}
where $\|\cdot\|$ denotes the Euclidean norm, and,
for every $s \in \{1,\ldots,S\}$,
 $\cb_s \in \eR^{P_s}$, $\Vb_s \in \eR^{P_s \times N}$ and $\psi_{s,\delta} \colon \eR \to \eR$. An important
example of such a framework is when, for every $s\in \{1,\ldots,S\}$,
$P_s = 1$ and $\cb_s = 0$, 
 and $\mathcal{V}=\big\{\Vb_s^\top, s \in \{1,\ldots,S\}\big\}\subset \eR^N$ constitutes a frame
of $\eR^N$, leading to a so-called frame-analysis regularization \cite{Elad07}.
For every $s \in \{1,\ldots,S\}$, 
$\Vb_s$ may also be a matrix serving to compute discrete
gradients (or higher-order differences), useful for edge preservation.
In particular, if $S = N$ and, for every $s\in \{1,\ldots,N\}$, 
$P_s = 2$, $\cb_s = \zerob$ and 
$\Vb_s = [\Deltab_s^{\mathrm{h}}\;\; \Deltab_s^{\mathrm{v}}]^\top$ where $\Deltab_s^{\mathrm{h}}\in
\eR^N$ (resp.  $\Deltab_s^{\mathrm{v}} \in \eR^N$) corresponds  to a horizontal (resp. vertical) gradient
operator, and $(\forall t\in \eR)$ 
$\psi_{s,\delta}(t)= \lambda |t|$ with $\lambda > 0$, the first term
in the right hand side of
\eqref{e:Psid} corresponds to a discrete version of the isotropic total
variation semi-norm \cite{Rudin92}. \newtext[
Note that other choices of $\Vb_s$ lead to different penalization strategies.
For instance, one can use nonlocal mean regularization, which has been recently studied in the context of edge preserving functions
in~\cite{Peter_2010_NID}.
]

In order to preserve significant coefficients in $\mathcal{V}$,
one may require the functions $(\psi_{s,\delta})_{1 \le s \le S}$ to have a slower-than-parabolic growth, 
as this limits the cost associated with these components. Two of
the main families of such functions known in the literature are:
\begin{enumerate}
	\item $\ell_2-\ell_1$ functions, i.e. convex, continuously differentiable,
asymptotically linear functions with a quadratic behavior
near $0$ \cite{Allain06,Charbonnier97,Lange90,Zibulevsky10}. Typical examples are the functions
$(\forall s \in \{1,\ldots,S\})$ $(\forall t
\in \eR)$ $\psi_{s,\delta}(t) = \lambda\sqrt{t^2 + \delta^2}$ with
$\lambda > 0$.
In the limit case when $\delta \to 0$, the classical $\ell_1$ penalty is obtained.
\item $\ell_2-\ell_0$ functions, i.e. asymptotically constant
  functions with a quadratic behavior near $0$ \cite{Fornasier12,Hebert92,Nikolova08,Veksler07,Zhang10}. Typical examples are
  the truncated quadratic functions $(\forall s \in \{1,\ldots,S\})$ 
$(\forall t \in \eR)$ $\psi_{s,\delta}(t)
  = \lambda 
\min(t^2/(2\delta^2),1)$ with $\lambda > 0$. When $\delta \to 0$, an
  $\ell_0$ penalty is obtained.
\end{enumerate}
The last quadratic penalty term $\xb \mapsto \|\Vb_0 \xb \|^2$ in \eqref{e:Psid} plays a
role similar to the elastic net regularization introduced in \cite{Zou05}.  It
allows us to guarantee some properties of the minimizers and minimization algorithms, when $\Hb$ is not
injective (e.g. an ideal low-pass filtering
operator). 

The $\ell_2-\ell_0$ approach has been shown in the literature to be advantageous
in many applications, for instance sparse component analysis~\cite{Mohimani09},
compressive sensing~\cite{Hyder09}, matrix completion~\cite{Malek11}, robust
regression~\cite{Meer91}, \newtext[segmentation~\cite{Rivera03}], and image recovery~\cite{Delaney98,Peter_2010_NID}. This paper mainly
addresses the latter problem, where $\ell_2-\ell_0$ is recognized for its
ability to preserve edges between homogeneous regions~\cite{Nikolova05b}. The
nonconvexity and sometimes non-differentiability of the potential function lead
however to a difficult optimization problem. In this paper, we consider a
class of nonconvex differentiable potential functions, which can be viewed as
smoothed versions of a truncated quadratic penalty function.

An effective approach for the minimization of differentiable criteria
is to consider a subspace descent algorithm
\cite{Elad06,Zibulevsky10}. For such methods, at each iteration, a
step size vector allowing an optimized combination of several search
directions is computed through a multidimensional search. Recently,
an original step size strategy based on a Majorize-Minimize (MM)
recursion was introduced in~\cite{Chouzenoux11}. This latter approach leads to a
closed-form algorithm whose practical efficiency has been demonstrated
in the context of image restoration, when using convex
penalized least squares criteria. 

Our main contributions in this paper are:
\begin{itemize}
\item to establish conditions under which 
a solution to an $\ell_0$
penalized criterion can be asymptotically obtained by using the considered
  class of penalty functions;
\item to extend the approach in \cite{Chouzenoux11} to
non necessarily convex minimization problems of the form
\eqref{Eq_CritJ};
\item to provide a proof of convergence of the iterates of
  the subspace MM algorithm;
\item to show the good \newtext[practical] performance of the proposed method for several applications.
\end{itemize}
It must be stressed that the convergence proofs
in this paper rely on recent results underlining the prominent
role played by the Kurdyka-\L{}ojasiewicz inequality \cite{Attouch08,Attouch10b,Attouch10a,Bolte10b} in the convergence study of various iterative optimization methods. Our results
constitute a significant improvement over those in \cite{Chouzenoux11}. In this
previous article, the analysis
was restricted to showing that the gradient of the objective function converges to zero.

The rest of the paper is organized as follows: 
properties of the considered optimization problem are first investigated
in section~\ref{Sec_Intro}. Then, we introduce in section~\ref{Sec_MM}
a minimization strategy based on an MM subspace
scheme. In section~\ref{Sec_CV}, we investigate the general convergence properties for the proposed
algorithm. Finally, section~\ref{Sec_Res} illustrates the
performance of our algorithm through a set of comparisons and experiments in image processing. 

\section{Considered class of objective functions}

\label{Sec_Intro}


In this section, we briefly mention some useful properties of problem~\eqref{Eq_CritJ}.

\vspace*{0.2cm}

\subsection{Existence of a minimizer}
First, we provide a preliminary result concerning the existence of a solution to the
problem under the following assumption on the functions in
\eqref{Eq_CritJ}
and on the nullspaces $\operatorname{Ker}
\Hb$ and $\operatorname{Ker}{\Vb_0}$ of
$\Hb$ and $\Vb_0$, respectively.
\begin{assumption}
\label{a:existmin}
\begin{enumerate}
\item \label{a:existmini} $\Phi$ is continuous and coercive (that is $\lim_{\|\zb\|\to +\infty} \Phi(\zb) = +\infty$).
\item\label{a:existminii} For every $\delta > 0$ and $s \in \{1,\ldots,S\}$,
$\psi_{s,\delta}$ is continuous and takes nonnegative values.
\item\label{a:existminiii} $\operatorname{Ker}
\Hb \cap \operatorname{Ker}{\Vb_0} = \{\mathbf{0}\}$.
\end{enumerate}
\end{assumption}

\begin{proposition}
\label{p:existmin}
Suppose that Assumption \ref{a:existmin} holds.
Then, for every $\delta > 0$, 
\begin{enumerate}
\item \label{p:existi} $F_\delta$ is coercive;
\item \label{p:existii} the set of minimizers of $F_\delta$ is
  nonempty and compact.
\end{enumerate}
\end{proposition}
\begin{proof}
Let $\delta > 0$.
Since, for every $s\in\left\{1,\cdots,S\right\}$, $\psi_{s,\delta}\ge 0$, we have 
\begin{equation}\label{e:Fdboundinf}
(\forall \xb\in \eR^N)\qquad
F_\delta(\xb) \ge 
\Phi(\Hb \xb - \yb) + \|\Vb_0 \xb\|^2 = \underline{F}(\xb).
\end{equation}
This implies that, for every $\eta \in \eR$,
\begin{equation}\label{e:inclulevel}
\operatorname{lev}_{\le \eta} F_\delta 
= \{\xb \in \eR^N \mid
F_\delta(\xb) \le \eta\} 
\subset \operatorname{lev}_{\le
  \eta}\underline{F}.
\end{equation}
As $\Phi$ is continuous and coercive, $\inf \Phi > -\infty$. 
For every
$\xb \in \eR^N$ and $\eta \in \eR$, if $\xb \in
\operatorname{lev}_{\le \eta}\underline{F}$, then
\begin{align}&\Phi(\Hb \xb - \yb) \le \eta \label{e:levPhi}\\
&\|\Vb_0 \xb\|^2 \le \eta -\inf \Phi. \label{e:levPi}
\end{align}
Then, as a consequence of \eqref{e:levPhi} and the coercivity of $\Phi$,
there exists $\zeta > 0$ such that, for every $\xb \in 
\operatorname{lev}_{\le \eta} \underline{F}$, 
\begin{equation}\label{e:boundxhp}
\|\Hb\xb\| \le \zeta.
\end{equation}
The combination of \eqref{e:levPi} and \eqref{e:boundxhp} shows that
there exists $\zeta'> 0$ such that, for every $\xb \in 
\operatorname{lev}_{\le \eta} \underline{F}$, 
$\| {\boldsymbol A} \xb \| \le \zeta'$ where 
\begin{equation}
{\boldsymbol A} = \begin{bmatrix}
\Hb\\
\Vb_0
\end{bmatrix}.
\end{equation}
It can be deduced that, for every $\xb \in 
\operatorname{lev}_{\le \eta} \underline{F} \cap
(\operatorname{Ker}{\boldsymbol A})^\perp$,
\begin{equation}
\underline{\nu}  \|\xb\| \le \zeta'
\end{equation}
where $\underline{\nu}$ 
is the minimum non-zero singular value of ${\boldsymbol A}$
(the existence of which is guaranteed since ${\boldsymbol A} \neq \zerob$).
In addition, $\operatorname{Ker}{\boldsymbol A} = \operatorname{Ker}
\Hb \cap \operatorname{Ker}\Vb_0 = \{\mathbf{0}\}$, which
implies that $(\operatorname{Ker}{\boldsymbol A})^\perp = \eR^N$.
Hence, $\underline{F}$ is a level-bounded function,
that is, for every $\eta \in \eR$, $\operatorname{lev}_{\le \eta} \underline{F}$
is bounded (and possibly empty).
Using \eqref{e:inclulevel}, we can conclude that
$F_\delta$ is a level-bounded function  (or equivalently,
it is coercive \cite[Proposition 11.11]{Rockafellar97}).
As $F_\delta$ is also continuous, 
\ref{p:existii} follows from \cite[Theorem~1.9]{Rockafellar97}.
\end{proof}

\begin{remark}
\begin{enumerate}
\item In the particular case when $\Hb$ is injective, 
Assumption \ref{a:existmin}\ref{a:existminiii} is satisfied if 
$\Vb_0 = \zerob$. The injectivity of $\Hb$
obviously holds when $\Hb = \Ib$
in \eqref{e:modellin}, which typically corresponds to denoising applications.
\item When $\Vb_0 = \zerob$, the existence of a minimizer of $F_\delta$ with $\delta > 0$ can also be guaranteed
under other useful conditions. For example, this property holds under  
Assumptions~\ref{a:existmin}\ref{a:existmini} and \ref{a:existmin}\ref{a:existminii},
if $\operatorname{Ker}\Hb \cap \bigcap_{s=1}^S \operatorname{Ker} \Vb_s = \{0\}$,
and when for every $s\in \{1,\ldots,S\}$, $\psi_{s,\delta}$ is coercive.
\end{enumerate}
\end{remark}

\subsection{Non-convex regularization functions}

In the remainder of this work, we will be interested in potentials
satisfying the following additional property:
\begin{assumption}
\label{as:psideb}
\begin{enumerate}
\item \label{as:psidebii} $(\forall s \in \{1,\ldots,S\})$ 
$(\forall (\delta_1,\delta_2) \in  (0,+\infty)^2)$
$\delta_1 \le \delta_2$ $\Rightarrow$ $(\forall t\in \eR)$
$\psi_{s,\delta_1}(t) \ge \psi_{s,\delta_2}(t)$.
\item \label{as:psidebiii} There exists $\lambda > 0$ such that
\begin{equation}
(\forall s \in \{1,\ldots,S\})
(\forall t \in \eR)\qquad \lim_{\substack{\delta \to 0\\\delta > 0}} \psi_{s,\delta}(t) 
= \lambda \chi_{\eR\backslash \left\{0\right\}}(t)
\end{equation}
where
$
 \chi_{\eR\backslash \left\{0\right\}}(t) =
\begin{cases}
0 & \mbox{if $t = 0$}\\
1 & \mbox{otherwise.}
\end{cases}
$
\end{enumerate}
\end{assumption}

Assumption \ref{as:psideb}\ref{as:psidebiii} implies that a binary penalty function is asymptotically
obtained. Examples of functions $\psi_{s,\delta}$ with $s\in
\{1,\ldots,S\}$ and $\delta > 0$  satisfying Assumptions \ref{a:existmin}\ref{a:existminii} and \ref{as:psideb} are
provided below:

\begin{example}\label{ex:psi}
\begin{enumerate}
\item \label{ex:quadtrunc} Truncated quadratic potential~\cite{Veksler_O_1999_phd_efficient_gbemmcv}: $$(\forall t
\in \eR) \qquad \psi_{s,\delta}(t) = \lambda \min \left(
\frac{t^2}{2
  \delta^2},1 \right), \quad \lambda > 0.$$
\item \label{ex:psi1} Geman-McClure potential~\cite{Geman_1985_bayesian_image_analysis}: $$(\forall t
\in \eR)\qquad \psi_{s,\delta}(t) = \frac{\lambda t^2}{2 \delta^2 +
  t^2}, \quad \lambda > 0.$$
\item \label{ex:psi2} Welsch potential~\cite{Welsh_1978_robust_regression}: $$(\forall t
\in \eR) \qquad \psi_{s,\delta}(t) = \lambda \Big(1-\exp(-\frac{t^2}{2
  \delta^2})\Big), \quad \lambda > 0.$$
\item \label{ex:psi3} Hyberbolic tangent potential: 
$$(\forall t
\in \eR) \qquad \psi_{s,\delta}(t) = \lambda
\tanh\Big(\frac{t^2}{2\delta^2}\Big), \quad \lambda > 0.
$$
\item \label{ex:psi4} Tukey biweight potential~\cite{Black_1998_robust_anisotropic_diffusion}:
$$
(\forall t \in \eR) \qquad \psi_{s,\delta}(t) = \left\{ \begin{array}{ll} \lambda \left(1-(1- \frac{t^2}{6 \delta^2})^3\right) & \text{if} \quad |t| \le \sqrt{6} \delta\\
\lambda & \text{otherwise}
\end{array} \right., \quad \lambda > 0.
$$
\end{enumerate}
\end{example}

The latter four functions are such that $\psi_{s,\delta}(t) \sim \lambda
t^2/(2\delta^2)$ as $t\to 0$. They can thus be viewed as smoothed versions of the one-variable truncated 
quadratic function in Example~\ref{ex:psi}\ref{ex:quadtrunc} (see Figure~\ref{fig:compar}).

\begin{figure}[h]
\centering
\includegraphics[width=8cm]{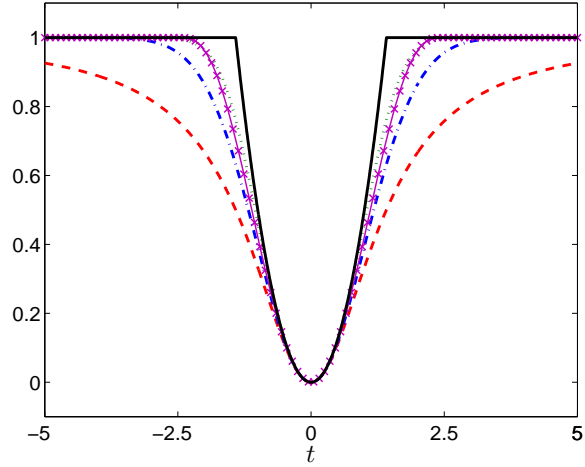}
\caption{\small Truncated quadratic penalty in Example \ref{ex:psi}\ref{ex:quadtrunc} (black, full) and its smooth
  approximations $\psi_{s,\delta}(t)$ as defined in Examples
  \ref{ex:psi}\ref{ex:psi1} (red, dashed), \ref{ex:psi}\ref{ex:psi2}
  (blue, dash-dot), \ref{ex:psi}\ref{ex:psi3}
  (green, dot), and \ref{ex:psi}\ref{ex:psi4}
  (magenta, cross), for parameters $\lambda = 1$ and $\delta = 1$.}
   \label{fig:compar}
\end{figure}

\vspace*{0.2cm}
\subsection{Asymptotic convergence to $\ell_0$ criterion}

The asymptotic behavior of the considered class of potentials can now be derived by showing the epi-convergence of
$F_\delta$ to the following block (or group) $\ell_0$-penalized objective function:
\begin{equation}
F_0\colon \xb \mapsto \Phi(\Hb \xb - \yb) + \lambda \ell_{0}(\Vb \xb-\cb) +
\|\Vb_0\xb\|^2,
\end{equation}
where $\Vb = \left[\Vb_1^\top\,|\,\ldots\,|\,\Vb_S^\top\right]^\top$,
$\cb = \left[\cb_1^\top,\ldots,\cb_S^\top\right]^\top$,
 and $\ell_{0}$ denotes the so-called `block $\ell_0$ cost' \cite{Eldar10}
defined as
\begin{equation}\label{e:defl0}
(\forall \tb = [\tb_1^\top,\ldots,\tb_S^\top]^\top
\in \eR^{P_1+\cdots+P_S})\qquad \ell_0(\tb) = \sum_{s=1}^S \chi_{\eR\backslash \left\{0\right\}}(\|\tb_s\|),
\end{equation}
where, for every $s\in \{1,\ldots,S\}$, $\tb_s \in \eR^{P_s}$.
When $P_1 = \ldots = P_S=1$, \eqref{e:defl0} provides the standard expression of
the $\ell_0$ cost of $\eR^S$.

\begin{proposition} Suppose that Assumptions \ref{a:existmin} and
\ref{as:psideb} hold.
Let $(\delta_n)_{n\in \mathbb{N}}$ be a decreasing sequence of positive real numbers
converging to $0$. Then,
\begin{enumerate}
\item $\inf F_{\delta_n} \to \inf F_0$ as $n \to +\infty$.
\item If $(\forall n\in \mathbb{N})$ $\widehat{\xb}_n$ is a minimizer
of $F_{\delta_n}$, then the sequence $(\widehat{\xb}_n)_{n\in \mathbb{N}}$ is
bounded and all its cluster points are minimizers of $F_0$.
\item If $F_0$ has a unique minimizer $\widetilde{\xb}$, then 
$\widehat{\xb}_n \to \widetilde{\xb}$ as $n\to +\infty$.
\end{enumerate}
\end{proposition}

\begin{proof}
First, note that, according to Assumption~\ref{as:psideb}\ref{as:psidebii}, for every $n\in  \mathbb{N}$,
$F_{\delta_{n+1}} \ge F_{\delta_n}$. In addition, for every $n\in
\mathbb{N}$, $F_{\delta_n}$ is a continuous function as a consequence
of Assumptions \ref{a:existmin}\ref{a:existmini} and 
\ref{a:existmin}\ref{a:existminii}.
Then it can be deduced from
\cite[Theorem 7.4(d)]{Rockafellar97}
that $(F_{\delta_n})_{n\in \mathbb{N}}$ epi-converges
to $\sup_{n\in \mathbb{N}} F_{\delta_n}$. The latter function is equal to
$F_0$ by virtue of Assumption~\ref{as:psideb}\ref{as:psidebiii}.
In addition,  $(F_{\delta_n})_{n\in \mathbb{N}}$ is eventually
level-bounded\footnote{$(F_{\delta_n})_{n\in \mathbb{N}}$is eventually level-bounded if, for every $\eta\in \eR$, there exists some subset $\mathcal{N}$
of $\eN$ such that $\eN \setminus \mathcal{N}$ is finite and $\cup_{n\in
  \mathcal{N}}\operatorname{lev}_{\le \eta} F_{\delta_n}$ is bounded.}
 as a consequence of \cite[Ex. 7.32(a)]{Rockafellar97},
the lower bound in \eqref{e:Fdboundinf} and the fact that 
$\underline{F}\colon \xb \mapsto \Phi(\Hb \xb - \yb) + \|\Vb_0 \xb\|^2$
is level-bounded (as shown in the proof of Proposition~\ref{p:existmin}).
We complete the proof by noticing that $F_0$ is lower semicontinuous 
and proper, and by applying \cite[Theorem 7.33]{Rockafellar97}.
\end{proof}

The above proposition guarantees that a minimizer of $F_0$ can be
well-appro\-xi\-ma\-ted by choosing a small enough $\delta$.
Note that the existence/uniqueness of a minimizer of $F_0$ is
discussed in the literature on compressed sensing under some specific
assumptions \cite{Candes08,Davenport12,Donoho05,Nikolova13}.

We will now turn our attention to numerical methods allowing us to
efficiently solve Problem \eqref{Eq_CritJ} when all the involved functions are smooth.

\section{Proposed optimization method}

\label{Sec_MM}

 \subsection{Subspace algorithm}

A classical strategy to minimize the criterion $F_{\delta}$ 
consists of building a sequence
$(\xb_k)_{k\in \mathbb{N}}$ of $\eR^N$ such that
\begin{equation}
(\forall k \in  \mathbb{N})\qquad F_\delta(\xb_{k+1}) \le F_\delta(\xb_k).
\end{equation}
This can be performed by translating the current solution $\xb_k$ 
at each iteration $k\in \mathbb{N}$ along a suitable direction $\db_k\in \eR^N$:
\begin{equation}
\label{Eq_AlgoDescente}
\xb_{k+1} = \xb_k + \alpha_k \db_k, 
\end{equation}
where $\alpha_k >0$ is the \emph{step size}, and $\db_k$ is a
\emph{descent direction}.
When $F_\delta$ is differentiable, this direction is chosen such that
 $\gb_k^\top \db_k \le 0$ where $\gb_k$ denotes the gradient of $F_{\delta}$ at $\xb_k$. 

A significant practical improvement regarding the convergence rate is
achieved by performing subspace acceleration, i.e. by 
considering a set of $M$ search directions 
$\{\db_k^1,\ldots,\db_k^{M}\}\subset \eR^N$ and 
by defining the new iteration as
\begin{equation}
\xb_{k+1} = \xb_k + \Db_k \sb_k, \label{Eq_IterSub} 
\end{equation} 
where $\Db_k= [\db_k^1,\ldots,\db_k^{M}]\in\eR^{N\times M}$ is the
search direction matrix and $\sb_k\in \eR^M$ is a multivariate
step size, which is computed so as to 
minimize 
\begin{equation}
f_{k,\delta}\colon \sb\mapsto F_{\delta}(\xb_k + \Db_k \sb). 
\end{equation} 
The memory gradient subspace algorithm, initially proposed in the late
1960's by Miele and Cantrell~\cite{Miele69}, corresponds to: 
\begin{equation}
(\forall k \ge 1)\qquad \Db_k=[-\gb_k \;\left|\right.\; \xb_k-\xb_{k-1}]. 
\end{equation}
When the objective function is quadratic, this algorithm is equivalent to the
linear conjugate gradient algorithm~\cite{Cantrell69}. More
recently, several other subspace algorithms have been proposed, where,
at each iteration $k\in \eN$, $\Db_k$ usually includes a descent
direction (e.g. gradient, Newton, truncated Newton directions) and a short history of previous directions (see \cite[Tab.1]{Chouzenoux11} for a general review). 

In addition, the subspace scheme~\eqref{Eq_IterSub} was shown to
outperform standard descent algorithms such as nonlinear conjugate gradient over a set of PLS minimization
problems in~\cite{Chouzenoux11,Zibulevsky10}. The convergence of 
Algorithm~\eqref{Eq_IterSub} however requires the design of a
proper strategy to determine the step sizes $(\sb_k)_{k\in\eN}$, which we discuss in the next section. 

\subsection{Majorize-Minimize step size}

At each iteration $k\in \eN$,
the minimization of $f_{k,\delta}$ using the Majorization-Minimization
(MM) principle is approximately performed by successive
minimizations of tangent majorant functions for $f_{k,\delta}$. 
Let $q_k \colon \eR^M \times \eR^M \to \eR$ and let $\sb'\in \eR^M$.
The function
$q_k(.,\sb')$ is said to be a tangent majorant for $f_{k,\delta}$ at
$\sb'$ if
\begin{equation}
\begin{cases}
(\forall \sb \in \eR^M) \quad
q_k(\sb,\sb') \ge f_{k,\delta}(\sb)\\
q_k(\sb',\sb') =  f_{k,\delta}(\sb').
\end{cases}
\label{Eq_PropMaj}
\end{equation}

From this point forward, we assume that $f_{k,\delta}$ is differentiable.
Following~\cite{Chouzenoux11}, we propose to employ a convex quadratic tangent majorant function of the form:
\begin{equation}
\label{Eq_Majorante}
(\forall \sb\in \eR^M)\quad
q_k(\sb,\sb') = f_{k,\delta}(\sb') + \nabla f_{k,\delta}(\sb')^\top (\sb - \sb') + \fracp{1}{2} (\sb - \sb')^\top \Bb_{k,\sb'}  (\sb - \sb'),
\end{equation}
where $\nabla f_{k,\delta}(\sb')$ denotes the derivative of
$f_{k,\delta}$ at $\sb'$, and $\Bb_{k,\sb'}$ is an $M \times M$
symmetric positive semi-definite matrix that ensures the fulfillment
of majorization properties \eqref{Eq_PropMaj}. The initial
minimization of $f_{k,\delta}$ is replaced by a sequence of easier
subproblems, corresponding to the following MM update rule:
\begin{equation}
\left\{
\begin{aligned}
\sb_k^0 & =  \zerob,\\
\forall j&\in \{1,\,\ldots,\,J\}\\
&\left\lfloor
\begin{array}{l}
\sb_k^{j} \in \underset{\sb\in\eR^M}{\operatorname{Argmin}}\;\; q_k(\sb,\sb_k^{j-1})\\
\end{array}
\right.
\end{aligned}
\right.
\label{Eq_MMRec}
\end{equation} 
\newtext[Note that for $M=1$,  this reduces to the scalar MM line search~\cite{Labat08}.]

\subsection{Construction of the majorizing approximation}

We now make the following assumption:
\begin{assumption}\label{a:conv}
\begin{enumerate} 
\item\label{a:convi} $\Phi$ is differentiable with an $L$-Lipschitzian gradient,
 i.e.
\begin{equation}
( \forall \zb\in \eR^Q)(\forall \zb' \in \eR^Q)\,\,
\|\nabla \Phi(\zb)-\nabla\Phi(\zb')\| \le L \| \zb - \zb'\|.
\end{equation}
\item \label{a:convii} For every $s \in \left\{1,\cdots,S\right\}$,
  $\psi_{s,\delta}$ is a differentiable function.
\item \label{a:conviii} For every $s\in \left\{1,\cdots,S\right\}$, $\psi_{s,\delta}(\sqrt{.})$ is concave on $[0,+\infty)$.
\item \label{a:conviv} For every $s \in \left\{1,\cdots,S\right\}$, there exists $\overline{\omega_s} \in [0,+\infty)$ such that
$(\forall t \in (0,+\infty))$ $0 \le  \dot{\psi}_{s,\delta}(t) \le
\overline{\omega_s} t$ where $\dot{\psi}_{s,\delta}$ is the
derivative of $\psi_{s,\delta}$. In addition,
$\lim_{\underset{t\neq 0}{t \to 0}}  \dot{\psi}_{s,\delta}(t)/t \in \eR$.
\end{enumerate}
\label{as:Maj}
\end{assumption}
We emphasize the fact that Assumptions
\ref{a:conv}\ref{a:convii}-\ref{a:conviv} hold for the
$\ell_2$-$\ell_0$ penalties in Examples~\ref{ex:psi}\ref{ex:psi1}-\ref{ex:psi4}. Morever, Tab.~\ref{Tab:DataFidel} presents several examples of functions fulfilling Assumption~\ref{a:conv}\ref{a:convi}. 

\begin{table}[t]
\centering
\renewcommand{\arraystretch}{1.2}
\begin{tabular}{|c|c|c|} 
\hline
Function name & $\Phi(\zb)$ & Lipschitz\\ 
& $\zb = (z_q)_{1\le q\le Q} \in\eR^Q$ & constant $L$ \\
\hline
\hline
Least squares &  $\frac12 \zb^\top \Lambda \zb$ & $\|\Lambda\|$ \\
  & $\Lambda\in \eR^{Q\times Q}$ symmetric positive semi-definite  &
   \\
\hline
$\ell_2$- $\ell_1$ & $\sum_{q=1}^Q \phi_q(z_q)$ & $\max_{1\le q\le Q}(\frac{1}{\sqrt{ \rho_q}})$ \\
\cite{Vogel96} & $(\forall t \in \eR)$ $\phi_q(t) = \sqrt{\rho_q  + t^2}$, $\rho_q > 0$  & \\
\hline
Huber &  $\sum_{q=1}^Q \phi_q(z_q)$ & $2 \max_{1\le q\le Q}\rho_q$ \\
\cite{Huber81} & $(\forall t \in \eR)$ $\phi_q(t) = \begin{cases}
\rho_q t^2 & \mbox{if $|t|\le \nu_q$}\\
\rho_q \nu_q(2 |t|-\nu_q|) & \mbox{if  $|t| > \nu_q$}
\end{cases}$ & \\
& $\nu_q > 0$, $\rho_q > 0$ &  \\
\hline
Cauchy & $\sum_{q=1}^Q \phi_q(z_q)$ & $\max_{1\le q\le Q}(\frac{2}{\rho_q})$  \\
\cite{Anto02} & $(\forall t \in \eR)$ $\phi_q(t) = \ln(\rho_q+t^2)$, $\rho_q > 0$ &  \\
\hline
Squared distance to & $\frac12 d_B^2(\zb)$ &  1\\
 a closed convex set $B$ \cite{Bauschke11} & &  \\
\hline
Smoothed max \cite{BenTal89b} & $ \rho \ln(\sum_{q=1}^Q e^{z_q / \rho})$, $\rho > 0$  & $1/\rho$  \\
\hline
Inf-convolution & $\inf_{\zb_1+\zb_2 = \zb}
\Phi_1(\zb_1)+\Phi_2(\zb_2)$ & $\rho$ \\
\cite{Bauschke11} & $\Phi_1\in \Gamma_0(\eR^Q)$,
$\Phi_2\in \Gamma_0(\eR^Q)$ & \\
&  $\Phi_2$ $\rho$-Lipschitz differentiable, $\rho > 0$, & \\
& such that $\lim_{\| \zb \| \to +\infty} \frac{\Phi_2(\zb)}{\|\zb\|} = +\infty$
& \\
\hline
\end{tabular}
\caption{Some examples of functions $\Phi$ with an $L$-Lipschitzian gradient. ($\|\Lambda\|$
  denotes the spectral norm of $\Lambda$ and $\Gamma_0(\eR^Q)$ denotes the class of proper
  lower-semicontinuous convex functions from $\eR^Q$ to $(-\infty,+\infty]$.)
}
\label{Tab:DataFidel}
\end{table}

By defining
\begin{equation}\label{e:defomega}
(\forall s \in \{1,\ldots,S\}) (\forall t \in \eR) \quad \omega_{s,\delta}(t) =
\dot{\psi}_{s,\delta}(t)/t,
\end{equation}
(the function $\omega_{s,\delta}$ is extended by continuity
at $0$), a tangent majorant can be built as described below:
\begin{lemma}{\rm \cite{Allain06}}
For every $\xb\in \eR^N$, let 
\begin{equation}
\Ab(\xb) = \mu \Hb^\top \Hb + 2 \Vb_0^\top \Vb_0+ \Vb^\top
\Diag\left\{\bb(\xb)\right\}\Vb,
\label{Eq_A}
\end{equation}
where $\mu \in [L,+\infty)$ and
$\bb(\xb)= \big(b_i(\xb)\big)_{1\le i \le SP}\in \eR^{SP}$
with $P = \sum_{s = 1}^S P_s$ is such that
\begin{equation}
(\forall s \in \{1,\ldots,S\})\, (\forall p \in \{1,\ldots,P_s\})\quad
b_{P_1+\cdots+ P_{s-1}+p}(\xb)=\omega_{s,\delta}(\| \Vb_s \xb-\cb_s\|).
\end{equation}
Let $\sb'\in \eR^M$ and $k\in \eN$.
Then,  under Assumption \ref{a:conv}, $q_k(\cdot,\sb')$ with
\begin{equation}\label{e:Bku}
\Bb_{k,\sb'} = \Db_k^\top \Ab(\xb_k + \Db_k \sb') \Db_k,
\end{equation}
is a convex quadratic tangent majorant of $f_{\delta,k}$ at $\sb'$.
\end{lemma}

Hence, according to~\eqref{Eq_Majorante} and \eqref{Eq_MMRec}, the optimality condition for
the choice of the step size in the MM iteration is given by:
\begin{equation}
(\forall k\in \eN)(\forall j\in \{1,\ldots,J\})\quad
 \Bb_{k,\sb_k^{j-1}} (\sb_k^{j} -\sb_k^{j-1}) + \nabla f_{k,\delta}(\sb_k^{j-1}) = \mathbf{0}.
\label{Eq_PAS}
\end{equation}
This yields the explicit step size formula
\begin{equation}
\sb_k^j = \sb_k^{j-1} - \, \Bb_{k,\sb_k^{j-1}}^{-1}\, \nabla
f_{k,\delta}(\sb_k^{j-1}),
\end{equation}
where $\Bb_{k,\sb_k^{j-1}}^{-1}$ is the pseudo-inverse of
$\Bb_{k,\sb_k^{j-1}}\in \eR^{M\times M}$. One of the main advantages of
this approach is that the computational cost of the
required inversion is low, provided that the number $M$ of search directions remains
small.
The resulting MM subspace algorithm reads
\begin{equation}
\label{e:blackpage} 
\left\{
\begin{array}{l}
\xb_0\in \eR^N,\\
\forall k \in \eN\\
\begin{array}{l}
\left\lfloor
\begin{array}{l}
\sb_k^0  =  \zerob,\\
\forall j\in \{1,\,\ldots,\,J\}\\
\left\lfloor
\begin{array}{l}
\Bb_{k,\sb_k^{j-1}} = \Db_k^\top \Ab(\xb_k + \Db_k \sb_k^{j-1}) \Db_k,\\
\sb_k^j = \sb_k^{j-1} - \, \Bb_{k,\sb_k^{j-1}}^{-1}\,\Db_k^\top 
\nabla F_{k,\delta}(\xb_k + \Db_k \sb_k^{j-1}),
\end{array}
\right.\\[1mm]
\xb_{k+1} = \xb_k + \Db_k \sb_k^J.
\end{array}
\right.\\
\end{array}
\end{array}
\right.
\end{equation}


\section{Convergence result}

\label{Sec_CV}

We first provide some preliminary technical lemmas before stating our main convergence result.
In the following, for every $k \in \mathbb{N}$ and $j\in \{0,\ldots,J\}$, we define
\begin{align}
\xb_k^j &= \xb_k +  \Db_k \sb_k^j,\label{Eq_SubIter}\\ 
\gb_k^j &= \nabla F_\delta(\xb_k^j),
\end{align}
(thus, $\xb_k^J = \xb_{k+1}$ and $\gb_k^J = \gb_{k+1}$). Moreover, we assume that the set of
directions $(\Db_k)_{k\in \eN}$ fulfills the following condition:

\begin{assumption}
For every $k \in \eN$, the matrix of directions $\Db_k$ is of size $N \times M$ with $1\le M \le N$ and the first subspace direction $\db_k^1$ is gradient-related i.e.,
	\begin{align}
\gb_k^\top \db_k^1 & \le - \gamma_0 \| \gb_k \|^2,\label{Eq_GR1}\\
\| \db_k^1 \| & \le \gamma_1 \| \gb_k \|,\label{Eq_GR2}
\end{align}
with $\gamma_0> 0$ and $\gamma_1 > 0$.
\label{as:GR}
\end{assumption}

As emphasized in~\cite[Sec.1.2]{Bertsekas99} and~\cite[Sec.III-D]{Chouzenoux11},
conditions \eqref{Eq_GR1} and \eqref{Eq_GR2} hold for a large family of descent
directions, such as the steepest descent direction or the truncated Newton direction.

\subsection{Preliminary results}\

\begin{lemma}\label{le:sufdec}
Under Assumptions~\ref{as:Maj} and~\ref{as:GR}, there exists a constant $\nu>0$ such that, for every $k\in \eN$ and $j\in \{1,\ldots,J\}$,
$F_{\delta}(\xb_k)-F_{\delta}(\xb_k^j) \ge \frac{\gamma_0^2}{\gamma_1^2}\nu^{-1}\|\gb_k \|^2$.
\end{lemma}
\begin{proof}
According to Assumption \ref{a:conv}\ref{a:conviv} and Eq. \eqref{e:defomega},
for every $s \in \left\{1,\cdots,S\right\}$, 
$\omega_{s,\delta}$ is upper-bounded on $(0,+\infty)$.
Hence, there exists $\nu>0$ such that, for every $\xb \in \eR^N$ and
$\vb \in \eR^N$, $\vb^\top A(\xb) \vb \le \nu \| \vb\|^2/2$. The
result then follows from \cite[Theorem 1]{Chouzenoux11}. 
\end{proof}

\begin{lemma}\label{le:sufdec2}
Under Assumptions~\ref{a:existmin} and~\ref{as:Maj}, the MM subspace iterates 
are such that
\begin{align}\label{eq:sufdec}
(\forall k \in \mathbb{N})(\forall j \in \{0,\ldots,J-1\})\qquad
 F_\delta(\xb_k^j)-F_\delta(\xb_k^{j+1}) & \ge \frac{\eta}{2} \| \xb_k^{j+1} - \xb_k^j\|^2
\end{align}
where $\eta > 0$ is the smallest eigenvalue of $\mu\Hb^\top\Hb + 2\Vb_0^\top\Vb_0$.
\end{lemma}
\begin{proof}
Let $k\in \eN$ and $j \in \{0,\ldots,J-1\}$.
According to \eqref{Eq_Majorante} and 
the definition of $\sb_k^{j+1}$,
\begin{equation}
f_{k,\delta}(\sb_k^j) - q_k(\sb_k^{j+1},\sb_k^j) = - \frac{1}{2} \nabla f_{k,\delta}(\sb_k^j)^\top (\sb_k^{j+1}-\sb_k^j).
\end{equation}
Furthermore, $q_k(\sb_k^{j+1},\sb^j) \ge f_{k,\delta}(\sb_k^{j+1})$. Thus,
\begin{equation}
f_{k,\delta}(\sb_k^j) - f_{k,\delta}(\sb_k^{j+1}) \ge - \frac{1}{2} \nabla f_{k,\delta}(\sb^j)^\top (\sb_k^{j+1}-\sb_k^j).
\end{equation}
The last inequality also reads
\begin{equation}
F_\delta(\xb_k^j) - F_\delta(\xb_k^{j+1}) \ge - \frac{1}{2}
\nabla f_{k,\delta}(\sb^j)^\top (\sb_k^{j+1}-\sb_k^j).
\end{equation}
So, using \eqref{e:Bku} and \eqref{Eq_PAS},
\begin{align}
F_\delta(\xb_k^j) - F_\delta(\xb_k^{j+1}) &\ge \frac{1}{2} \big(\Db_k (\sb_k^{j+1}-\sb_k^j)\big)^\top\Ab(\xb_k^j) \Db_k (\sb_k^{j+1}-\sb_k^j)\\
&\ge \frac{\eta}{2} \| \Db_k (\sb_k^{j+1}-\sb_k^j)\|^2.
\end{align}
In the latter inequality, we make use of the fact that, 
since $\sdm{Ker}\Hb \cap \sdm{Ker}\Vb_0= \left\{\zerob\right\}$, 
$\eta$ 
is positive, 
and 
\begin{equation}\label{e:mineigA}
(\forall \xb \in \eR^N)(\forall \vb \in \eR^N)\qquad
\vb^\top A(\xb) \vb \ge \eta \| \vb\|^2.
\end{equation}
\end{proof}

\begin{lemma}\label{le:reler}
Under Assumptions~\ref{a:existmin} and~\ref{as:Maj},
the MM subspace iterates are such that
\begin{align}\label{eq:reler}
(\forall k \in \mathbb{N})(\forall j \in \{0,\ldots,J-1\})\qquad
\eta \| \xb_{k}^{j+1} - \xb_k^j\| & \le \| \gb_k^j \|,
\end{align}
where $\eta > 0$ is the same constant as in Lemma \ref{le:sufdec2}.
\end{lemma}

\begin{proof}
According to \eqref{Eq_PAS}, we have, for every $k\in \mathbb{N}$ and $j\in \{0,\ldots,J-1\}$,
\begin{equation}
\Db_k^\top \gb_k^j + \Db_k^\top \Ab(\xb_k^j) \Db_k (\sb_k^{j+1}-\sb_k^j) = \zerob.
\end{equation}
Hence,
\begin{equation}
\big(\Db_k (\sb_k^{j+1}-\sb_k^j)\big)^\top \gb_k^j + 
\big(\Db_k (\sb_k^{j+1}-\sb_k^j)\big)^\top\Ab(\xb_k) \Db_k (\sb_k^{j+1}-\sb_k^j) = \zerob. \label{Eq_Grad}
\end{equation}
By using \eqref{e:mineigA}, 
\eqref{Eq_Grad} leads to
\begin{equation}
- \big(\Db_k (\sb_k^{j+1}-\sb_k^j)\big)^\top \gb_k^j \ge \eta \| \Db_k (\sb_k^{j+1}-\sb_k^j)\|^2.
\end{equation}
In addition, the Cauchy-Schwarz inequality leads to
\begin{equation}
- \big(\Db_k (\sb_k^{j+1}-\sb_k^j)\big)^\top \gb_k^j \le \| \gb_k^j\| \| \Db_k (\sb_k^{j+1}-\sb_k^j)\|.
\end{equation}
Thus, the latter two inequalities yield:
\begin{equation}
\eta \| \Db_k (\sb_k^{j+1}-\sb_k^j)\|^2\le \| \gb_k^j\| \| \Db_k (\sb_k^{j+1}-\sb_k^j)\|.
\end{equation}
Substituting with~\eqref{Eq_SubIter}, we obtain the desired result.
\end{proof}

\subsection{Convergence theorem}

Based on the two previous lemmas, classical results in the
optimization literature \cite{Ortega70} may allow us to deduce the convergence of the sequence $(\xb_k)_{k\in \mathbb{N}}$ generated by the MM subspace algorithm, 
but these results require restrictive conditions on the critical points of the objective function $F_\delta$. We propose here a more general approach based on 
recent results in nonconvex optimization
\cite{Attouch08,Attouch10b,Attouch10a}. We first recall the following definition
from \cite{Lojasiewicz63}:
\begin{definition}
A differentiable function $G\colon \eR^N \to \eR$ is said 
to satisfy the Kurdyka-\L{}o\-ja\-sie\-wicz inequality if,
for every $\widetilde{\xb} \in \eR^N$ and every bounded neighborhood $E$ of $\widetilde{\xb}$, 
there exist three
constants  $\kappa > 0$, $\zeta > 0$ and $\theta\in [0,1)$ such
that
\begin{equation}\label{e:ineqKL}
\|\nabla G(\xb)\| \ge \kappa |G(\xb)-G(\widetilde{\xb}) |^{\theta},
\end{equation}
for every $\xb \in E$ such that $|G(\xb)-G(\widetilde{\xb})| < \zeta$.
\end{definition}

The interesting point is that this inequality is satisfied for a wide class of
functions. In particular, it holds for real analytic functions, semi-algebraic functions
as well as many others~\cite{Bolte06,Bolte07,Kurdyka94,Lojasiewicz63}. Recall that a \newtext[function] $G\colon \eR^N \to \eR$ is semi-algebraic
if its graph $\{(\xb,\eta) \in \eR^N\times \eR\mid \eta = G(\xb)\}$ is a semi-algebraic
set, i.e. it can be expressed as a finite union of subsets of $\eR^N\times \eR$ defined by
a finite number of polynomial inequalities. The semi-algebraicity property is stable under
various operations (sum, product, inversion, composition,...). Examples of semi-algebraic functions
include $\xb \mapsto \|\Hb\xb-\yb\|^2$, $\Psi_\delta$ when the functions $(\psi_{s,\delta})_{1\le s \le S}$ are
given by Example~\ref{ex:psi}\ref{ex:psi1} or \ref{ex:psi}\ref{ex:psi4}, the squared distance to a closed convex semi-algebraic set.
In turn, examples of real-analytic functions include $\xb \mapsto \|\Hb\xb-\yb\|^2$ and
$\Psi_\delta$ when the functions $(\psi_{s,\delta})_{1\le s \le S}$ are
given by Examples~\ref{ex:psi}\ref{ex:psi1}-\ref{ex:psi}\ref{ex:psi3}.
Note that a more general local version of inequality \eqref{e:ineqKL} can also be found in the literature
\cite{Bolte07}.

Let us now state our main convergence result:
\begin{theorem}
Assume that $F_\delta$ satisfies the Kurdyka-\L ojasiewicz inequality.
Under Assumptions~\ref{a:existmin},~\ref{as:Maj} and~\ref{as:GR}, 
the MM subspace algorithm given by \eqref{e:blackpage} 
generates a sequence $(\xb_k)_{k\in \mathbb{N}}$ converging to a
  critical point $\widetilde{\xb}$ of $F_\delta$. Moreover,
  this sequence is of finite length, in the sense that
\begin{equation}
\sum_{k=0}^{+\infty} \|\xb_{k+1}-\xb_k\| < + \infty.
\end{equation}
\end{theorem}
\begin{proof}
As $(F_\delta(\xb_k))_{k\in \mathbb{N}}$ is a
decreasing sequence and $\operatorname{lev}_{\le F_\delta(\xb_0)} = \left\{\xb \in \eR^N | F_\delta(\xb) \le F_\delta(\xb_0)\right\}$ is a
bounded set (by virtue of Proposition~\ref{p:existmin}\ref{p:existi}), the sequence
$(\xb_k)_{k\in\mathbb{N}}$ belongs to a compact subset $E$ of $\eR^N$.
Hence, there exists a
subsequence $(\xb_{k_i})_{i\in \mathbb{N}}$ of $(\xb_k)_{k\in
  \mathbb{N}}$ converging to a vector $\widetilde{\xb}$ of $\eR^N$. 
Besides, since $F_\delta$ is a continuous function,
$(F_\delta(\xb_{k_i}))_{i\in \mathbb{N}}$ converges to
$F_\delta(\widetilde{\xb})$. 
As
$(F_\delta(\xb_k))_{k\in \mathbb{N}}$ is decreasing,
and Proposition~\ref{p:existmin}\ref{p:existi} shows that 
it is \newtext[bounded below], we deduce that 
$(F_\delta(\xb_k)-F_\delta(\widetilde{\xb}))_{k\in
  \mathbb{N}}$ is a nonnegative sequence converging to 0.

Now, by invoking Lemma \ref{le:sufdec} (with $j=J$), we have that, for every
$k\in \mathbb{N}$,
\begin{align}\label{e:dusoir}
\frac{\gamma_0^2 }{\gamma_1^2 } \nu^{-1}\|\gb_k \|^2 &\le 
F_{\delta}(\xb_k)-F_{\delta}(\xb_{k+1}) =  F_{\delta}(\xb_k)-
F_\delta(\widetilde{\xb})-\big(F_{\delta}(\xb_{k+1})-F_\delta(\widetilde{\xb})\big).
\end{align}
According to the \L ojasiewicz property, there exist
constants $\kappa > 0$, $\zeta > 0$ and $\theta\in [0,1)$ such
that
\begin{equation}
\|\nabla F_\delta(\xb)\| \ge \kappa |F_\delta(\xb)-F_\delta(\widetilde{\xb}) |^{\theta},
\end{equation}
for every $\xb \in E$ such that 
$|F_\delta(\xb)-F_\delta(\widetilde{\xb})| < \zeta$.
We now apply to the convex function
$\varphi\colon [0,+\infty) \to [0,+\infty)\colon u \mapsto u^{1/(1-\theta)}$ 
 the following gradient inequality
\begin{equation}
(\forall (u,v) \in [0,+\infty)^2)\qquad
\varphi(v) \ge \varphi(u)+ \dot{\varphi}(u) (v-u)
\end{equation}
which, after a change of variables, can be rewritten as
\begin{equation}
(\forall (u,v) \in [0,+\infty)^2)\qquad
u-v \le (1-\theta)^{-1} u^\theta(u^{1-\theta}-v^{1-\theta}).
\end{equation}
Combining the latter inequality with \eqref{e:dusoir} leads to
\begin{align}\label{e:BndDiffF}
F_{\delta}(\xb_k)-
F_\delta(\widetilde{\xb})-\big(F_{\delta}(\xb_{k+1})-F_\delta(\widetilde{\xb})\big)& \le (1-\theta)^{-1}(F_{\delta}(\xb_k)-
F_\delta(\widetilde{\xb}))^\theta \Delta_k
\end{align}
where 
\begin{equation}\label{e:DefDelta}
\Delta_k = \big(F_{\delta}(\xb_k)-
F_\delta(\widetilde{\xb})\big)^{1-\theta}-
\big(F_{\delta}(\xb_{k+1})-F_\delta(\widetilde{\xb})\big)^{1-\theta}.
\end{equation}
Thus,
\begin{align}
\| \gb_k \|^2 &  \le \frac{\gamma_1^2 }{\gamma_0^2 } \nu(1-\theta)^{-1}(F_{\delta}(\xb_k)-
F_\delta(\widetilde{\xb}))^\theta \Delta_k.
\end{align}
Since $(F_\delta(\xb_k))_{k\in \mathbb{N}}$ converges to
$F_\delta(\widetilde{\xb})$, there exists $k^* \in \eN$, such that, for every $k\ge k^*$,
$0 \le F_\delta(\xb_k) - F_\delta(\widetilde{\xb}) < \zeta$.
By applying the \L{}ojasiewicz inequality,
\begin{align}\label{e:boundgk2}
(\forall k \ge k^*)\qquad
\| \gb_k \|^2 
&\le \frac{\gamma_1^2 }{\gamma_0^2 } \nu\kappa^{-1}(1-\theta)^{-1} \|\gb_k\| \Delta_k.
\end{align}
This allows us to deduce that
\begin{equation}\label{e:sumgbk}
\sum_{k=k^*}^{+\infty}\| \gb_k \| \le 
\frac{\gamma_1^2 }{\gamma_0^2 } \nu\kappa^{-1}(1-\theta)^{-1}
\big(F_{\delta}(\xb_{k^*})-F_\delta(\widetilde{\xb})\big)^{1-\theta}.
\end{equation}

Furthermore, according to \eqref{Eq_SubIter}, 
\begin{equation}
\frac{\eta}{2} \|\xb_{k+1}-\xb_k\|^2
= \frac{\eta}{2} \Big\|\sum_{j=0}^{J-1} (\xb_k^{j+1} -
\xb_k^j) \Big\|^2
\end{equation}
which, by using Lemma \ref{le:sufdec2} and the convexity of the
squared norm, yields for every $k\in\mathbb{N}$,
\begin{align}
\frac{\eta}{2} \|\xb_{k+1}-\xb_k\|^2 &\le \frac{\eta J}{2} \sum_{j=0}^{J-1} \|\xb_k^{j+1}
-\xb_k^j\|^2\nonumber\\
&\le J \sum_{j=0}^{J-1} F_\delta(\xb_k^{j}) -
F_\delta(\xb_k^{j+1})
= J \big(F_\delta(\xb_k) - F_\delta(\xb_{k+1})\big).
\end{align}
By proceeding similarly to the derivation of \eqref{e:boundgk2}, we obtain: for every $k\ge k^*$,
\begin{equation}\label{e:eta2xb}
\frac{\eta}{2} \|\xb_{k+1}-\xb_k\|^2
\le J(1-\theta)^{-1}  \big(F_{\delta}(\xb_k)-
F_\delta(\widetilde{\xb})\big)^{\theta} \Delta_k
\le J \kappa^{-1}(1-\theta)^{-1} \|\gb_k\| \Delta_k.
\end{equation}
By using the fact that, for every $(u,v) \in
[0,+\infty)^2$,  $(uv)^{1/2} \le u + \frac{v}{4}$, and taking $u = J \eta^{-1}\kappa^{-1}(1-\theta)^{-1} \Delta_k$ and $v= 2 \| \gb_k\|$, \eqref{e:eta2xb} leads to
\begin{equation}
\|\xb_{k+1}-\xb_k\| \le J \eta^{-1}\kappa^{-1}(1-\theta)^{-1} \Delta_k
+ \frac{1}{2}\|\gb_k\|.
\end{equation}

By summing now over $k$ and using \eqref{e:DefDelta} and \eqref{e:sumgbk}, we finally obtain
\begin{equation}
\sum_{k=k^*}^{+\infty}\|\xb_{k+1}-\xb_k\| \le 
\kappa^{-1}(1-\theta)^{-1}(J\eta^{-1}+\frac{\gamma_1^2 }{\gamma_0^2 } \frac{\nu}{2})
\big(F_{\delta}(\xb_{k^*})-F_\delta(\widetilde{\xb})\big)^{1-\theta}.
\end{equation}
This gives us the desired finite length property. In addition, since this
condition implies that $(\xb_k)_{k\in\mathbb{N}}$ is a Cauchy
sequence, it converges towards a single point, which is necessarily
$\widetilde{\xb}$.
Finally, due to the continuity of $F_\delta$ and Lemma \ref{le:sufdec}, 
$(\gb_k)_{k \in \eN}$ converges to zero. As $(x_k,F_\delta(x_k)) \to (\widetilde{\xb},F_\delta(\widetilde{\xb}))$, 
the closedness property of the gradient implies that $\nabla
F_\delta(\widetilde{\xb}) = \zerob$, i.e. $\widetilde{\xb}$ must be a critical point of $F_\delta$.

\end{proof}   

Note that the inexact gradient methods that are studied in
\cite{Attouch10a} \newtext[are distinct from the subspace algorithms we
consider]. 

\section{Simulation results}\label{Sec_Res} 

The aim of this section is to illustrate and analyze the performance of the
proposed algorithm in the context of Problem~\eqref{Eq_CritJ}.  \newtext[We also
show the nonconvex penalization functions in Example \ref{ex:psi} to be appropriate for
image processing applications.]  To this end, four image processing
problems are considered, namely denoising, segmentation, deblurring and
tomographic reconstruction. For each of them, the produced image $\widehat{\xb}\in
\eR^N$ is defined as a minimizer of the function $F_\delta$, where $\Phi$,
$\Hb$, $\yb$ and $\Vb$ depend on the considered application. For the elastic net
regularization term, we choose $\Vb_0 = \tau \Ib$, $\tau \ge 0$. For deblurring
and tomographic applications, the linear operator $\Hb$ is not necessarily
injective. Thus, we set $\tau$ equal to a small positive value in order to
fulfill Assumption \ref{a:existmin}\ref{a:existminiii}. In the two other cases,
$\tau$ is set to zero.

%

For every $s\in \{1,\ldots,S\}$, we have set $\cb_s = \zerob$.
For the potential function $\psi_{s,\delta}$, we have tested the smooth convex $\ell_2-\ell_1$
function $\psi_{s,\delta}\colon t\mapsto \lambda (\sqrt{1 +
  \froc{t^2}{\delta^2}}-1)$ with $\lambda > 0$ (SC) and the smooth nonconvex
functions in Example \ref{ex:psi}\ref{ex:psi1} (SNC\ref{ex:psi1}), Example
\ref{ex:psi}\ref{ex:psi2} (SNC\ref{ex:psi2}), Example
\ref{ex:psi}\ref{ex:psi3} (SNC\ref{ex:psi3}) and Example
\ref{ex:psi}\ref{ex:psi4} (SNC\ref{ex:psi4}). Moreover, in the
denoising and segmentation examples, we provide optimization results for
four state-of-the-art combinatorial optimization algorithms, namely the
$\alpha$-expansion \cite{Boykov01} ($\alpha$-EXP), Quantized-Convex Move
Splitting~\cite{Jezierska11} (QCSM), Tree-Reweighted (TRW) \cite{Kolmogorov06}
and Belief Propagation (BP) \cite{Felzenszwalb10} algorithms, for which the
nonsmooth nonconvex truncated quadratic function in Example~\ref{ex:psi}\ref{ex:quadtrunc} (NSNC) is considered.  When
the linear degradation operator is not the identity matrix, we do not provide
any comparison with the combinatorial algorithms. Indeed, although a few algorithms
\cite{Raj05, Raj07} are applicable to inverse problems involving a linear
degradation operator, these methods are well-founded only for a sparse
convolution operator $\Hb$. Moreover, they rely on an adaptation of the graph
cut $\alpha$-expansion algorithm, which is shown in our segmentation and
denoising examples to be outperformed by our approach.

The computation of the proposed MM subspace algorithm requires specifying the
direction set $\Db_k$, for every $k\in \eN$, and the number of MM sub-iterations
$J$. \newtext[First, the memory-gradient direction matrices,]
 \begin{equation}
(\forall k \ge m)\qquad \Db_k=[-\gb_k \;\left|\right.\; \xb_k-\xb_{k-1} \;\left|\right.\; \cdots \;\left|\right.\; \xb_{k-m+1} - \xb_{k-m}]\in \eR^{N \times (m+1)},
 \end{equation}
 \newtext[with memory parameter $m\ge 0$, is considered. Moreover, in all our
 experiments, we set $J=1$. This choice was observed to yield the best results
 in terms of convergence profile in the context of MM-based step size computation
 \cite{Chouzenoux11,Labat08}. In the following, we compare our proposed subspace
 algorithm, denoted hereafter by 3MG-$m$ (for Majorize-Minimize Memory
 Gradient) with three other iterative first order descent methods. The methods
 we compare against are namely the nonlinear conjugate gradient (NLCG) algorithm
 \cite{Hager06}, the L-BFGS algorithm \cite{Liu89} with the memory parameter set
 to $3$, and the fast version of half quadratic (HQ) algorithm \cite{Allain06}.]
 For each descent algorithm, the MM scalar line search with $J=1$ is employed
 for the computation of the step size.  In the case of HQ, the inner optimization
 problems are solved partially with conjugate gradient iterations.
 \newtext[Note that this algorithm has been previously studied in the context of
 nonconvex regularization functions in~\cite{Delaney98,Rivera03}.]  In order to
 \newtext[limit the influence of possible local minima in the nonconvex case], the result of $10$
 iterations of convex minimization using an $\ell_2-\ell_1$ penalty is employed
 as an initialization.
 In the convex case, minimization is started with the constant null image. The
 computational complexity is evaluated in terms of iteration number and
 computational time in seconds necessary to achieve the global stopping rule
 $\|\gb_k\|/\sqrt{N} <
 10^{-4}$.  
 \textsf{C++} codes were compiled with the Intel compiler icpc
 (version 12.1.0) and were run on an Intel(R) Xeon(R) CPU X5570 at 2.93GHz, in a
 single thread.

\subsection{Image denoising}

The first problem considered in this section corresponds to the recovery of an
image $\overline{\xb}$ from noisy observations $\ub = \overline{\xb}+\wb$ where
$\wb$ is a realization of a zero-mean white Gaussian noise.  The vector
$\overline{\xb}$ here corresponds to the \texttt{Word} image of size $N = 128 \times
128$ pixels. The variance of the noise was adjusted to correspond to a
signal-to-noise ratio (SNR) of $15$ dB (Figure \ref{Fig:Word}). The recovery of
the original image is performed by solving \eqref{Eq_CritJ} where $Q= 2N$,
\begin{equation}
\Hb = \left[\begin{array}{c} \Ib \\ \Ib\end{array} \right] \qquad \yb = \left[\begin{array}{c} \ub \\ \zerob\end{array} \right],
\end{equation}
and 
\begin{equation}
(\forall \zb =(z_q)_{1\le q \le 2N})\qquad 
\Phi(\zb) = \frac{1}{2} \left(\sum_{q=1}^N z_q^2 + \beta \sum_{q=N+1}^{2N} d_B^2(z_q) \right),
\end{equation}
where $d_B$ denotes the distance to the closed convex interval $B = [0,255]$ and
$\beta >0$ is a weighting factor. Then, $\Phi$ is Lipschitz differentiable with
Lipschitz constant $L = \max(1,\beta)$. In the sequel, we choose $\beta = 1$
so that we have $L = 1$. Moreover, the penalization term \eqref{e:Psid} is
used, with $\tau = 0$ and an anisotropic penalization on neighboring pixels
i.e., $S = 2N$, and for every $s\in \{1,\ldots,N\}$ (resp. $s\in
\{N+1,\ldots,2N\}$), $P_s = 1$ and $\Vb_s$ corresponds to a horizontal
(resp. vertical) gradient operator. This anisotropic term is chosen so as to
compare more fairly our approach with the combinatorial methods. 

Parameters $\lambda$ and $\delta$ were chosen to maximize the SNR between the
original image and its reconstructed version. In Figure~\ref{Fig:WordRec}, the
reconstructed images are displayed and the corresponding SNR and MSSIM
\cite{Wang04a} values are provided. Morever, the absolute values of the
reconstruction errors $\widehat{\xb}-\overline{\xb}$ are illustrated. It should
be noticed that the nonconvex regularization strategy with penalty function
SNC\ref{ex:psi1} leads to the best results in terms of reconstruction
quality.

\begin{figure}[!hftbp]
\centering
\begin{tabular}{cc}
\includegraphics[width=5cm]{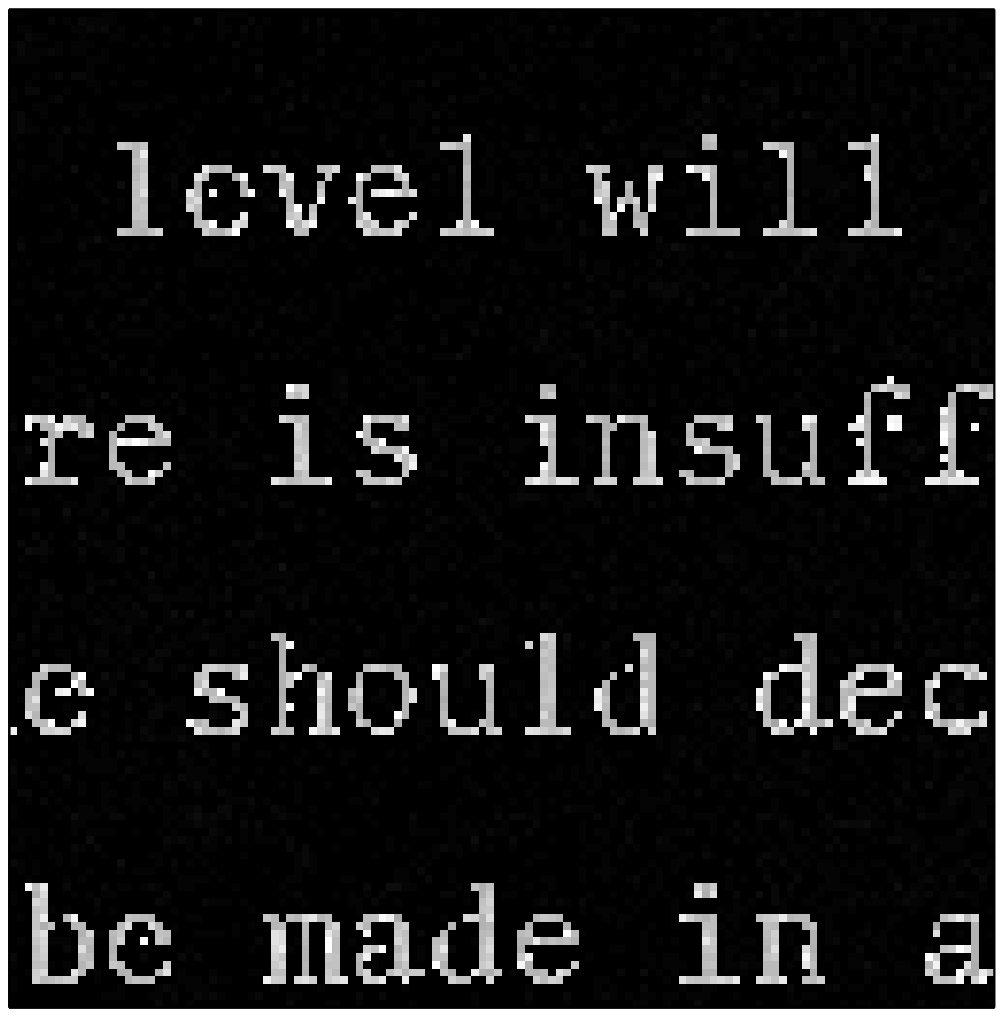}&
\includegraphics[width=5cm]{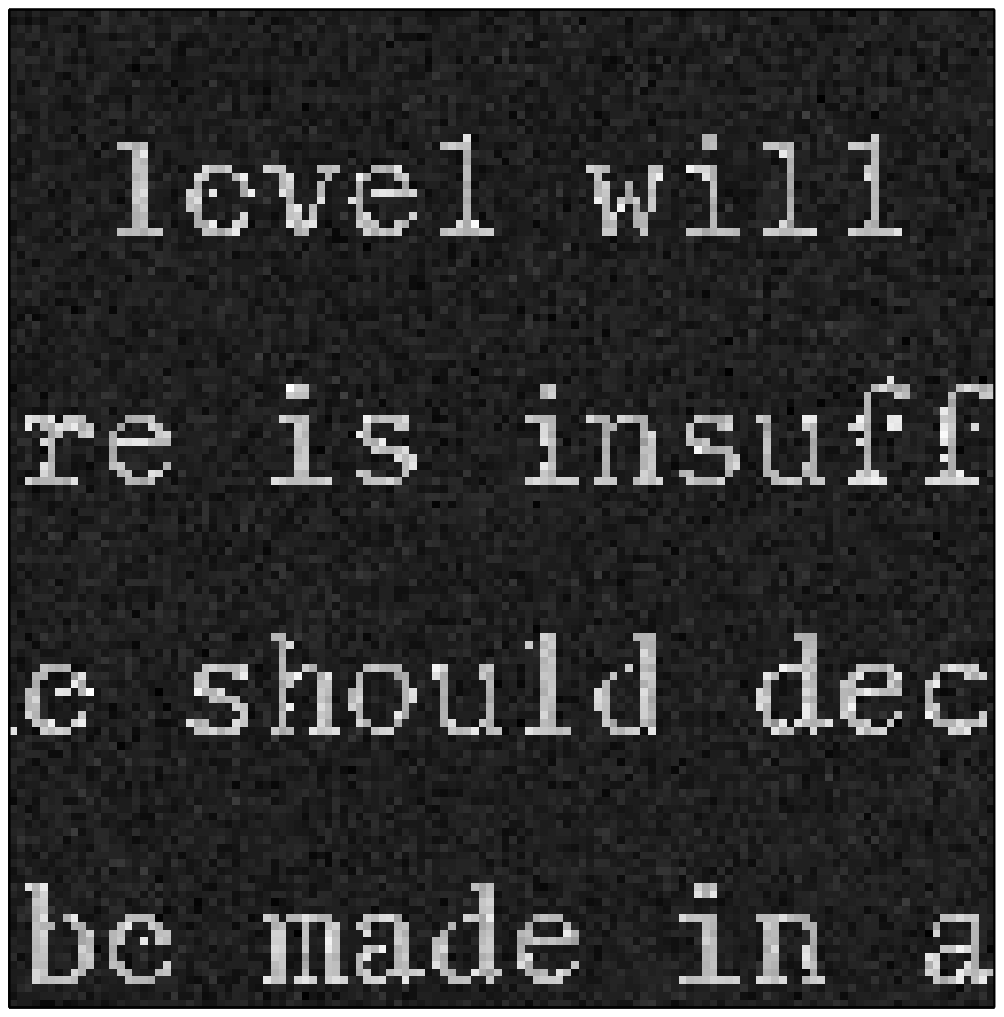}\\
$(a)$ & $(b)$ 
\end{tabular}
\caption{$(a)$ Original image with $128 \times 128$ pixels and $(b)$ noisy image
with SNR $= 15$ dB, MSSIM $= 0.66$, and noise standard deviation equal to $10$.}
\label{Fig:Word}
\vspace{0.7cm}
\centering
\begin{tabular}{ccc}
\includegraphics[width=5cm]{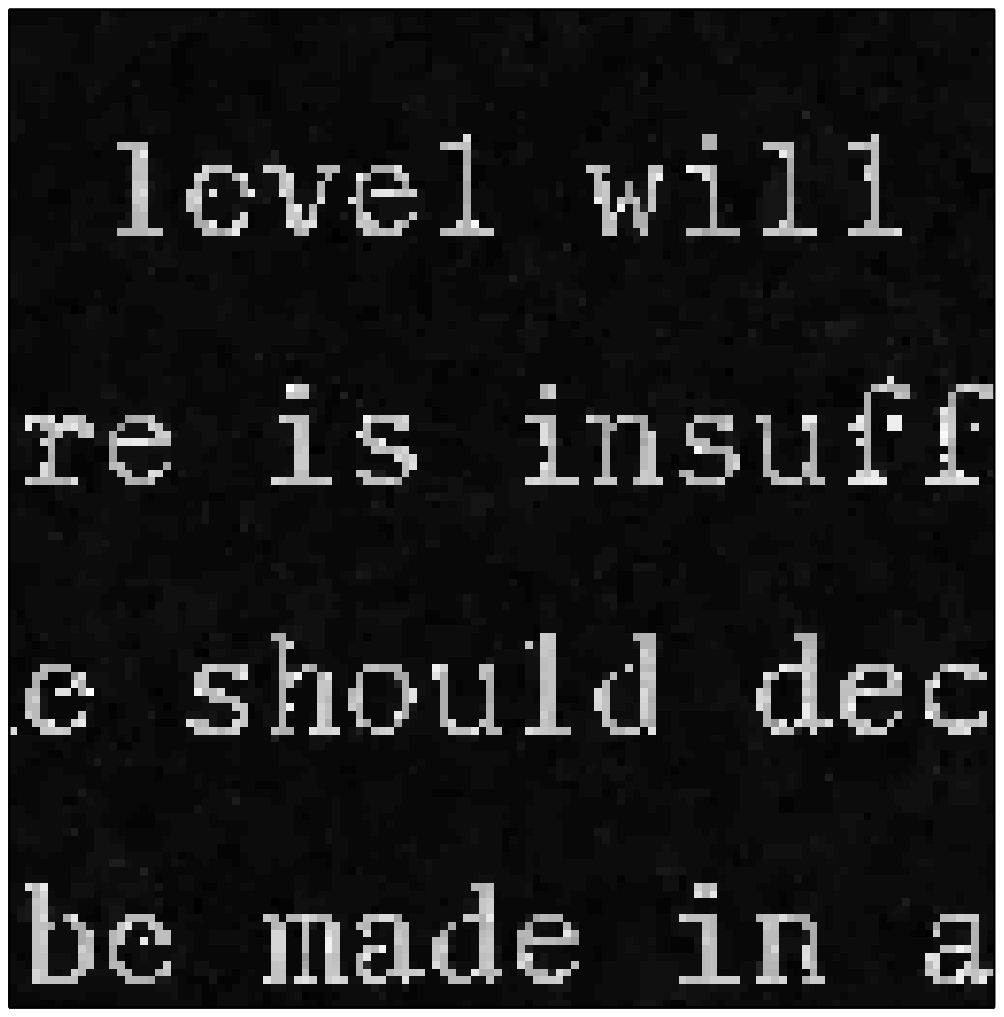}&
\includegraphics[width=5cm]{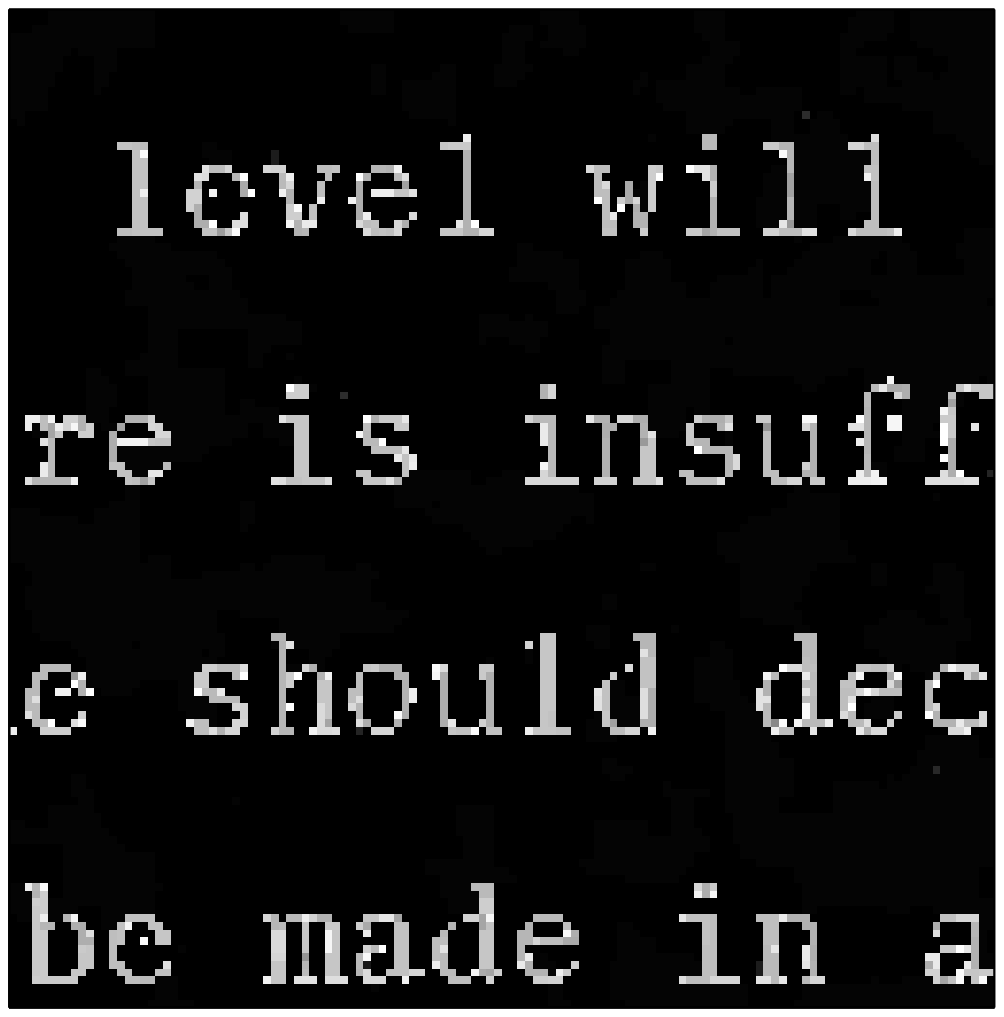} &
\includegraphics[width=5cm]{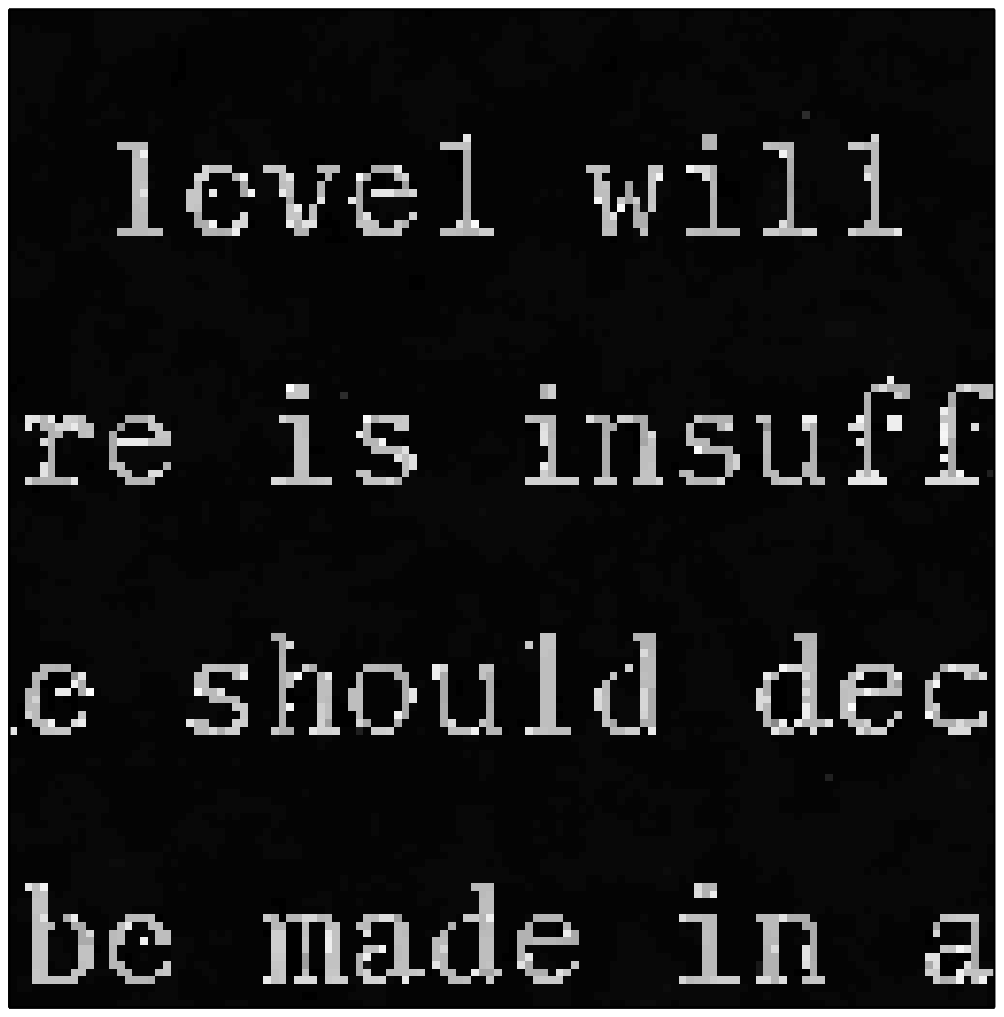}\\
\includegraphics[height=5cm]{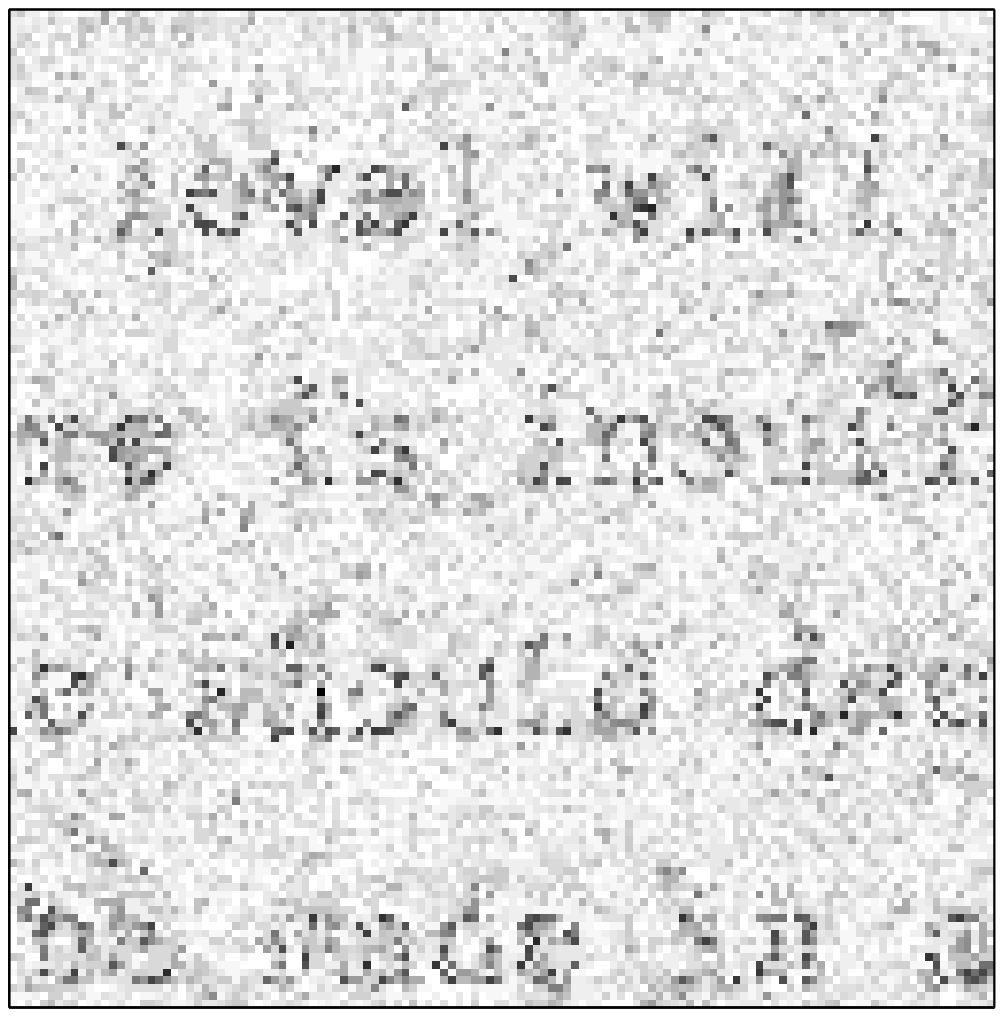} &
\includegraphics[height=5cm]{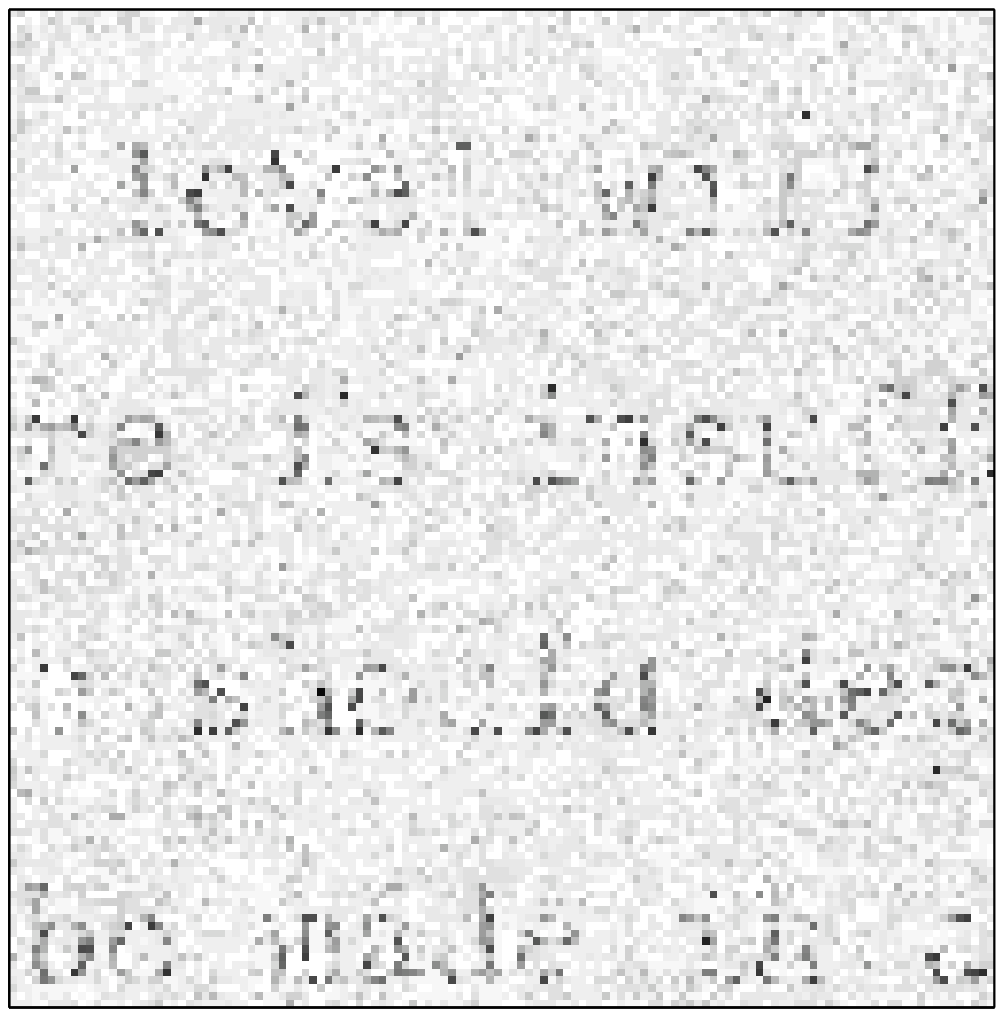} &
\includegraphics[height=5cm]{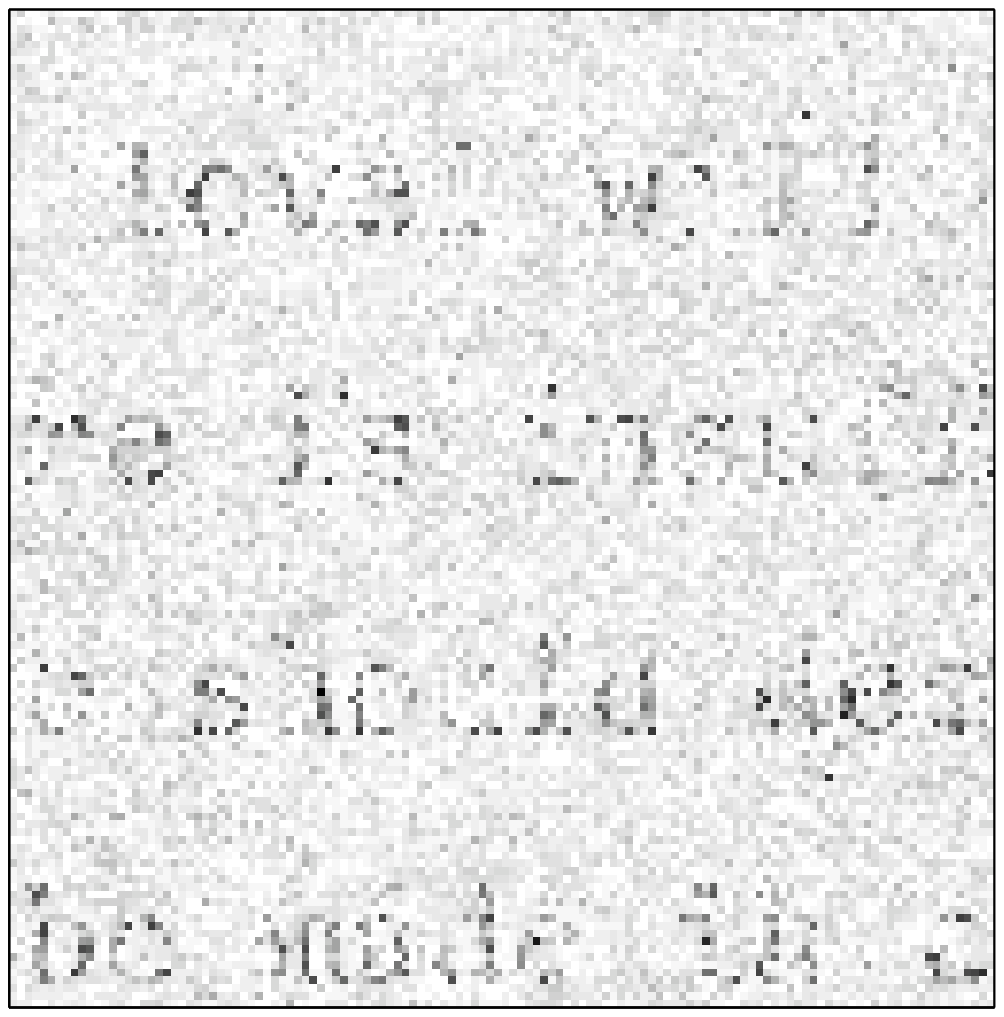}\\
$(a)$ & $(b)$ & $(c)$  
\end{tabular}
\caption{$(a)$ Denoising results and absolute reconstruction error with SC penalty using 3MG, $\lambda = 0.3$, $\delta = 0.07$, SNR = $20.41$ dB, MSSIM = $0.89$, $(b)$ with NSNC penalty using TRW, $\lambda = 350$, $\delta = 3.5$, SNR = $22.8$ dB, MSSIM = $0.93$, and $(c)$ with SNC\ref{ex:psi1} penalty using 3MG, $\lambda = 280$, $\delta = 7.25$, SNR = $22.74$ dB, MSSIM = $0.92$.}
\label{Fig:WordRec}
\end{figure}

\subsubsection{Influence of memory size}
\newtext[
We first analyze the effect of the memory size $m$ on the performance of our algorithm. 
We recall that the detailed performance analysis of 3MG algorithm with respect to the size of the memory was provided in~\cite{Chouzenoux11}, but it was restricted to the convex case. 
The results in Tab.~\ref{Tab:DenoisItWordMM} illustrate that the choice where memory equals one, which corresponds to a subspace with size $2$, leads to the best results in terms of computational time. Hence, our experiments confirm the conclusions drawn in~\cite{Chouzenoux11} for the convex case. Consequently, $m=1$, i.e.] $\Db_k=[-\gb_k\;\left|\right.\; \xb_k-\xb_{k-1}]$ for all $k\ge 1$ \newtext[was retained for the remaining experiments presented in the paper, and the shorter notation 3MG is employed for denoting the 3MG-$1$ algorithm.
]

 \begin{table}[!hftbp]
\centering
\renewcommand{\arraystretch}{1.2}
\begin{tabular}{|c|c|c|c|c|c| }
\hline
Penalty function $(\lambda,\delta)$ & Algorithm & Iteration & Time & $F_\delta$ & SNR (dB)\\
\hline
\hline
SNC\ref{ex:psi1} $(280,7.25)$ 
& 3MG-$0$ & $998$ & $1.08$  & $1.54 \cdot 10^6$ & $22.74$ \\   
& 3MG-$1$ & $270$ & \underline{$0.35$}  & $1.54 \cdot 10^6$ & $22.74$ \\   
& 3MG-$2$& $247$ & $0.38$   & $1.54 \cdot 10^6$ & $22.74$ \\   
& 3MG-$3$& $248$ & $0.44$   & $1.54 \cdot 10^6$ & $22.74$ \\   
& 3MG-$4$& $243$ & $0.51$   & $1.54 \cdot 10^6$ & $22.74$ \\   
& 3MG-$5$& $239$ & $0.59$   & $1.54 \cdot 10^6$ & $22.74$ \\  
\hline
\hline
SNC\ref{ex:psi2} $(301,8.76)$ 
& 3MG-$0$ & $536$ & $0.66$  & $1.59 \cdot 10^6$ & $22.55$ \\   
& 3MG-$1$ & $101$ & \underline{$0.21$}  & $1.59 \cdot 10^6$ & $22.55$ \\   
& 3MG-$2$& $159$ & $0.28$   & $1.59 \cdot 10^6$ & $22.55$ \\   
& 3MG-$3$& $158$ & $0.32$   & $1.59 \cdot 10^6$ & $22.55$ \\   
& 3MG-$4$& $156$ & $0.36$   & $1.59 \cdot 10^6$ & $22.55$ \\   
& 3MG-$5$& $155$ & $0.41$   & $1.59 \cdot 10^6$ & $22.55$ \\   
\hline
\hline
SNC\ref{ex:psi3} $(381,10)$ 
& 3MG-$0$ & $287$ & $0.61$  & $1.8 \cdot 10^6$ & $22.47$ \\   
& 3MG-$1$ & $69$ & \underline{$0.16$}  & $1.8 \cdot 10^6$ & $22.47$ \\   
& 3MG-$2$& $70$ & $0.19$   & $1.8 \cdot 10^6$ & $22.47$ \\   
& 3MG-$3$& $67$ & $0.21$   & $1.8 \cdot 10^6$ & $22.47$ \\   
& 3MG-$4$& $66$ & $0.22$   & $1.8 \cdot 10^6$ & $22.47$ \\   
& 3MG-$5$& $67$ & $0.28$   & $1.8 \cdot 10^6$ & $22.47$ \\   
\hline
\hline
SNC\ref{ex:psi4} $(386,9)$ 
& 3MG-$0$ & $202$ & $0.42$  & $1.8 \cdot 10^6$ & $22.48$ \\   
& 3MG-$1$ & $49$ & \underline{$0.11$}  & $1.8 \cdot 10^6$ & $22.48$ \\   
& 3MG-$2$& $51$ & $0.13$   & $1.8 \cdot 10^6$ & $22.48$ \\   
& 3MG-$3$& $51$ & $0.16$   & $1.8 \cdot 10^6$ & $22.48$ \\   
& 3MG-$4$& $52$ & $0.17$   & $1.8 \cdot 10^6$ & $22.48$ \\   
& 3MG-$5$& $52$ & $0.21$   & $1.8 \cdot 10^6$ & $22.48$ \\   
\hline
\end{tabular}
\caption{Denoising problem with \texttt{word} image. Influence of memory parameter $m$ in 3MG algorithm.}
\label{Tab:DenoisItWordMM} 
\end{table}

\subsubsection{Comparison with NLCG algorithm}

\newtext[
The NLCG algorithm is based on the following
iterations:
\begin{equation}
(\forall k \ge 1)\qquad \xb_{k+1} = \xb_k + \alpha_k (-\gb_k + \beta_k (\xb_k - \xb_{k-1})),
\end{equation}
where $\alpha_k > 0$ is the step size and $\beta_k\in \eR$ is the conjugacy
parameter. Tab.~\ref{Tab:DenoisItWordGCNL} summarizes the performances of NLCG
for five different conjugacy strategies described in \cite{Hager06}. Contrary to the
convex case, in the nonconvex case the conjugacy formula has a major influence
on the convergence speed (see Tab.~\ref{Tab:DenoisItWordGCNL} results related to
NLCG in rows 1-6 and 7-30). In particular the conjugacy strategies FR and DY 
do not appear well-adapted to the nonconvex problems. On the other hand, the HS, LS and PRP+
conjugacy parameters yield good numerical performance. Thus, they have been selected
for the numerical experiments in the following. For comparison, we include in
Tab.~\ref{Tab:DenoisItWordGCNL} the results of 3MG for $m=1$.  Although 
the superiority of 3MG versus NLCG is not established
theoretically, these experimental results are very promising.
They show that 3MG algorithm is faster than the considered  
non-linear conjugate gradient algorithms.
]

 \begin{table}[!hftbp]
\centering
\renewcommand{\arraystretch}{1.2}
\begin{tabular}{|c|c|c|c|c|c| }
\hline
Penalty function $(\lambda,\delta)$ & Algorithm & Iteration & Time & $F_\delta$ & SNR (dB)\\
\hline
\hline
SC $(0.3,0.07)$ 
& NLCG-HS & $138$ &  \underline{$0.84$}  & $2.7 \cdot 10^6$ & $20.41$ \\  
& NLCG-FR & $305$ & $1.86$  & $2.7 \cdot 10^6$ & $20.41$ \\  
& NLCG-PRP+ & $143$ & $0.87$ & $2.7 \cdot 10^6$ & $20.41$ \\  
& NLCG-LS & $158$ & $0.96$  & $2.7 \cdot 10^6$ & $20.41$ \\  
& NLCG-DY & $223$ & $1.35$  & $2.7 \cdot 10^6$ & $20.41$ \\ 
\cline{2-6}
& 3MG & $122$ & \underline{$0.22$} & $2.7 \cdot 10^6$ & $20.41$ \\  
\hline
\hline
SNC\ref{ex:psi1} $(280,7.25)$ 
& NLCG-HS & $1250$ & $2.34$ & $1.54 \cdot 10^6 $ & $22.74$\\
& NLCG-FR & $>10000$ & $-$ & $-$ & $-$\\
& NLCG-PRP+ & $292$ & \underline{$0.55$} & $1.54 \cdot 10^6 $ & $22.74$\\
& NLCG-LS & $320$ & $0.79$ &  $1.54 \cdot 10^6 $ & $22.74$\\
& NLCG-DY & $>10000$ & $-$ &  $-$ & $-$\\
\cline{2-6}
& 3MG & $270$ & \underline{$0.35$}  & $1.54 \cdot 10^6$ & $22.74$ \\   
\hline
\hline
SNC\ref{ex:psi2} $(301,8.76)$ 
& NLCG-HS & $112$ & \underline{$0.26$} & $1.59 \cdot 10^6$ & $22.55$\\
& NLCG-FR & $>10000$ & $-$ & $-$ & $-$\\
& NLCG-PRP+ & $179$ & $0.42$ & $1.59 \cdot 10^6$ & $22.55$\\
& NLCG-LS & $210$ & $0.54$ &  $1.59 \cdot 10^6$ & $22.55$\\
& NLCG-DY & $>10000$ & $-$ &  $-$ & $-$\\
\cline{2-6}
& 3MG & $101$ & \underline{$0.21$}  & $1.59 \cdot 10^6$ & $22.55$ \\   
\hline
\hline
SNC\ref{ex:psi3} $(381,10)$ 
& NLCG-HS & $102$ &  $1.1$  & $1.8 \cdot 10^6$ & $22.47$ \\  
& NLCG-FR & $3289$ & $36.3$  & $1.8 \cdot 10^6$ & $22.47$ \\  
& NLCG-PRP+ & $79$ & \underline{$0.9$} & $1.8 \cdot 10^6$ & $22.47$ \\  
& NLCG-LS & $90$ & $1$  & $1.8 \cdot 10^6$ & $22.47$ \\  
& NLCG-DY & $3342$ & $36.8$  & $1.8 \cdot 10^6$ & $22.47$ \\  
\cline{2-6}
& 3MG & $69$ & \underline{$0.16$}  & $1.8 \cdot 10^6$ & $22.47$ \\   
\hline
\hline
SNC\ref{ex:psi4} $(386,9)$ 
& NLCG-HS & $52$ &  \underline{$0.15$}  & $1.8 \cdot 10^6$ & $22.48$ \\  
& NLCG-FR & $>10000$ & $-$  & $-$ & $-$ \\  
& NLCG-PRP+ & $55$ & $0.16$ & $1.8 \cdot 10^6$ & $22.48$ \\  
& NLCG-LS & $56$ & $0.16$  & $1.8 \cdot 10^6$ & $22.48$ \\  
& NLCG-DY & $>10000$ & $-$  & $-$ & $-$ \\  
\cline{2-6}
& 3MG & $49$ & \underline{$0.11$}  & $1.8 \cdot 10^6$ & $22.48$ \\   
\hline
\end{tabular}
\caption{Denoising problem with \texttt{word} image. Influence of conjugacy parameter $\beta_k$ in NLCG algorithm.}
\label{Tab:DenoisItWordGCNL} 
\end{table}

 \subsubsection{Summary}
 \newtext[ We summarize the results by comparing the performance of continuous
 and discrete algorithms with SC, SNC and NSNC potential functions (see
 Tab.~\ref{Tab:DenoisItWord}).  One can observe that the considered discrete
 optimization algorithms lead to a SNR which is very similar to that obtained
 with smooth nonconvex regularization. However, they are more demanding in
 terms of computational time than 3MG.  Thus, we can conclude that the 3MG
 algorithm behaves well in comparison with the considered continuous and discrete
 algorithms.
]
 \begin{table}[!hftbp]
\centering
\renewcommand{\arraystretch}{1.2}
\begin{tabular}{|c|c|c|c|c|c| }
\hline
Penalty function $(\lambda,\delta)$ & Algorithm & Iteration & Time & $F_\delta$ & SNR (dB)\\
\hline
\hline
SC $(0.3,0.07)$ 
& 3MG & $122$ & \underline{$0.22$} & $2.7 \cdot 10^6$ & $20.41$ \\  
& \newtext[NLCG-HS] & $138$ &  $0.35$  & $2.7 \cdot 10^6$ & $20.41$ \\   
& \newtext[NLCG-PRP+] & \newtext[$143$] &  \newtext[$0.37$]  & \newtext[$2.7 \cdot 10^6$] & \newtext[$20.41$] \\   
& \newtext[NLCG-LS] & \newtext[$158$] & \newtext[$0.96$]  & \newtext[$2.7 \cdot 10^6$] & \newtext[$20.41$] \\  
& L-BFGS & $209$ & $0.73$  & $2.7 \cdot 10^6$ & $20.41$ \\  
& HQ & $670$ & $3.03$ & $2.7\cdot 10^6$ & $20.41$ \\  
\hline
\hline
SNC\ref{ex:psi1} $(280,7.25)$ 
& 3MG & $270$ & \underline{$0.35$}  & $1.54 \cdot 10^6$ & $22.74$ \\   
& \newtext[NLCG-HS] & $1250$ & $2.34$ & $1.54  \cdot 10^6$ & $22.74$ \\   
& \newtext[NLCG-PRP+] & \newtext[$292$] & \newtext[$0.55$] & \newtext[$1.54  \cdot 10^6$] & \newtext[$22.74$] \\  
& \newtext[NLCG-LS] & \newtext[$320$] & \newtext[$0.79$] & \newtext[$1.54 \cdot 10^6 $] & \newtext[$22.74$]\\ 
& L-BFGS & $332$ & $0.96$ & $1.54  \cdot 10^6$ & $22.73$\\  
& HQ & $1025$ & $3.84$ & $1.54 \cdot 10^6$ & $22.74$ \\ 
\hline
\hline
SNC\ref{ex:psi2} $(301,8.76)$ 
& 3MG & $101$ & \underline{$0.21$}  & $1.59 \cdot 10^6$ & $22.55$ \\   
& \newtext[NLCG-HS] & $112$ & $0.26$ & $1.59  \cdot 10^6$ & $22.55$ \\   
& \newtext[NLCG-PRP+] & \newtext[$179$] & \newtext[$0.42$] & \newtext[$1.59  \cdot 10^6$] & \newtext[$22.55$] \\ 
& \newtext[NLCG-LS] & \newtext[$210$] & \newtext[$0.54$] &  \newtext[$1.59 \cdot 10^6$] & \newtext[$22.55$]\\  
& L-BFGS & $351$ & $1.08$ & $1.59  \cdot 10^6$ & $22.55$\\  
& HQ & $604$ & $2.53$ & $1.59 \cdot 10^6$ & $22.54$ \\ 
\hline
\hline
SNC\ref{ex:psi3} $(381,10)$ 
& 3MG & $69$ & \underline{$0.16$}  & $1.8 \cdot 10^6$ & $22.47$ \\   
& \newtext[NLCG-HS] & $102$ & $0.27$ & $1.8 \cdot 10^6$ & $22.47$ \\   
& \newtext[NLCG-PRP+] & \newtext[$79$] & \newtext[$0.21$] & \newtext[$1.8 \cdot 10^6$] & \newtext[$22.47$] \\
& \newtext[NLCG-LS] & \newtext[$90$] & \newtext[$1$]  & \newtext[$1.8 \cdot 10^6$] & \newtext[$22.47$] \\     
& L-BFGS & $94$ & $0.32$ & $1.8 \cdot 10^6$ & $22.46$\\  
& HQ & $287$ & $1.36$ & $1.8 \cdot 10^6$ & $22.47$ \\ 
\hline
\hline
SNC\ref{ex:psi4} $(386,9)$ 
& 3MG & $49$ & \underline{$0.11$}  & $1.8 \cdot 10^6$ & $22.48$ \\   
& \newtext[NLCG-HS] & $52$ & $0.15$ & $1.8 \cdot 10^6$ & $22.48$ \\   
& \newtext[NLCG-PRP+] & \newtext[$55$] & \newtext[$0.16$] & \newtext[$1.8 \cdot 10^6$] & \newtext[$22.48$] \\
& \newtext[NLCG-LS] & \newtext[$56$] & \newtext[$0.16$]  & \newtext[$1.8 \cdot 10^6$] & \newtext[$22.48$] \\     
& L-BFGS & $80$ & $0.25$ & $1.8 \cdot 10^6$ & $22.48$\\  
& HQ & $202$ & $1.1$ & $1.8 \cdot 10^6$ & $22.48$ \\ 
\hline
\hline
NSNC $(350,3.5)$  
& $\alpha$-EXP &   $4$   & $4.67 $   &  $1.31 \cdot 10^6$ & $22.69$\\ 
& QCSM & $2 $ & \underline{$1.25 $} & $1.31 \cdot 10^6 $& $ 22.60$\\ 
& TRW  & $5 $ & $ 1.65$ & $ 1.31 \cdot 10^6$ & $22.80$\\ 
& BP   & $18 $ & $5.33 $ & $ 1.31 \cdot 10^6$& $22.73$\\ 
\hline
\end{tabular}
\caption{Results for the denoising problem.}
\label{Tab:DenoisItWord}
\end{table}

\subsection{Image segmentation}

\begin{figure}[!htbp]
\centering
\includegraphics[width=5cm]{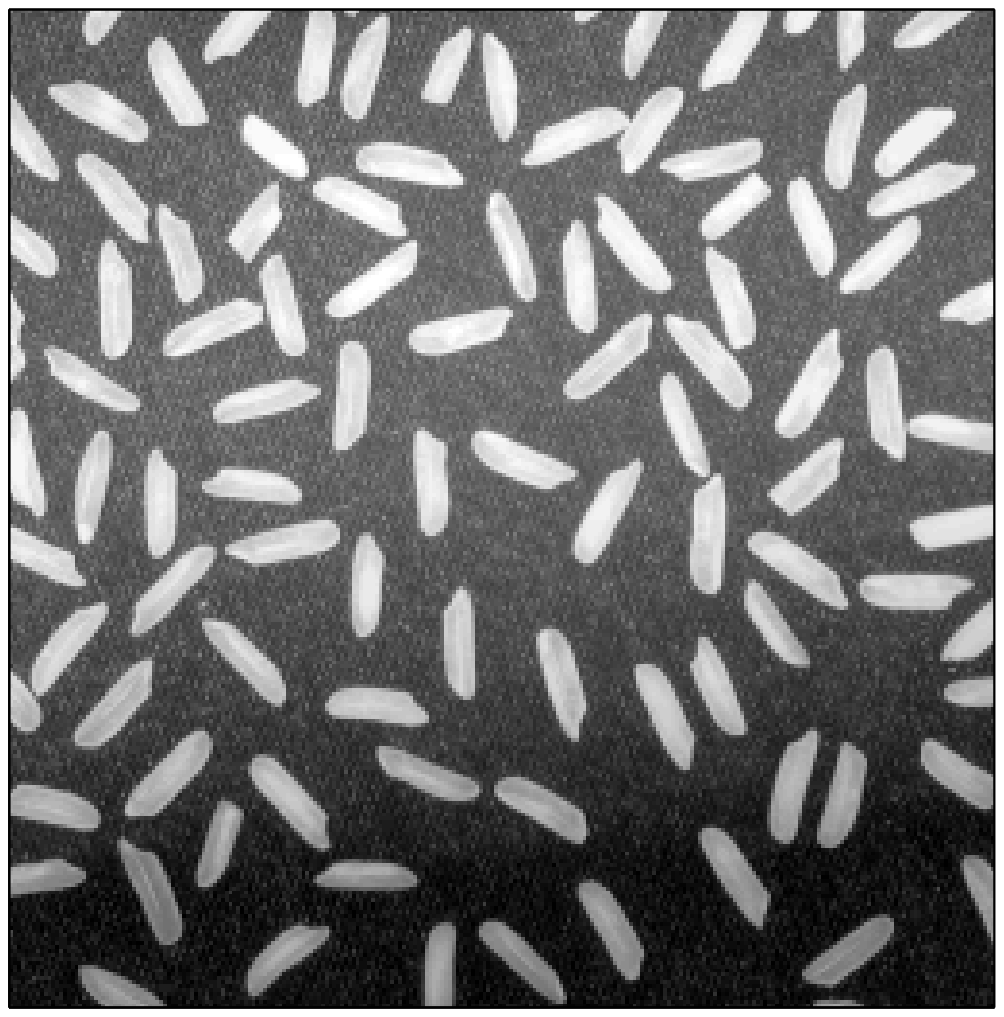}
\caption{\small Initial gray level image with $256 \times 256$ pixels.}
\label{Fig:Rice}
\vspace{0.7cm}
\begin{tabular}{ccc}
\includegraphics[width=5cm]{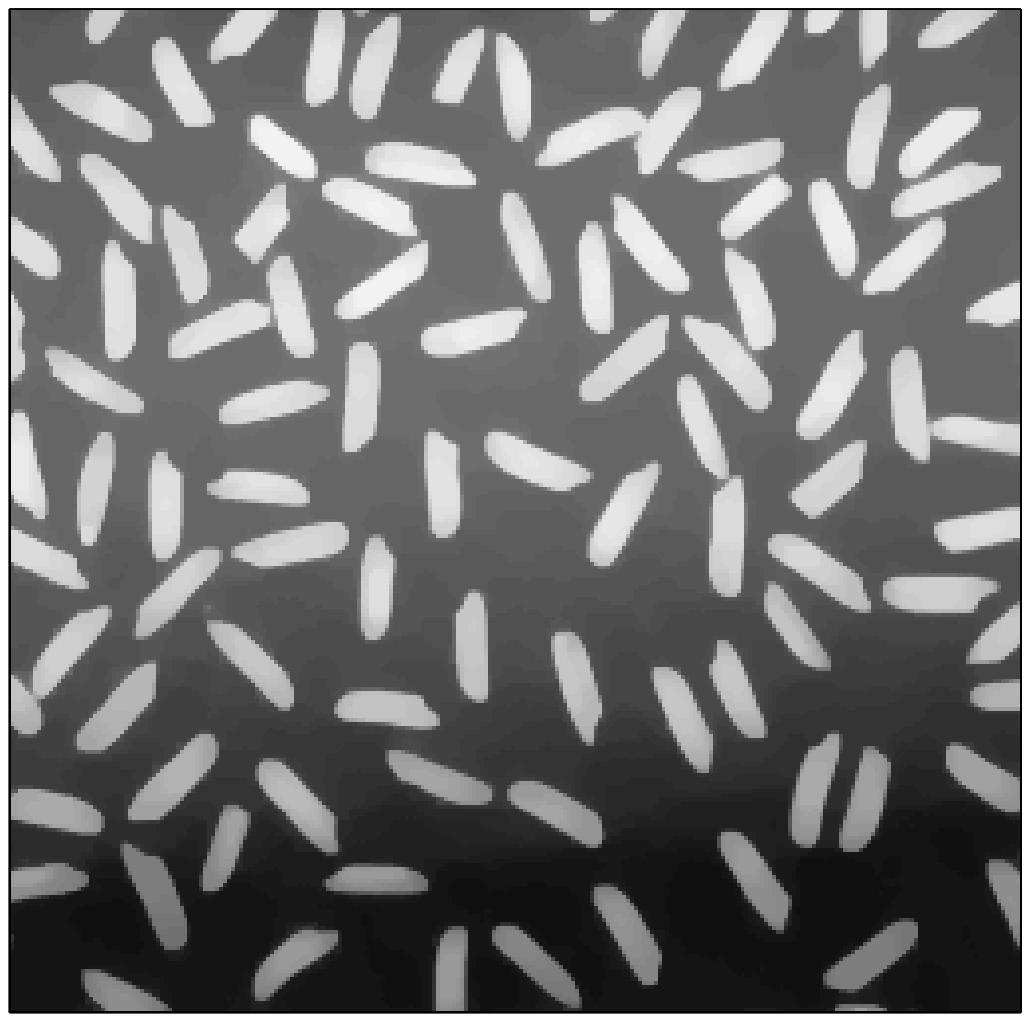}&
\includegraphics[width=5cm]{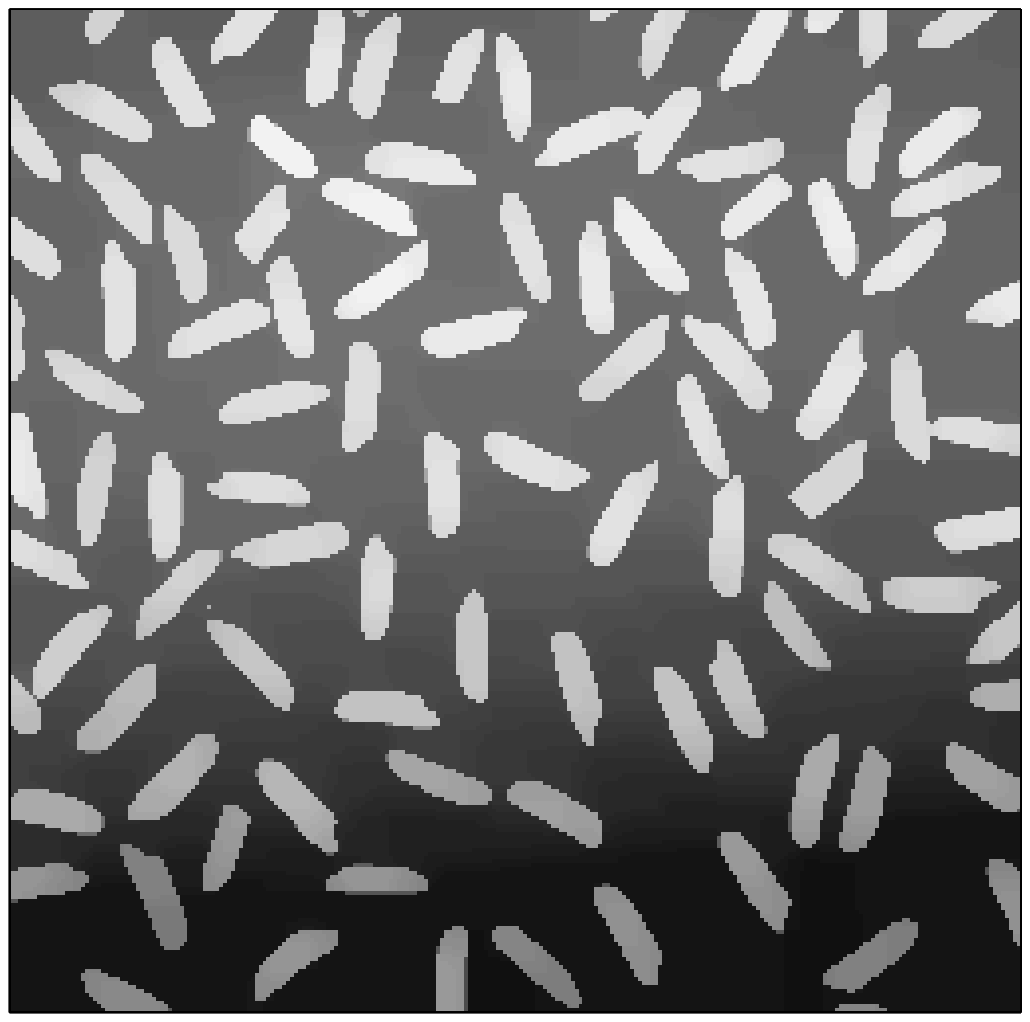}&
\includegraphics[width=5cm]{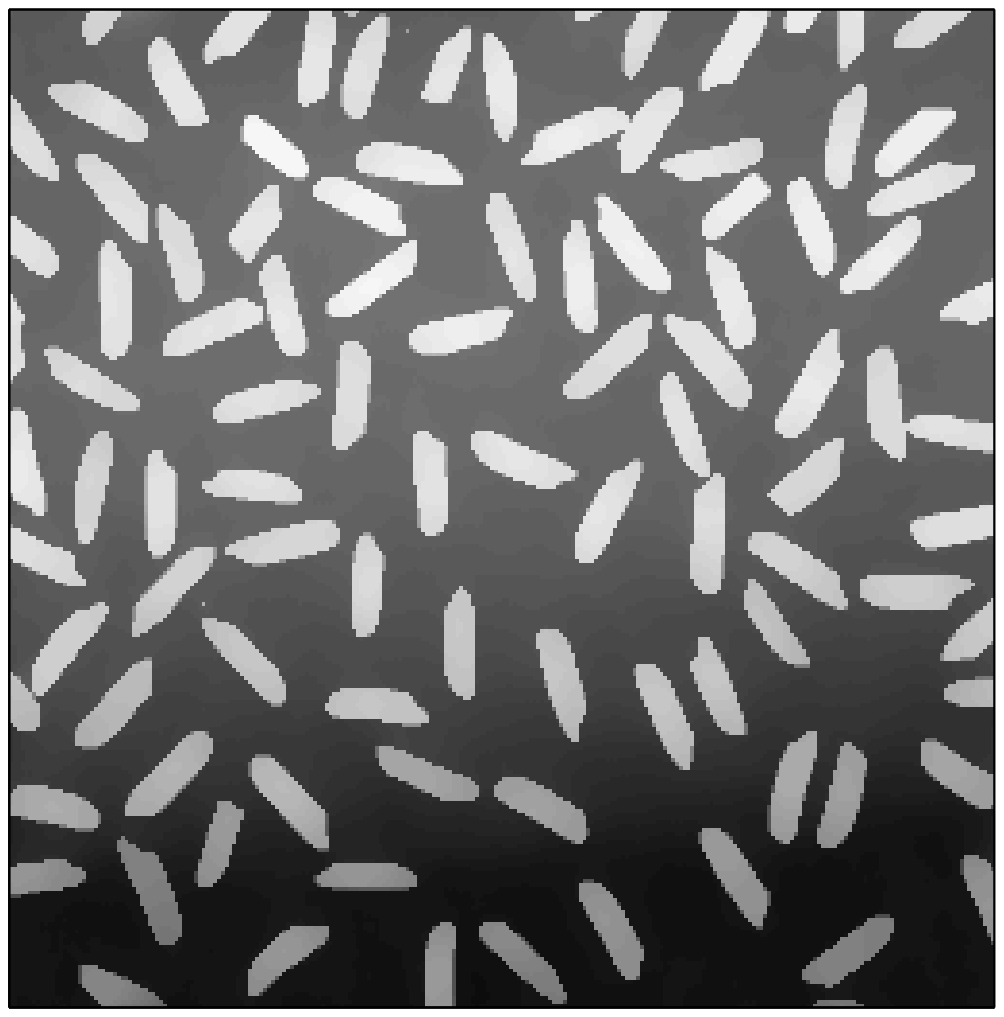}\\
\includegraphics[width=5cm]{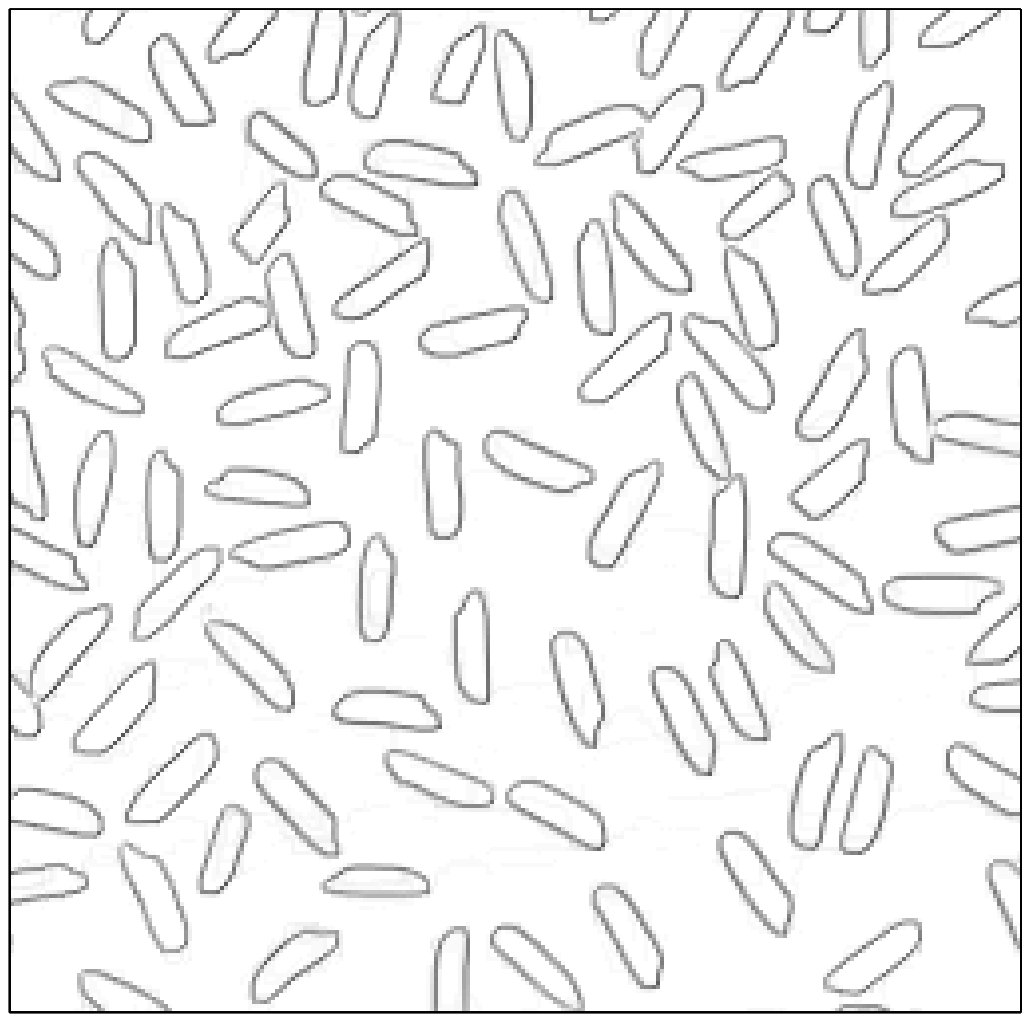}&
\includegraphics[width=5cm]{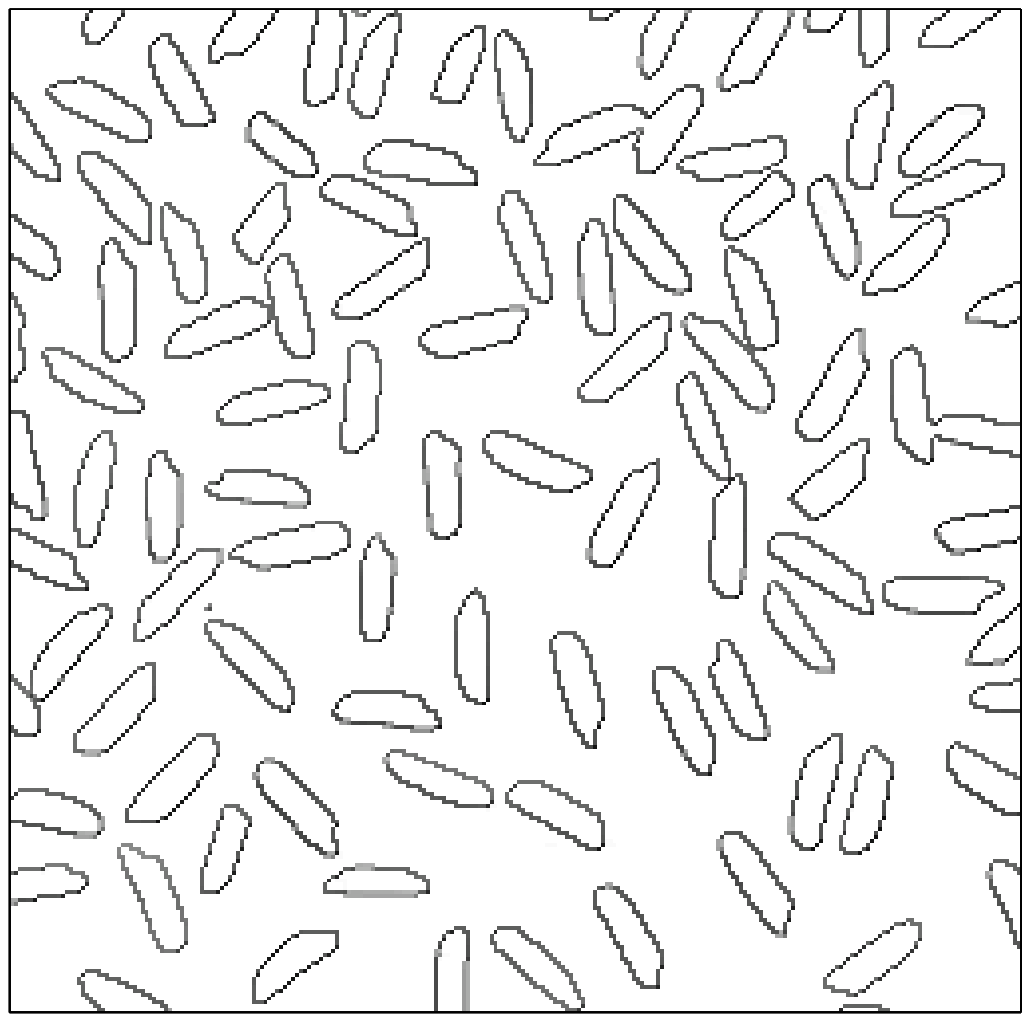}&
\includegraphics[width=5cm]{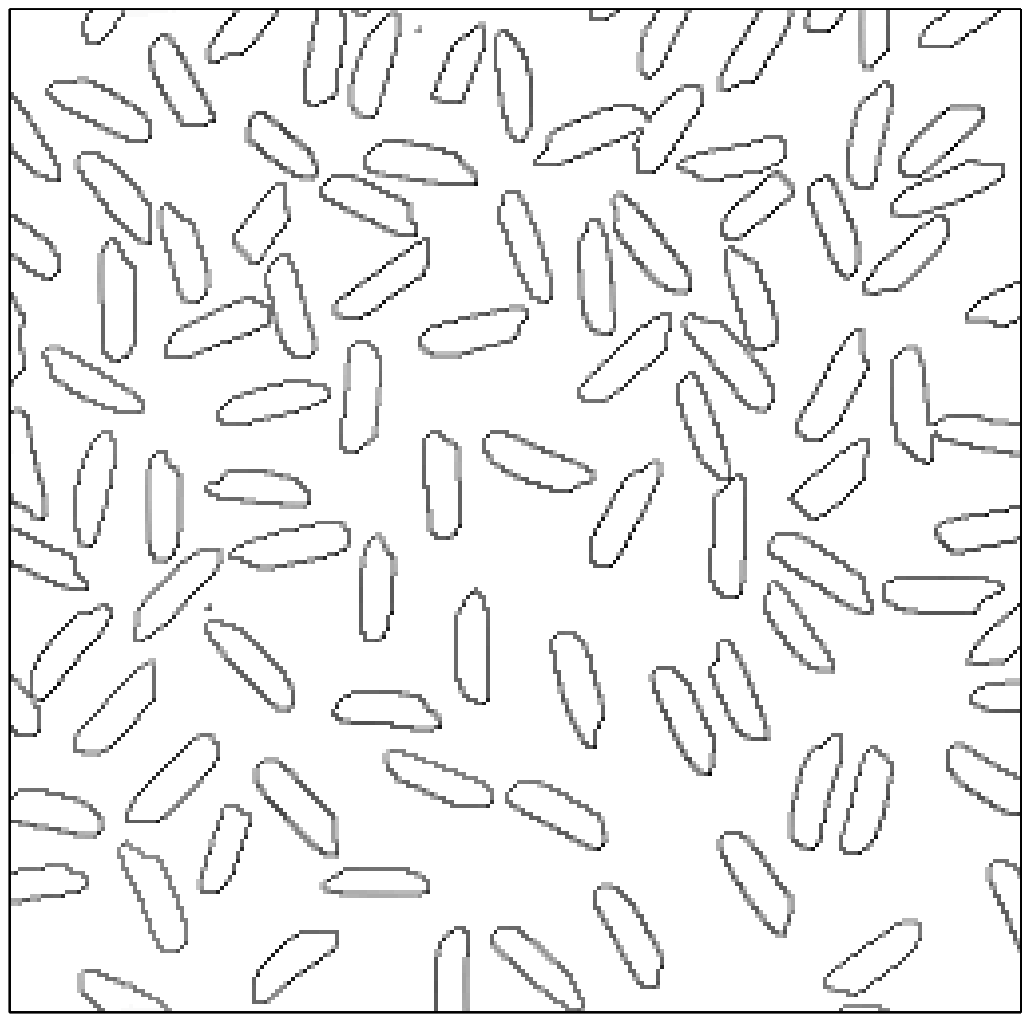}\\
$(a)$ & $(b)$ & $(c)$
\end{tabular}
\caption{$(a)$ Segmented images and their gradient for SC penalty using 3MG, $\lambda = 2$, $\delta = 0.2$, $(b)$ for NSNC penalty using TRW, $\lambda =  1550$, $\delta = 3.5$, and $(c)$ for SNC\ref{ex:psi2} penalty using 3MG, $\lambda = 1500$, $\delta = 8$.}
\label{Fig:RiceSeg}
\end{figure}

\begin{figure}[!htbp]
\centering
\begin{tabular}{ccc}
\includegraphics[width=5cm]{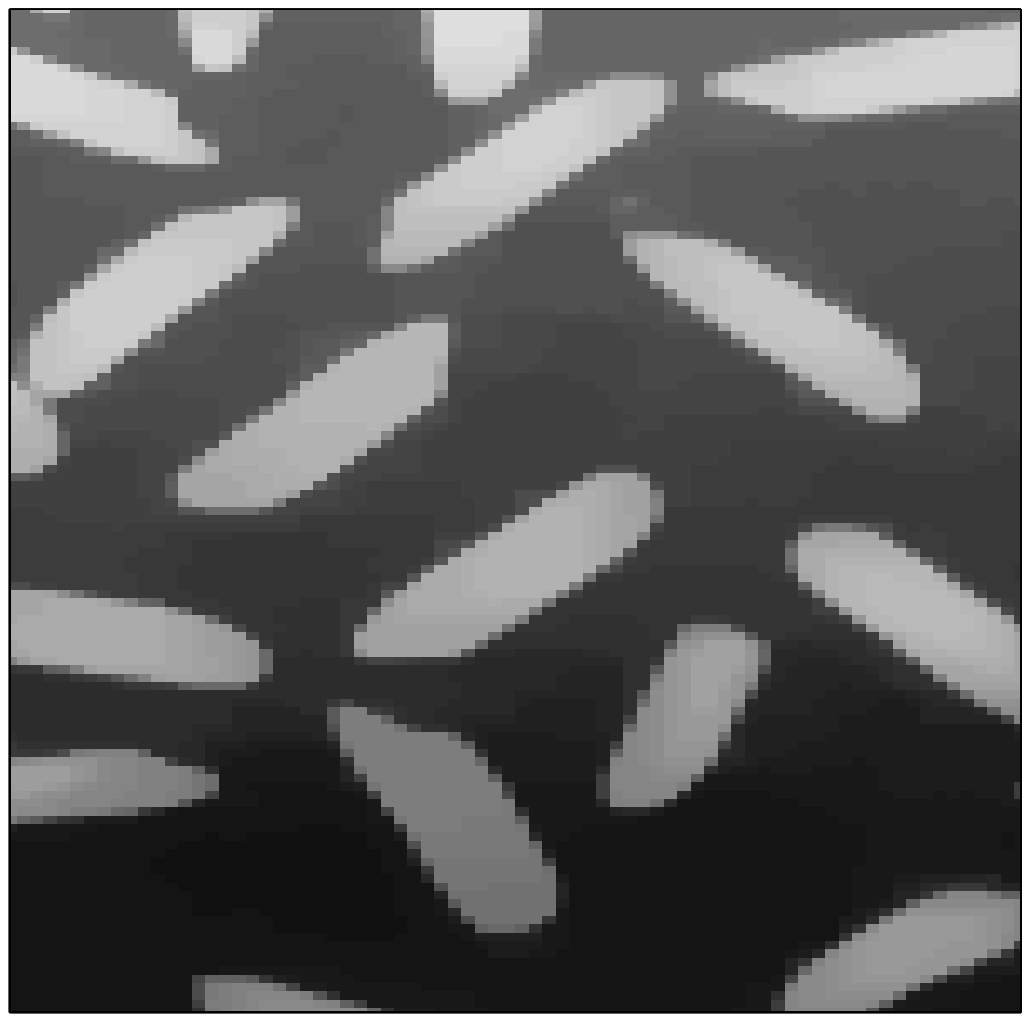}&
\includegraphics[width=5cm]{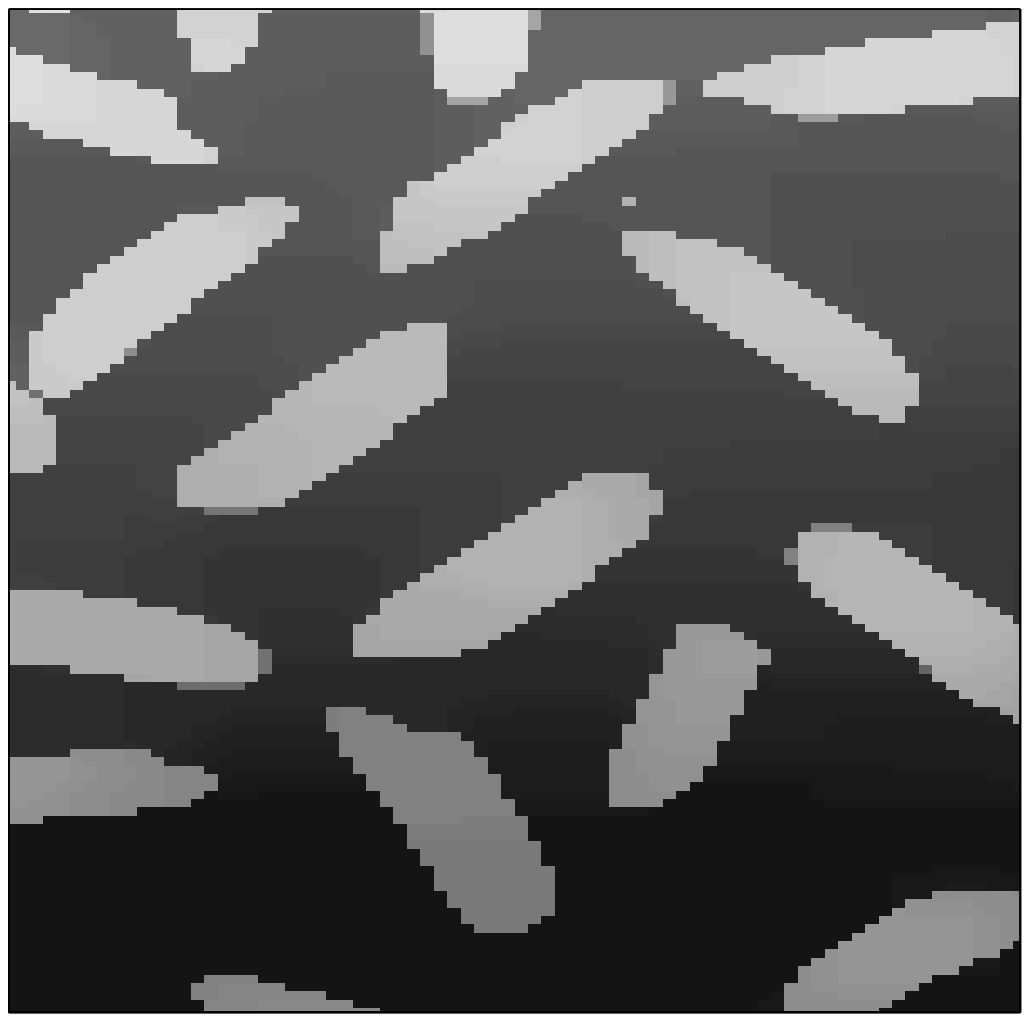}&
\includegraphics[width=5cm]{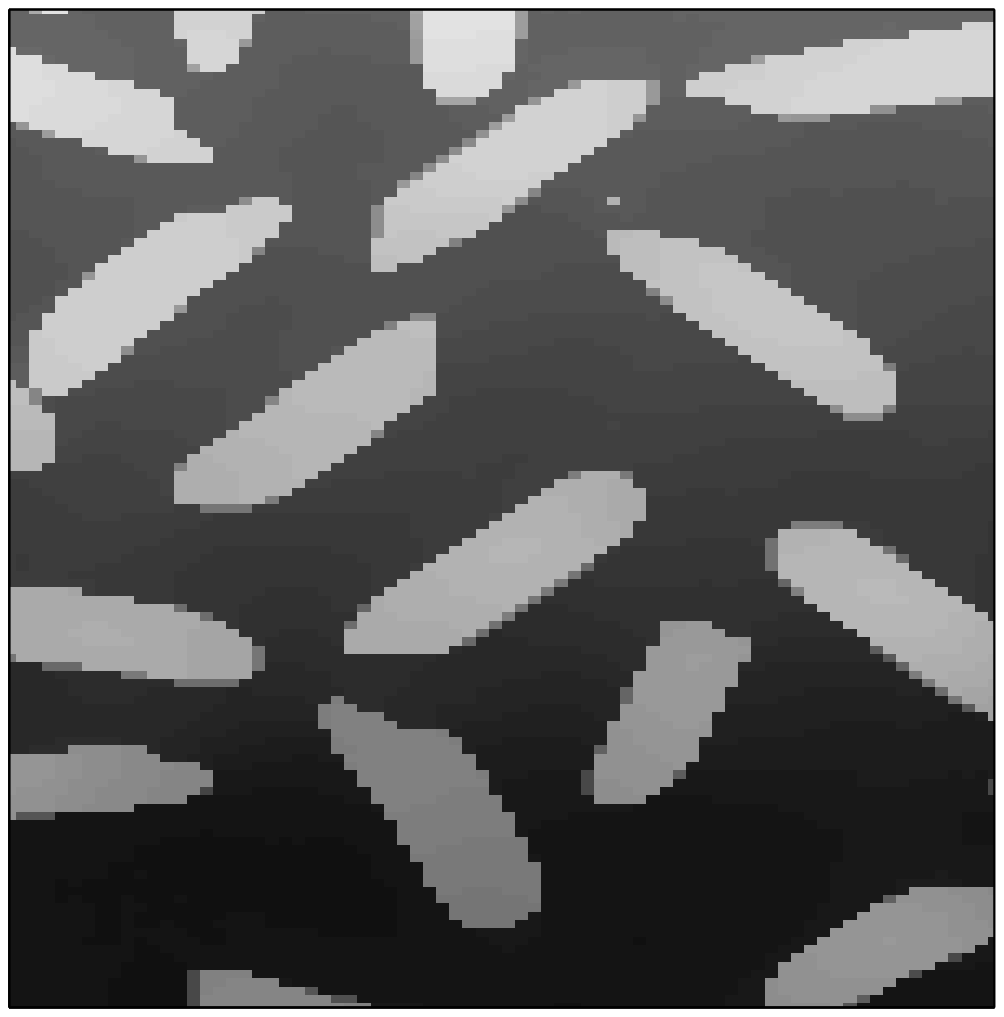}\\\\
\includegraphics[width=5cm]{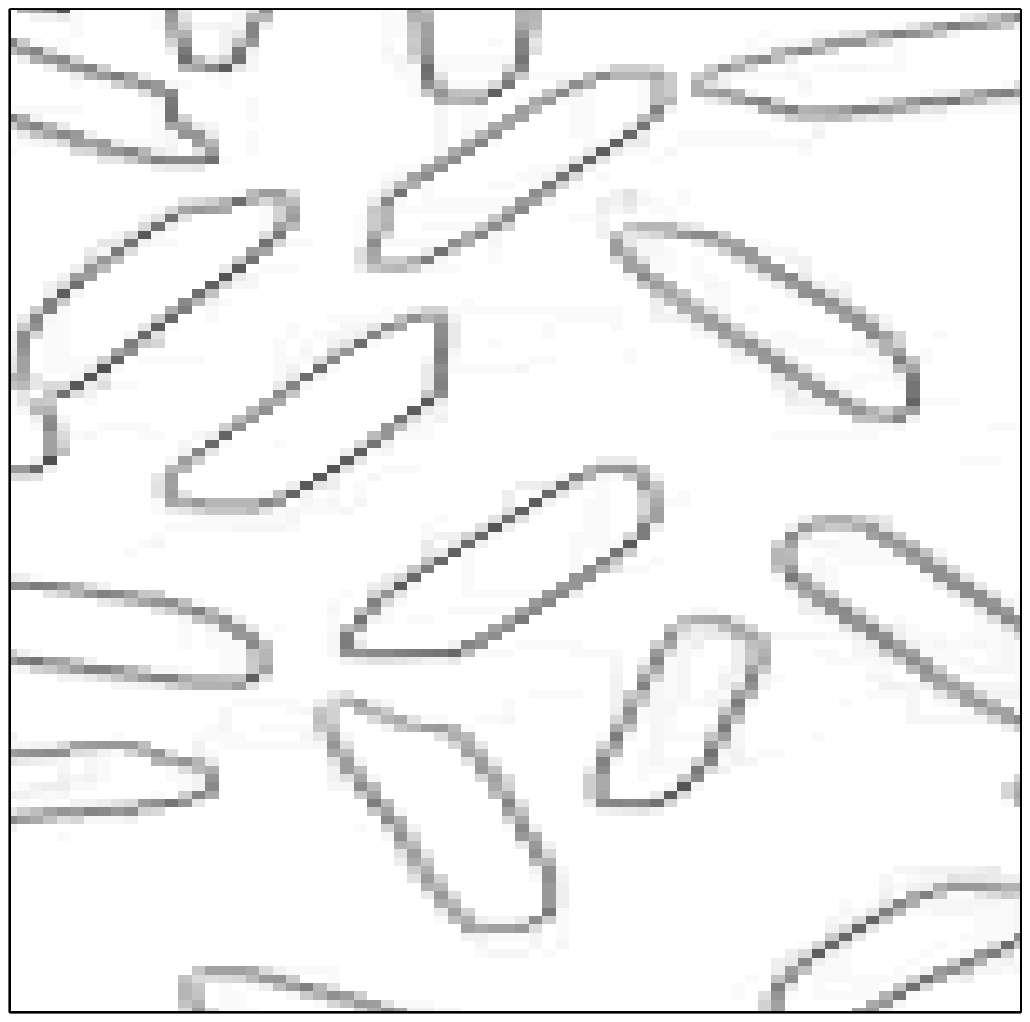}&
\includegraphics[width=5cm]{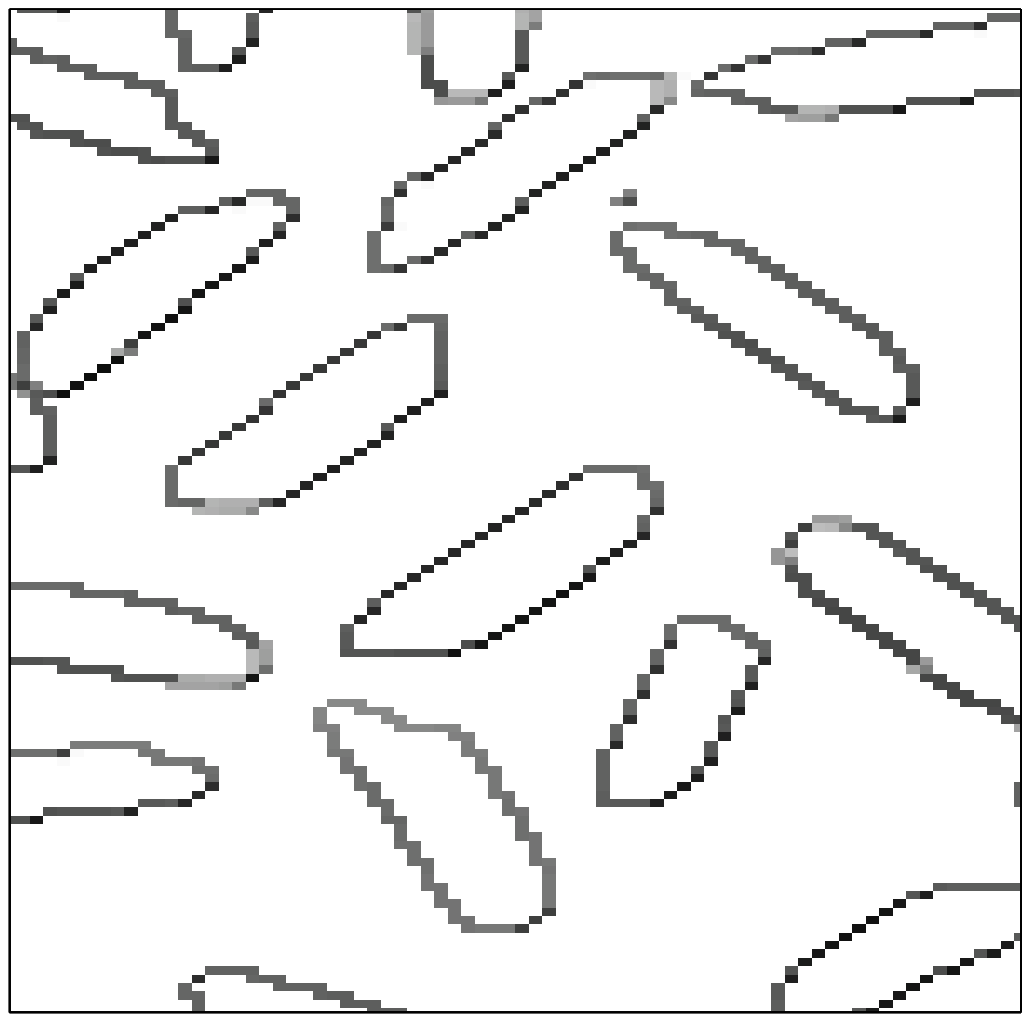}&
\includegraphics[width=5cm]{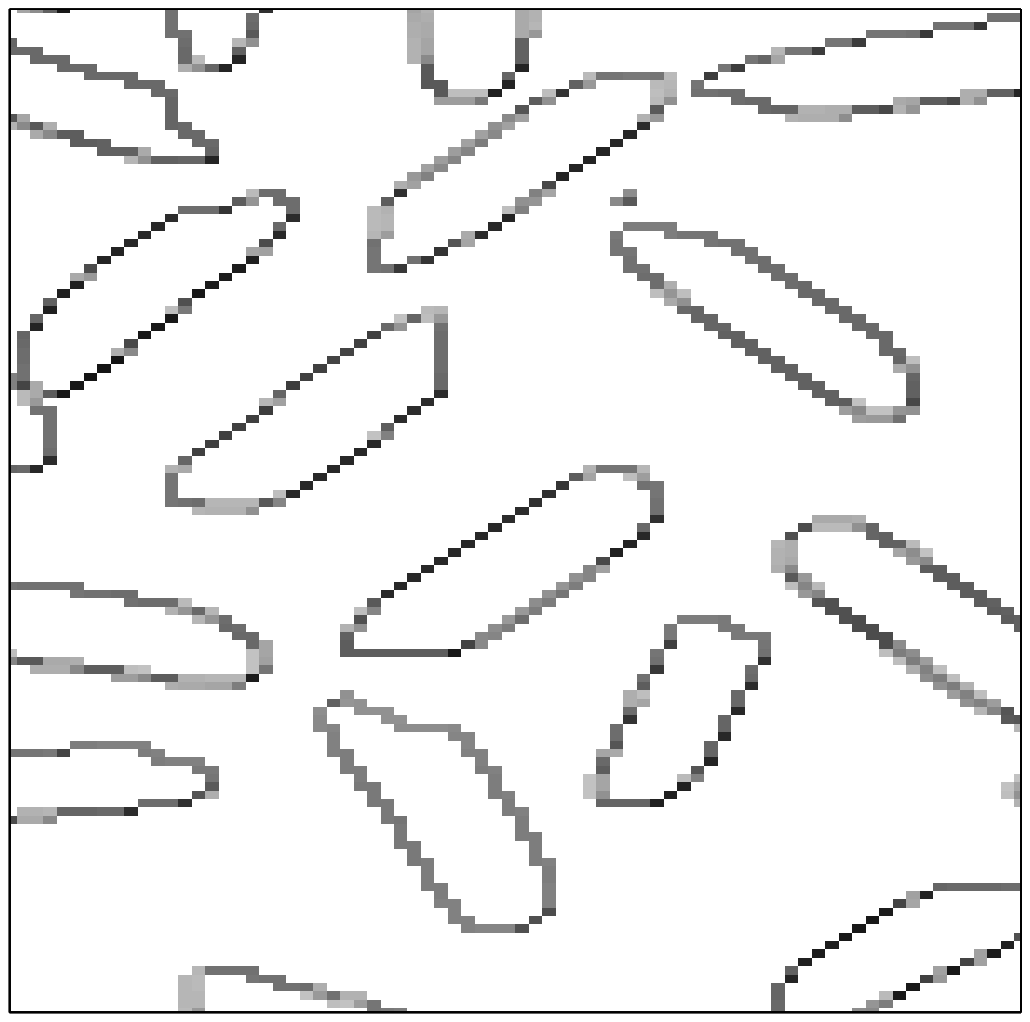}\\
$(a)$ & $(b)$ & $(c)$
\end{tabular}
\caption{Detail of segmented images and their gradient $(a)$ for SC penalty using 3MG, $\lambda = 2$, $\delta = 0.2$, $(b)$ for NSNC penalty using TRW, $\lambda =  1550$, $\delta = 3.5$, and $(c)$ for SNC\ref{ex:psi2} penalty using 3MG, $\lambda = 1500$, $\delta = 8$.}
\label{Fig:RiceSeg_zoom}
\end{figure}

\begin{figure}
\centering
\vspace{0.7cm}
\includegraphics[width=14cm]{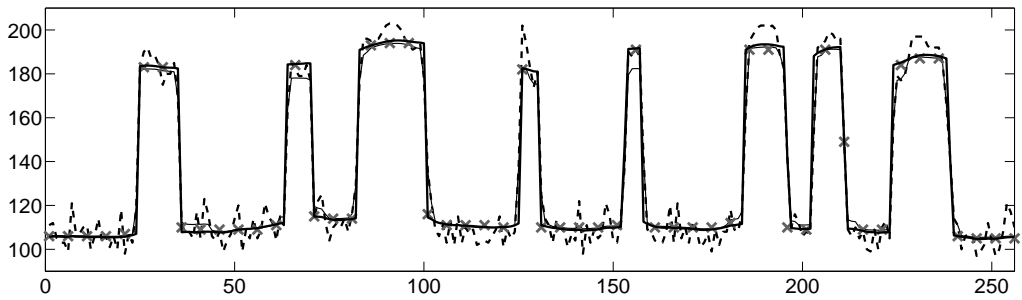}
\caption{Comparison of $50$th line of segmented images using SC (thin line), NSNC (crosses) and SNC\ref{ex:psi2} (thick line) potential functions. The $50$th line of the original image is indicated in dotted plot.}
\label{Fig:RiceSegLine}
\end{figure}

In the second experiment, we consider the segmentation of \texttt{Rice} image of
size $N=256 \times 256$ (see Figure \ref{Fig:Rice}). We define the segmented image
as a minimizer of $F_\delta$, where $\Hb = \Ib$, $\yb$ identifies with the
original image and $(\forall \zb\in \eR^N)$ $\Phi(\zb) =
\frac{1}{2} \|\zb\|^2$.
The anisotropic penalization term is again used with $\tau = 0$ for the same
reason as earlier.  Figs.~\ref{Fig:RiceSeg} and \ref{Fig:RiceSeg_zoom}
illustrate the resulting images and their gradient for SC, NSNC and SNC\ref{ex:psi2} penalty functions, 
when regularization parameters $(\lambda,\delta)$ are
tuned in order to obtain the best visual results in terms of segmentation. The
gradients of the resulting images are evaluated by displaying, for every $n \in
\left\{1,\dots,N\right\}$, $G_n = \| \Deltab_n \widehat{\xb}\|$ with $\Deltab_n
= [\Deltab_n^{\mathrm{h}}\;\; \Deltab_{n}^{\mathrm{v}}]^\top \in \eR^{2\times
  N}$ where $\Deltab_n^{\mathrm{h}}\in \eR^N$ and $\Deltab_{n}^{\mathrm{v}} \in
\eR^N$ represent the first-order difference operators in the horizontal and
vertical directions. Finally, the intensity values along the (arbitrarily
chosen) $50$th line of each image are plotted in Figure \ref{Fig:RiceSegLine} to
better illustrate the behaviors of the different approaches.


 \begin{table}[h]
\centering
\renewcommand{\arraystretch}{1.2}
\begin{tabular}{|c|c|c|c|c| }
\hline
Penalty function$(\lambda,\delta)$ &Algorithm &  Iteration & Time & $F_\delta$ \\
\hline
\hline
SC $(2,0.2)$ 
& 3MG &  $132$ &  \underline{$0.99$} & $6.69\cdot 10^6$ \\ 
& \newtext[NLCG-HS] & $144$ & $1.49$& $6.69\cdot 10^6$ \\  
& \newtext[NLCG-PRP+] & \newtext[$143$] & \newtext[$1.47$]& \newtext[$6.69\cdot 10^6$] \\  
& \newtext[NLCG-LS] &  \newtext[$148$] &  \newtext[$1.54$]  & \newtext[$6.69\cdot 10^6$] \\ 
& L-BFGS & $215$ & $3.44$ & $6.69\cdot 10^6$ \\  
& HQ  &   $898$ & $18.19$ & $6.69\cdot 10^6$ \\ 
\hline
\hline
SNC\ref{ex:psi2} $(1500,8)$ 
& 3MG  &  \newtext[$491$] &   \underline{\newtext[$3.43$]} & $1.59\cdot 10^7$ \\ 
& \newtext[NLCG-HS] &   $1578$ & $14.93$ & $1.59\cdot 10^7$ \\  
& \newtext[NLCG-PRP+] &  \newtext[$463$] & \newtext[$4.25$] & \newtext[$1.59\cdot 10^7$] \\ 
& \newtext[NLCG-LS] &  \newtext[$598$] & \newtext[$5.64$]   & \newtext[$1.59\cdot 10^7$] \\    
& L-BFGS &   $632$ & $9.57$ & $1.59\cdot 10^7$ \\  
& HQ &  $3553$ & $65.2$ & $1.59\cdot 10^7$\\   
\hline
\hline
NSNC $(1550,3.5)$ 
& $\alpha$-EXP  &   $9$   & $57.97$   &  $5.58\cdot 10^6$\\  
& QCSM & $1$   & $7.05$   &  $5.52\cdot 10^6$\\  
& TRW & $ 5$ & \underline{$6.71$} &  $5.52\cdot 10^6$\\  
& BP  & $50 $ & $61.83$ &   $5.52\cdot 10^6$\\  
\hline
\end{tabular} 
\caption{Results for the segmentation problem.}
\label{Tab:RiceIt}
\end{table}

According to Tab.~\ref{Tab:RiceIt}, the best performance in terms of
computational time is obtained by the 3MG algorithm with the SC
penalty. However, the convex penalization strategy leads to poor segmentation
results. Indeed, the boundaries of the reconstructed image are smooth and the
background suffers from staircasing effect. In contrast, the nonconvex penalties give
rise to truly piecewise constant images. The considered algorithms for the
truncated quadratic penalty lead to segmented images very similar to the one
obtained with SNC regularization. However, Tab.~\ref{Tab:RiceIt} shows that
they are more demanding in terms of computational time than 3MG.

\subsection{Image deblurring}

Our third experiment corresponds to the problem of restoring the
\texttt{montage} image $\overline{x}$, with size $256 \times 256$, from blurred
and noisy observations $\ub = \Rb \overline{\xb}+\wb$ where $\wb$ is a
realization of a zero-mean white Gaussian noise and $\Rb$ models a linear
uniform blur with size $3 \times 3$.  The recovery of the original image is performed by solving \eqref{Eq_CritJ} with $Q = 2N$,
\begin{equation*}
\Hb = \left[\begin{array}{c} \Rb \\ \Ib \end{array} \right] \qquad \yb = \left[\begin{array}{c} \ub \\ \zerob\end{array} \right],
\end{equation*}
and 
\begin{equation*}
(\forall \zb =(z_q)_{1\le q \le 2N})\qquad 
\Phi(\zb) = \frac{1}{2} \left(\sum_{q=1}^N z_q^2 + \beta \sum_{q=N+1}^{2N} d_B^2(z_q) \right),
\end{equation*}
where $d_B$ denotes the distance to the closed convex interval $B = [0,255]$ and
$\beta =0.01$. Furthermore, function $\Psi_{\delta}$ is given by \eqref{e:Psid} with $\tau = 10^{-10}$ and $S = 2N$. We consider, for every $s\in \{1,\ldots,N\}$, an isotropic regularization between neighbooring pixels, i.e., $P_s = 2$ and $\Vb_s = [\Deltab_s^{\mathrm{h}}\;\; \Deltab_s^{\mathrm{v}}]^\top$ where $\Deltab_s^{\mathrm{h}}\in \eR^N$ (resp.  $\Deltab_s^{\mathrm{v}} \in \eR^N$) corresponds to a horizontal (resp. vertical) gradient operator, and, for every $s\in \{N+1,\ldots,2N\}$, the Hessian-based penalization from \cite{Lefkimmiatis12} i.e., $P_s = 3$ and $\Vb_s = [\Deltab_s^{\mathrm{hh}}\;\; \sqrt{2}\Deltab_s^{\mathrm{hv}}\;\;\Deltab_s^{\mathrm{vv}}]^\top$ where $\Deltab_s^{\mathrm{hh}}\in \eR^N$, $\Deltab_s^{\mathrm{hv}}\in \eR^N$ and $\Deltab_s^{\mathrm{vv}}\in \eR^N$ model the second-order finite difference operators between neighbooring pixels, as described in \cite[Sec.III-A]{Lefkimmiatis12}. 
For $s\in \{N+1,\ldots,2N\}$ we consider the $\ell_2-\ell_1$ function $\psi_{s,\delta} \colon t \mapsto \rho (\sqrt{1 +
  \froc{t^2}{(\theta \delta)^2}}-1)$, where $\rho$ and $\theta$ take positive values. 
Tab.~\ref{Tab:Montage} presents the results for SC and SNC\ref{ex:psi1} regularization of the image gradient (i.e. $\psi_{s,\delta}$ for $s\in \{1,\ldots,N\}$ ). 
   Parameters $(\rho,\theta,\lambda,\delta)$ are tuned to maximize the SNR of the restored image.
    In both cases, the 3MG algorithm outperforms the three considered descent
    algorithms in terms of time efficiency. Additionally, the nonconvex strategy
    leads to better results in terms of SNR (see Figure~\ref{Fig:Montage}). One can
    also observe that in this case the staircasing effect is reduced (see some
    image details in Figure~\ref{Fig:MontageRec}).  
\begin{figure}[!hftbp]
\centering
\begin{tabular}{cc}
\includegraphics[width=5cm]{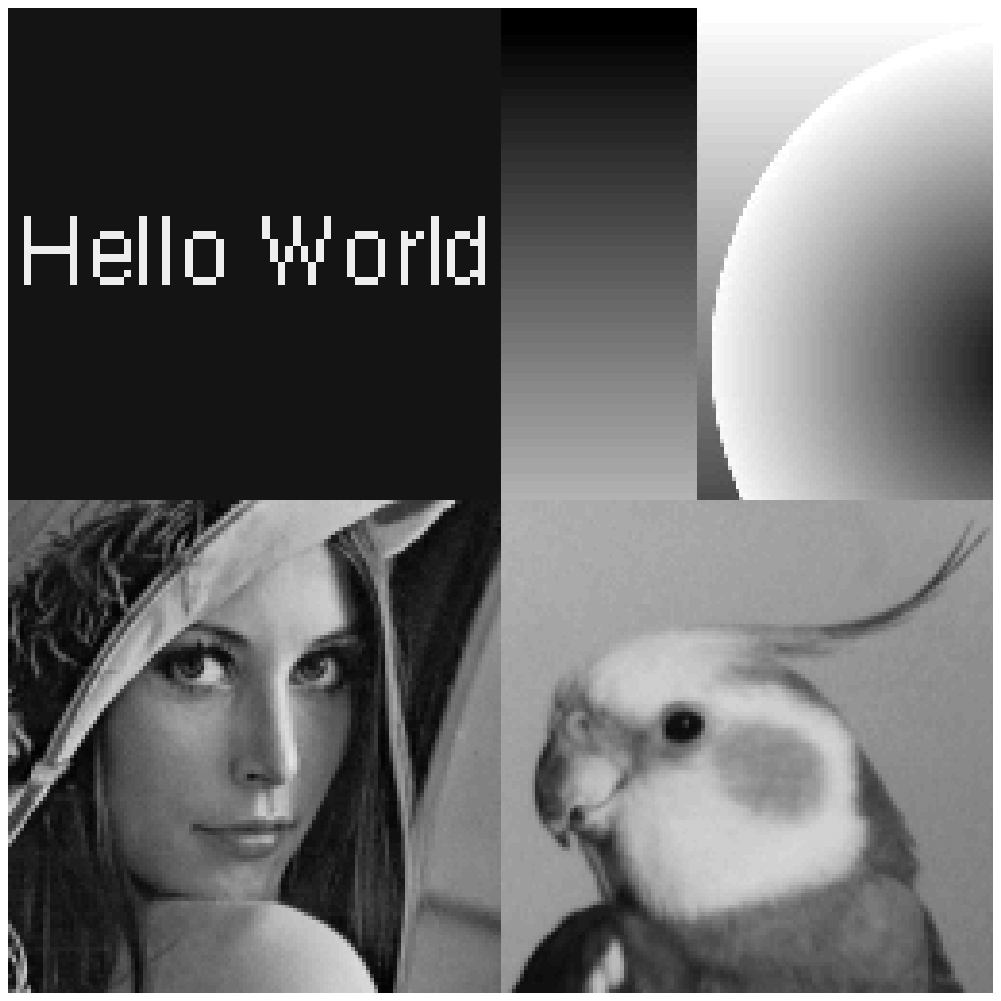}&
\includegraphics[width=5cm]{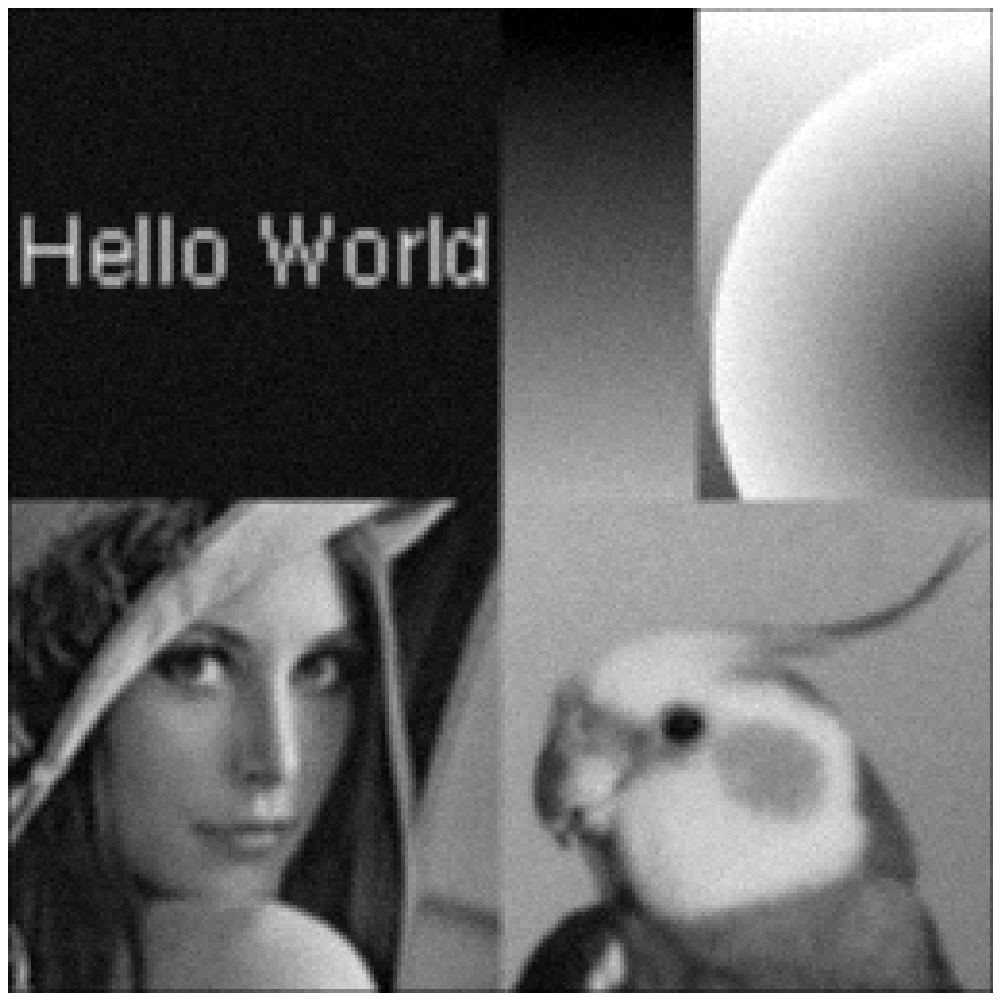}\\
$(a)$ & $(b)$ 
\end{tabular}
\caption{$(a)$ Original image with $256 \times 256$ pixels and $(b)$ blurred noisy image with SNR$= 18.65$ dB, MSSIM $= 0.82$, $3\times3$ uniform blur, noise standard deviation equal to $4$.}
\label{Fig:Montage}
\vspace{0.7cm}
\centering
\begin{tabular}{cc}
\includegraphics[width=5cm]{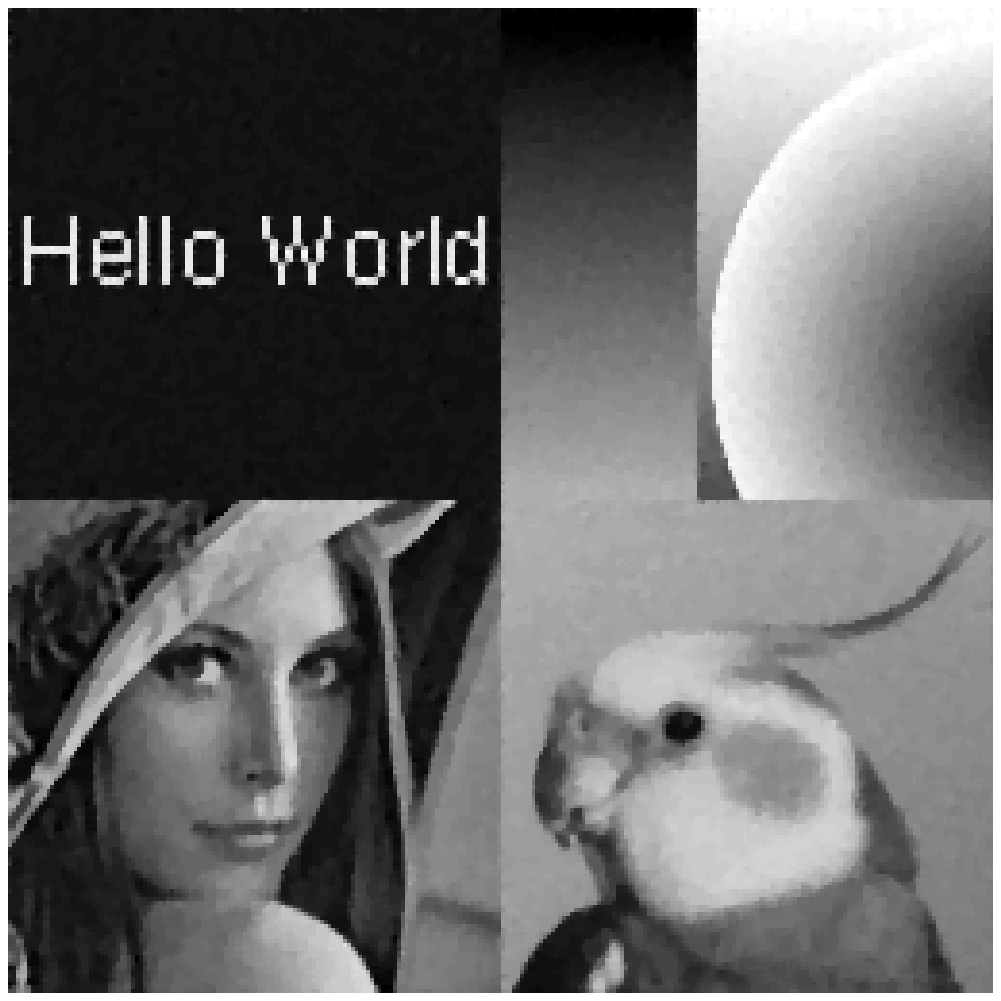} &
\includegraphics[width=5cm]{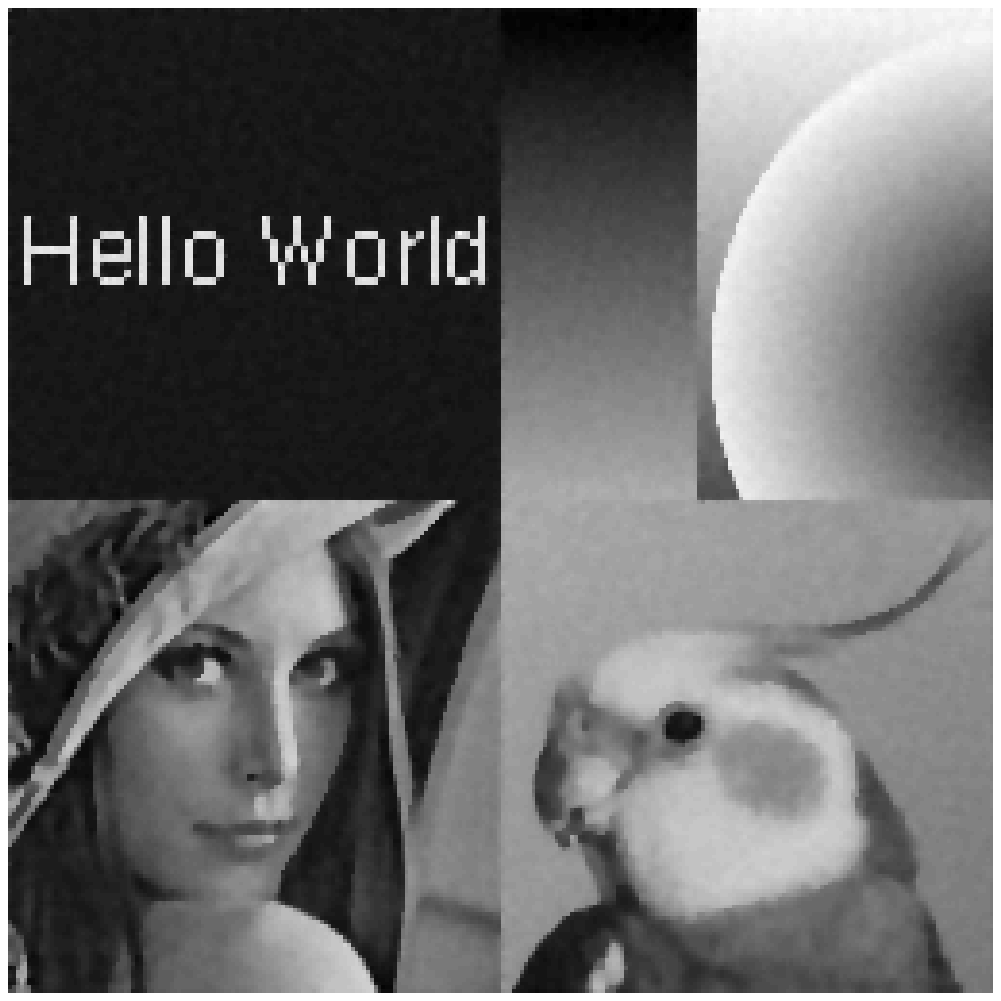}\\
\includegraphics[width=5cm]{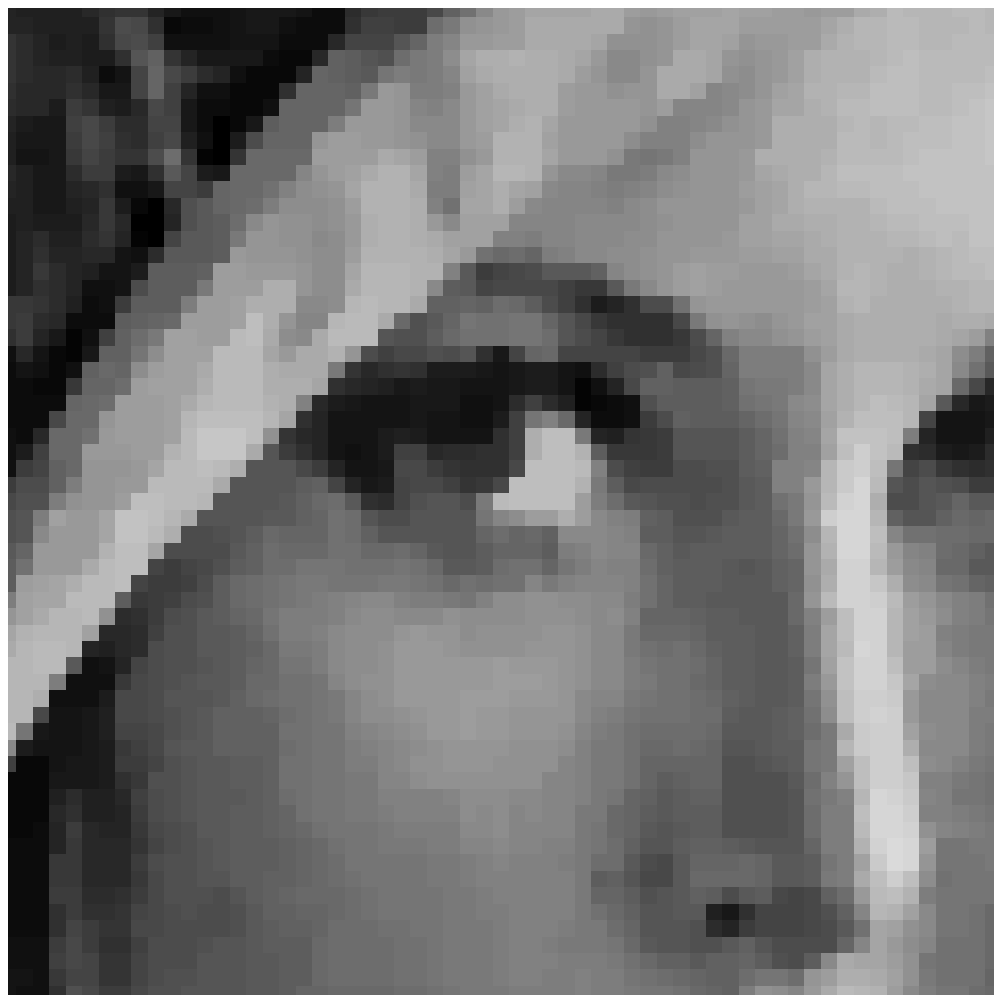} &
\includegraphics[width=5cm]{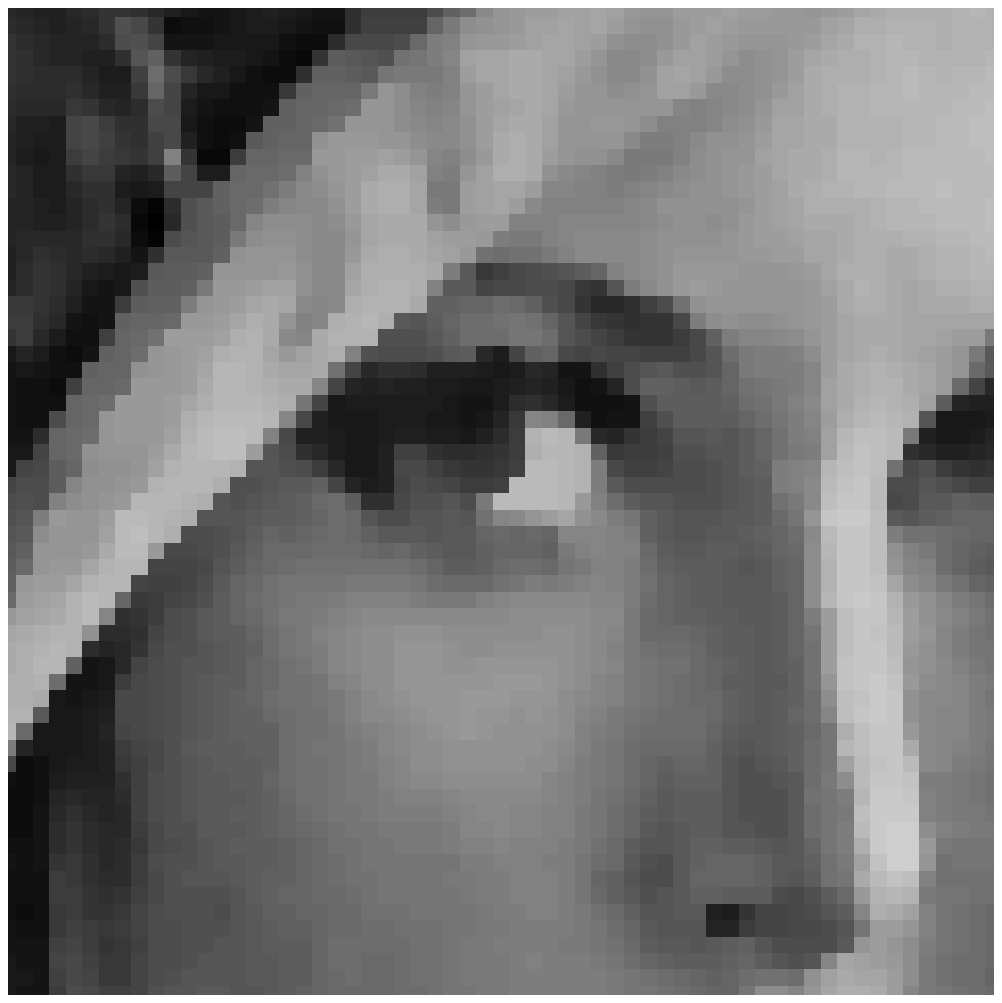}\\
\includegraphics[width=5cm]{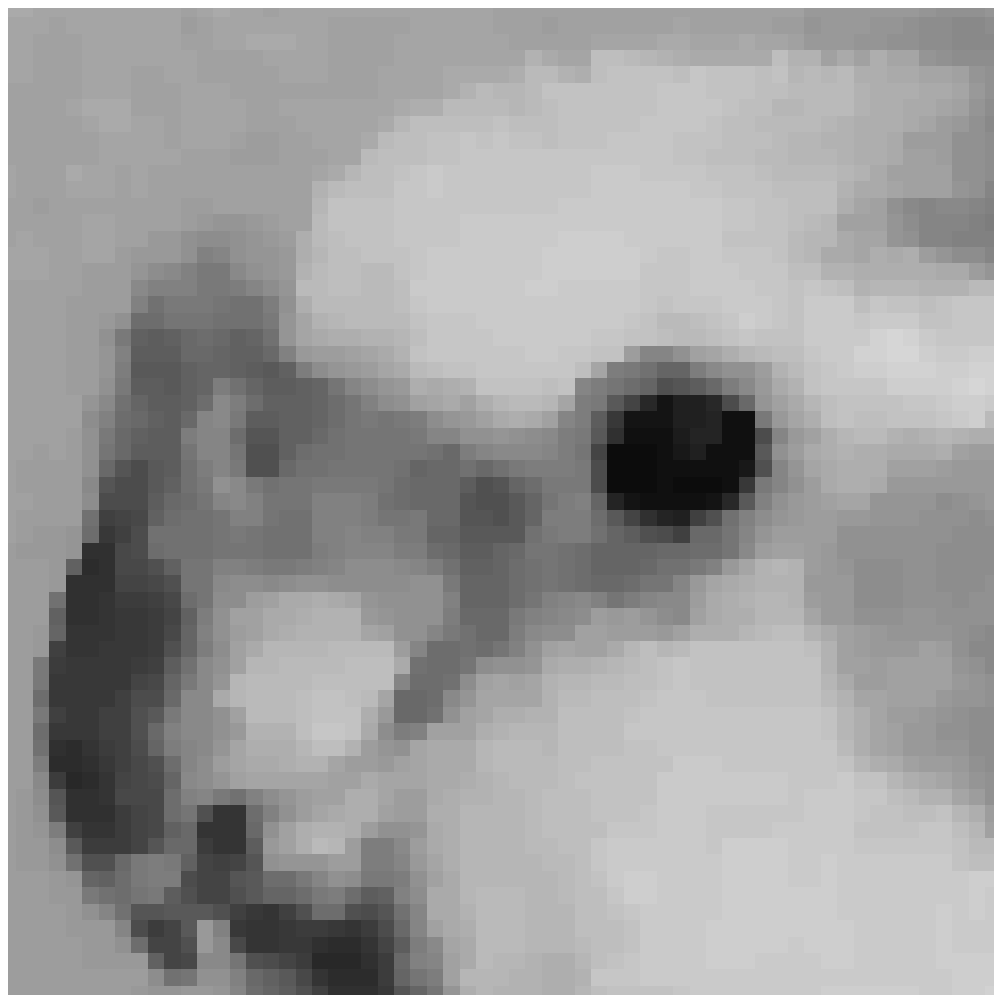} &
\includegraphics[width=5cm]{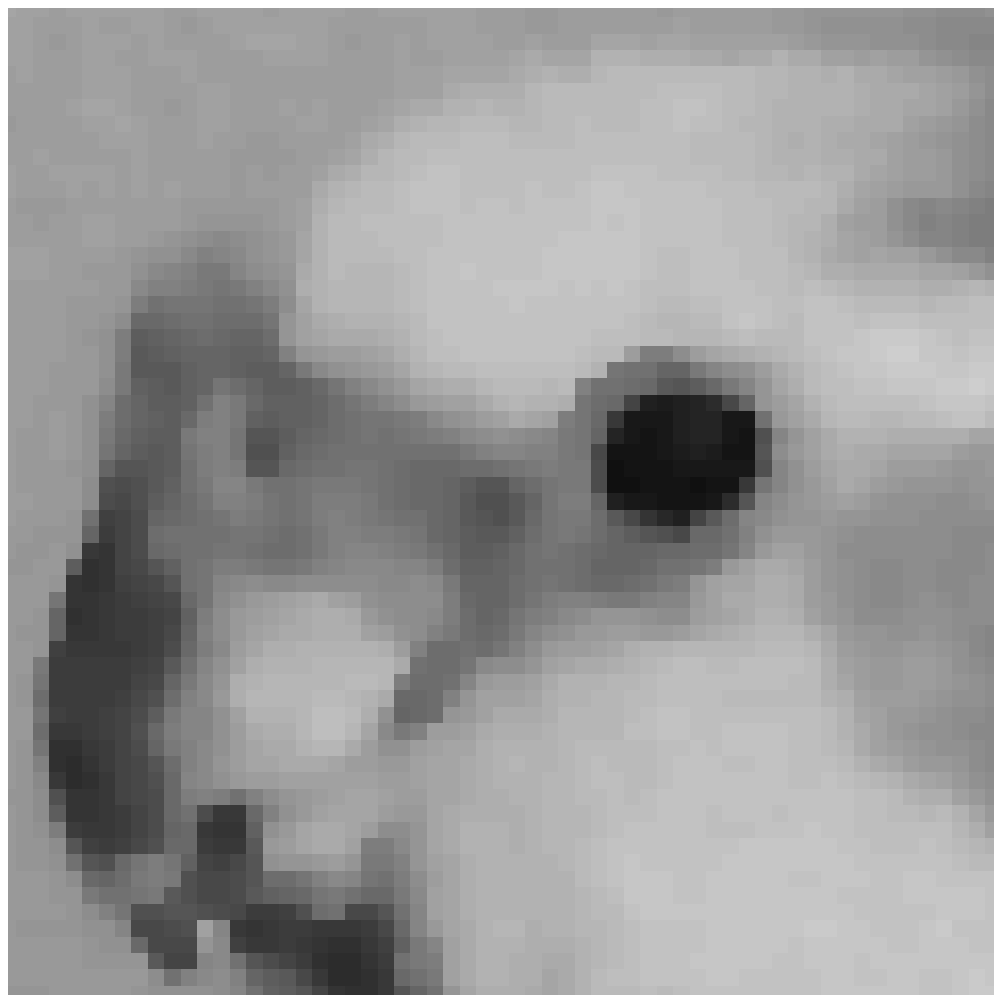}\\
$(a)$ & $(b)$
\end{tabular}
\caption{Deblurring results with $(a)$ SC penalty using 3MG, $\rho = 0.56$, $\theta = 0.18$, $\lambda = 0.042$, $\delta = 4.19$, SNR = $26.90$ dB, MSSIM = $0.94$ and $(b)$ with SNC\ref{ex:psi1} penalty using 3MG, $\rho = 41.55$, $\theta = 0.86$, $\lambda = 3.68$, $\delta = 18.65$, SNR = $27.69$ dB, MSSIM = $0.94$.}
\label{Fig:MontageRec}
\end{figure} 
 
  \begin{table}[h]
\centering
\renewcommand{\arraystretch}{1.2}
\begin{tabular}{|c|c|c|c|c|c| }
\hline
Penalty function$(\rho,\theta,\lambda,\delta)$ & Algorithm &  Iteration & Time & $F_\delta$ & SNR\\
\hline
\hline
SC $(0.56,0.18,0.042,4.19)$ 
& 3MG &  $121$ & $\underline{8.36}$   & $8.22\cdot 10^6$ & $26.90$\\ 
& NLCG-HS & $121$ & $8.92$& $8.22\cdot 10^6$ & $26.90$\\  
& NLCG-PRP+ & $129$  & $9.32$ & $8.22\cdot 10^6$ & $26.90$\\  
& NLCG-LS & $131$  & $9.51$  & $8.22\cdot 10^6$ & $26.90$\\ 
& L-BFGS & $162$ & $12.42$ &$8.22\cdot 10^6$ & $26.90$\\  
& HQ  &   $418$ & $94.3$ & $8.22\cdot 10^6$ & $26.90$\\ 
\hline
\hline
SNC\ref{ex:psi1} $(41.55,0.86,3.68,18.65)$ 
& 3MG  & $196$  & $\underline{11.58}$   & $7.92\cdot 10^6$ & $27.69$\\ 
& NLCG-HS & $243$ & $15.93$ & $7.92\cdot 10^6$ & $27.69$\\  
& NLCG-PRP+ & $221$ & $14.41$ & $7.92\cdot 10^6$ & $27.69$\\ 
& NLCG-LS & $246$  & $15.62$   & $7.92\cdot 10^6$  & $27.69$\\ 
& L-BFGS &   $216$ & $14.78$ & $7.92\cdot 10^6$ & $27.69$\\  
& HQ &  $616$ & $104.9$ & $7.92\cdot 10^6$ & $27.69$ \\   
\hline
\end{tabular} 
\caption{Results for the deblurring problem.}
\label{Tab:Montage}
\end{table}
 
 
\subsection{Image reconstruction}
 
In our last experiment, we consider the problem of reconstructing an image $\overline{\xb}\in \eR^N$ 
from noisy tomographic acquisitions, modeled as
\begin{equation}
\ub = \Rb \, \xb + \wb,
\end{equation} 
where $\Rb$ is the Radon projection matrix whose $(r,n)$ element ($1\le r \le
R$, $1 \le n \le N$) models the contribution of the $n$th pixel to the $r$th
datapoint, and $\wb$ represents an additive noise component. In this example, we
consider one slice of the standard \texttt{Zubal} phantom~\cite{Zubal94} with
dimensions $N=128 \times 128$, and $R = 46336$ measurements from $181$
projection lines and $256$ angles. This image is corrupted with a zero-mean
independent and identically distributed Laplacian noise (SNR = $23.5$
dB). Figure \ref{Fig:SebTomo} shows the original image and its noisy sinogram.

The reconstruction is performed by minimizing $F_{\delta}$ with $Q= R+N$,
\begin{equation}
\Hb = \left[\begin{array}{c} \Rb \\ \Ib \end{array} \right] \qquad \yb = \left[\begin{array}{c} \ub \\ \zerob\end{array} \right],
\end{equation}
and 
\begin{equation}
(\forall \zb = (z_q)_{1\le q \le Q})\qquad
\Phi(\zb) = \frac{1}{2}\left(\sum_{q=1}^{R} \sqrt{1 + (z_q/\rho)^2} + \beta \sum_{q = R+1}^{Q} d_B^2(z_q)\right)
\end{equation}
with $B = [0,255]$. Thus, $\Phi$ has a Lipschitz gradient with constant $L = \max(\frac{1}{2 \rho^2},\beta)$. In the sequel, we take $\beta = 10^{-2}$. Furthermore, the regularization function \eqref{e:Psid}, with $\tau = 10^{-10}$ and an isotropic edge-preserving penalty is considered i.e., $S = N$ and, for every $s\in \{1,\ldots,N\}$, $P_s = 2$ and $\Vb_s = [\Deltab_s^{\mathrm{h}}\;\; \Deltab_s^{\mathrm{v}}]^\top$ where $\Deltab_s^{\mathrm{h}}\in \eR^N$ (resp.  $\Deltab_s^{\mathrm{v}} \in \eR^N$) corresponds to a horizontal (resp. vertical) gradient operator.

\begin{figure}[!htbp]
\centering
\begin{tabular}{cc}
\includegraphics[height=5cm]{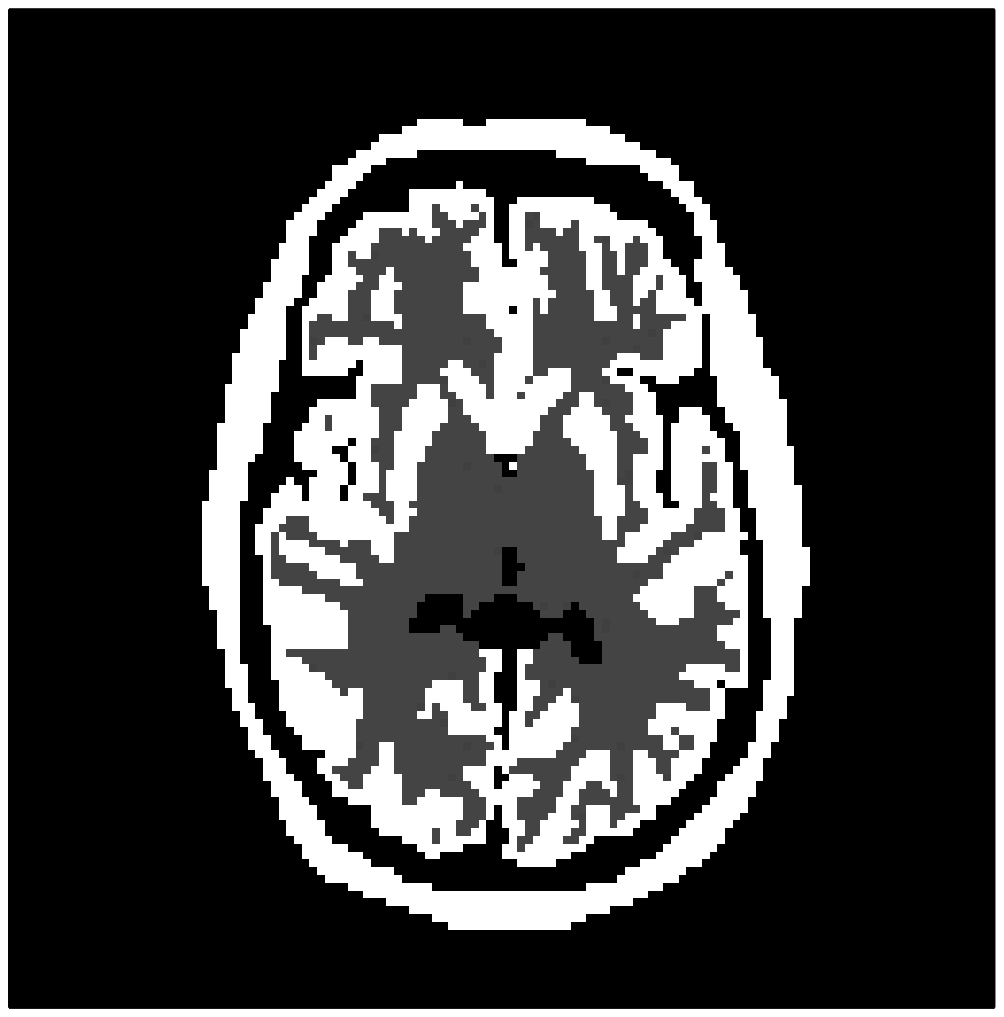}&
\includegraphics[height=5cm]{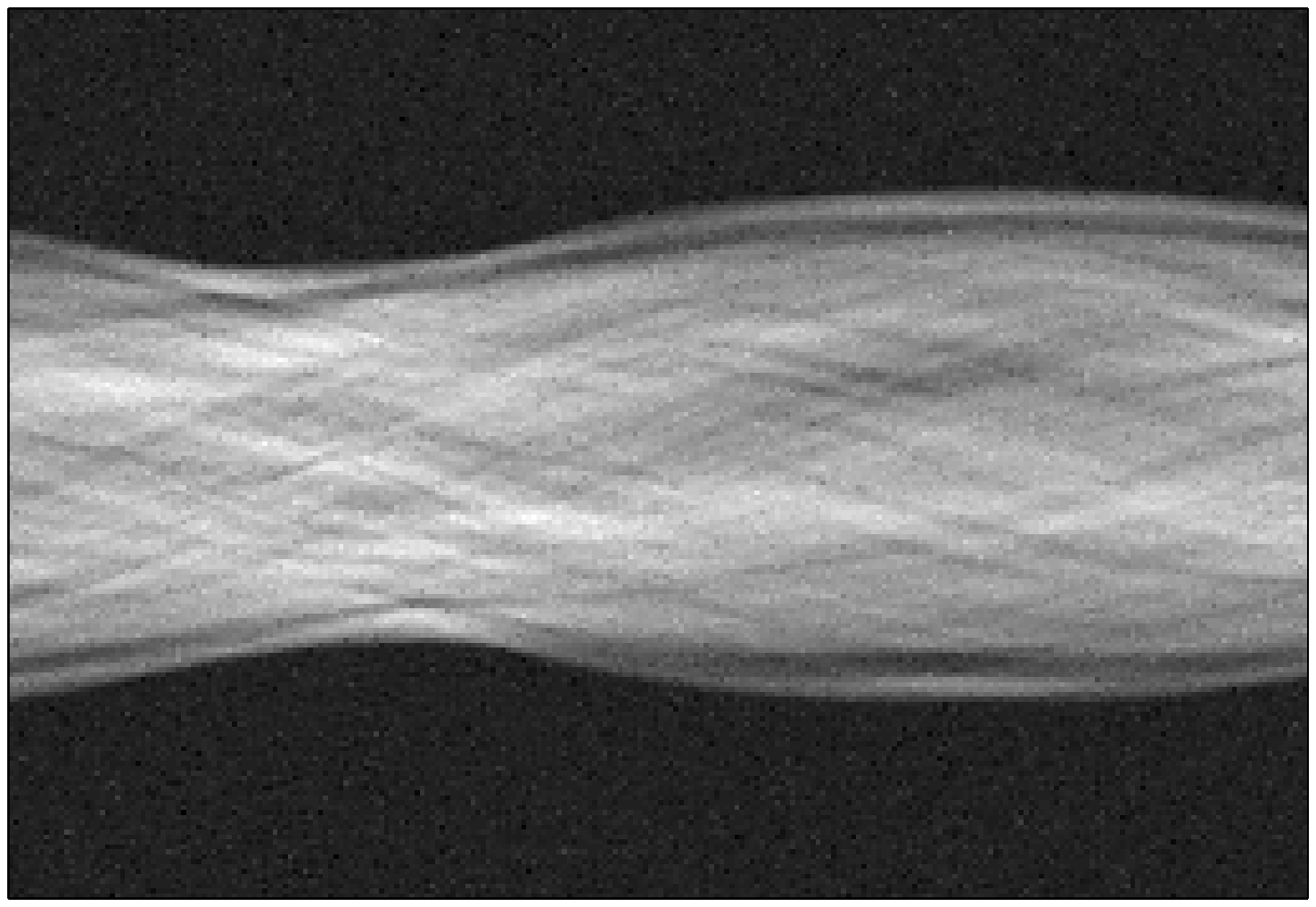}\\
$(a)$ & $(b)$
\end{tabular}
\caption{$(a)$ Initial gray level image with $128 \times 128$ pixels and $(b)$ noisy sinogram
with SNR=$23.5$ dB.}
\label{Fig:Sebal}
  \centering
\begin{tabular}{cc}
\includegraphics[height=5cm]{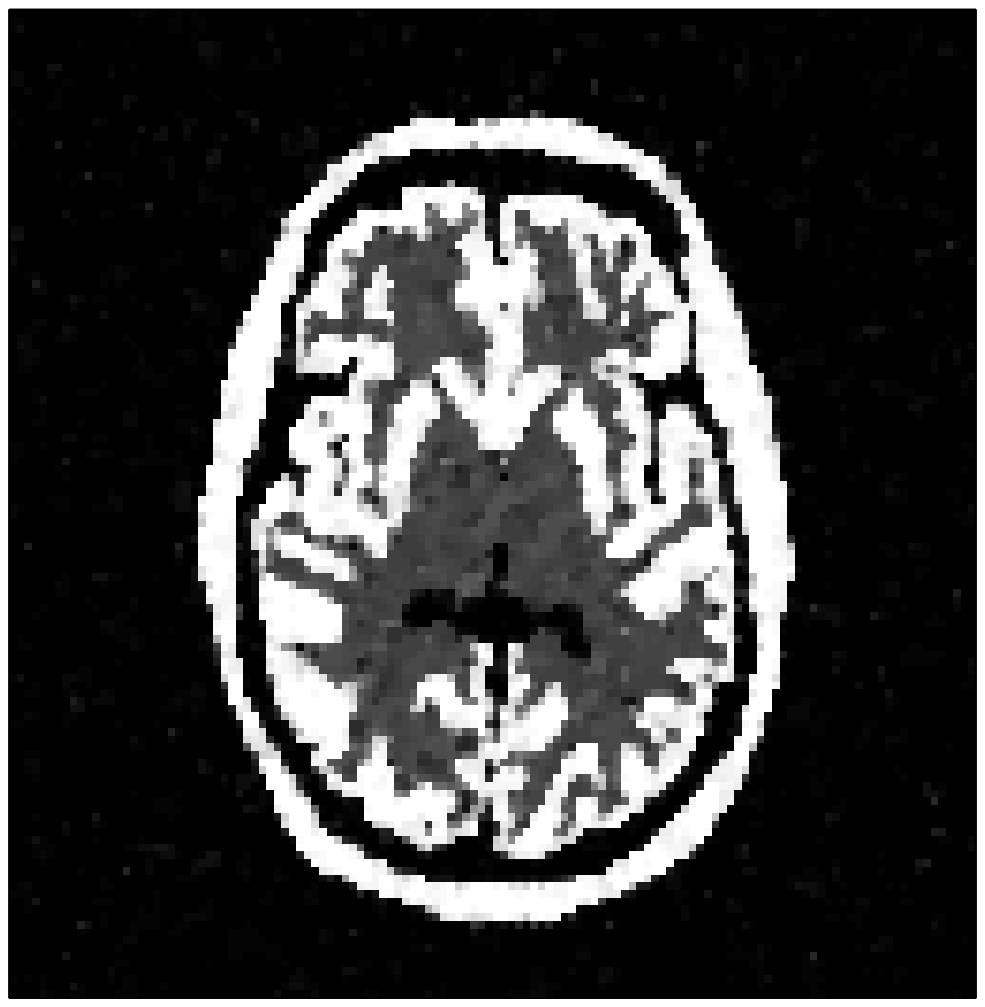}&
\includegraphics[height=5cm]{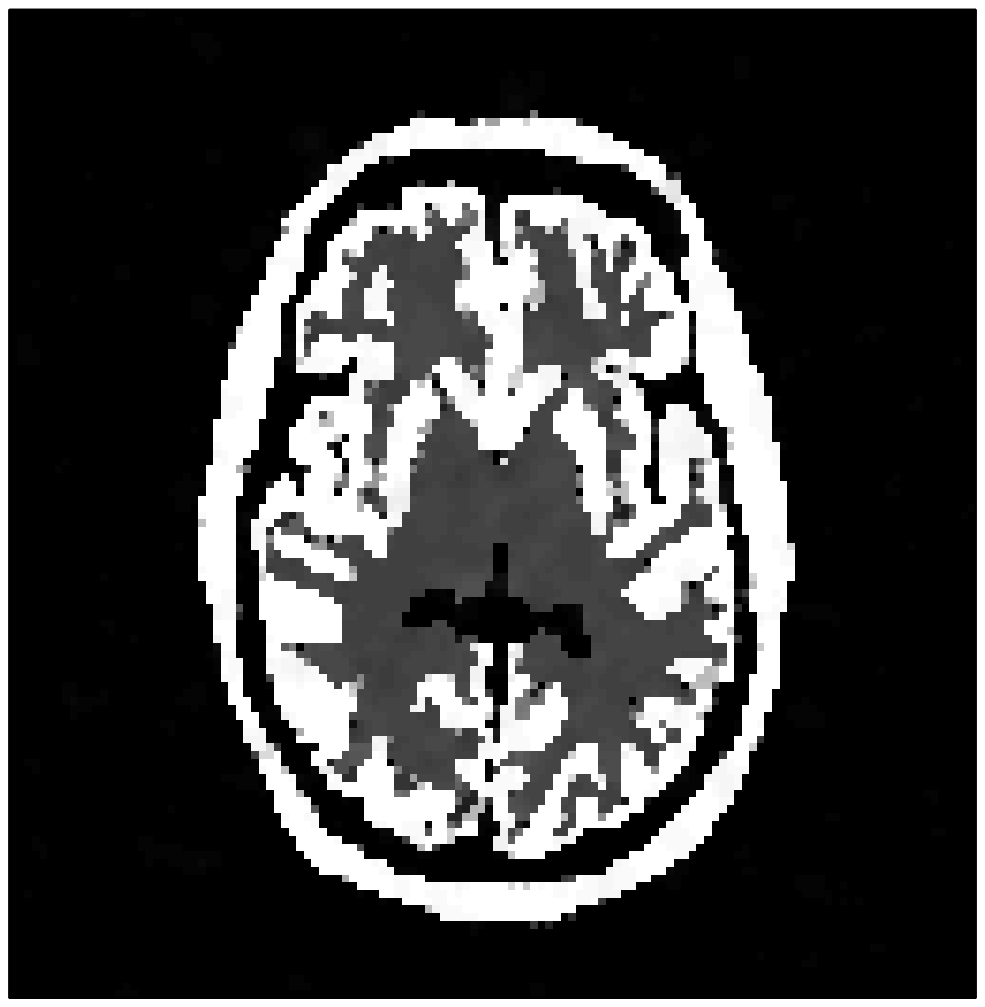}\\
$(a)$ & $(b)$
\end{tabular}
\caption{Reconstructed image using $(a)$ SC penalty function with 3MG, $\lambda = 0.06$, $\delta= 2.9$, $\rho = 1.6$, SNR = $18.05$ dB, MSSIM = $0.81$, or $(b)$ using SNC\ref{ex:psi1} penalty function with 3MG, $\lambda = 1.2$, $\delta= 11.1$, $\rho = 2.2$, SNR = $21.13$ dB, MSSIM = $0.92$.}
\label{Fig:SebTomo}
  \centering
\begin{tabular}{ccc}
\includegraphics[height=5cm]{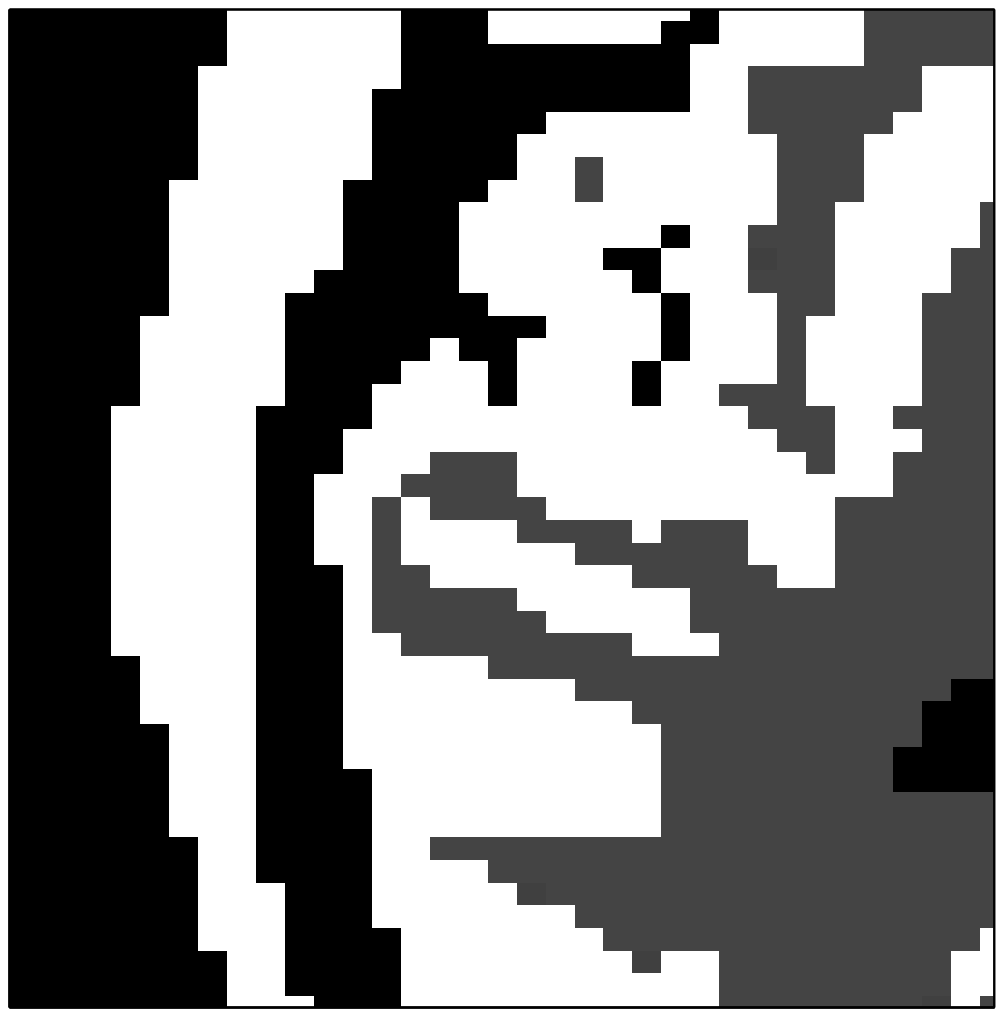}&
\includegraphics[height=5cm]{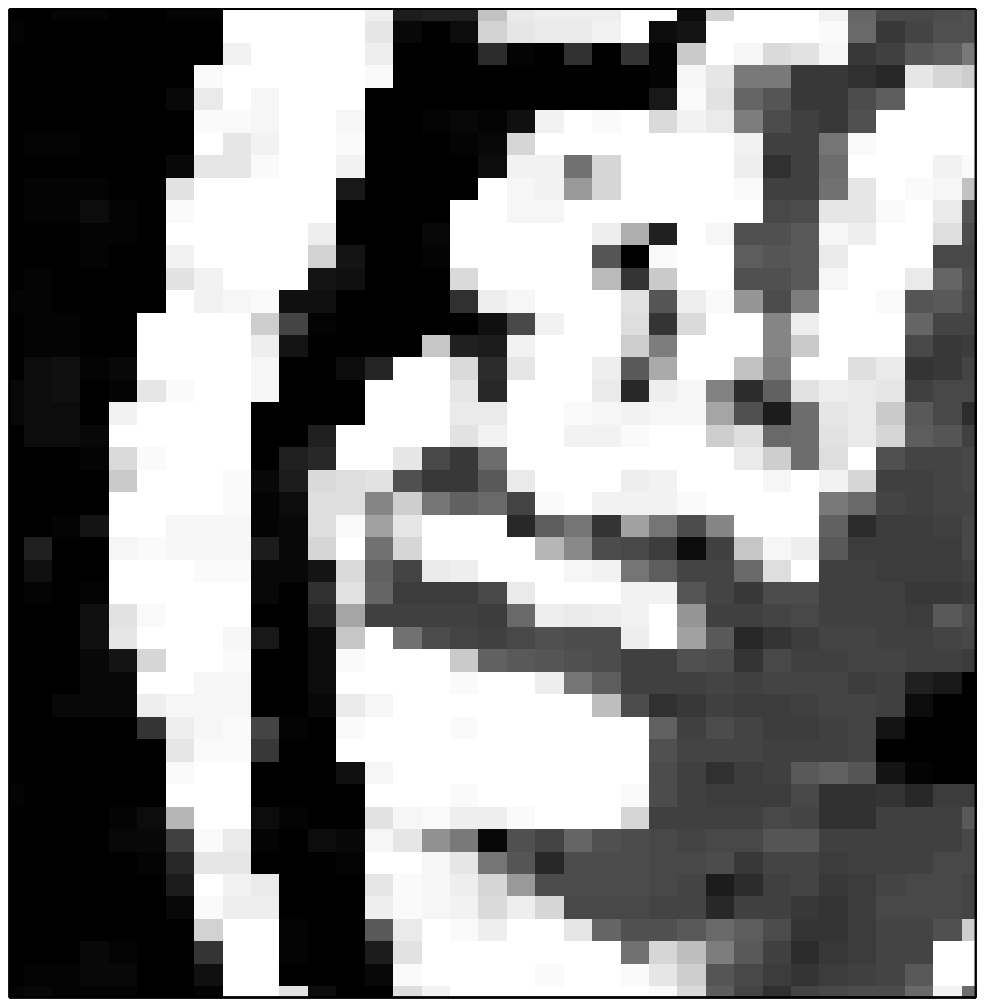}&
\includegraphics[height=5cm]{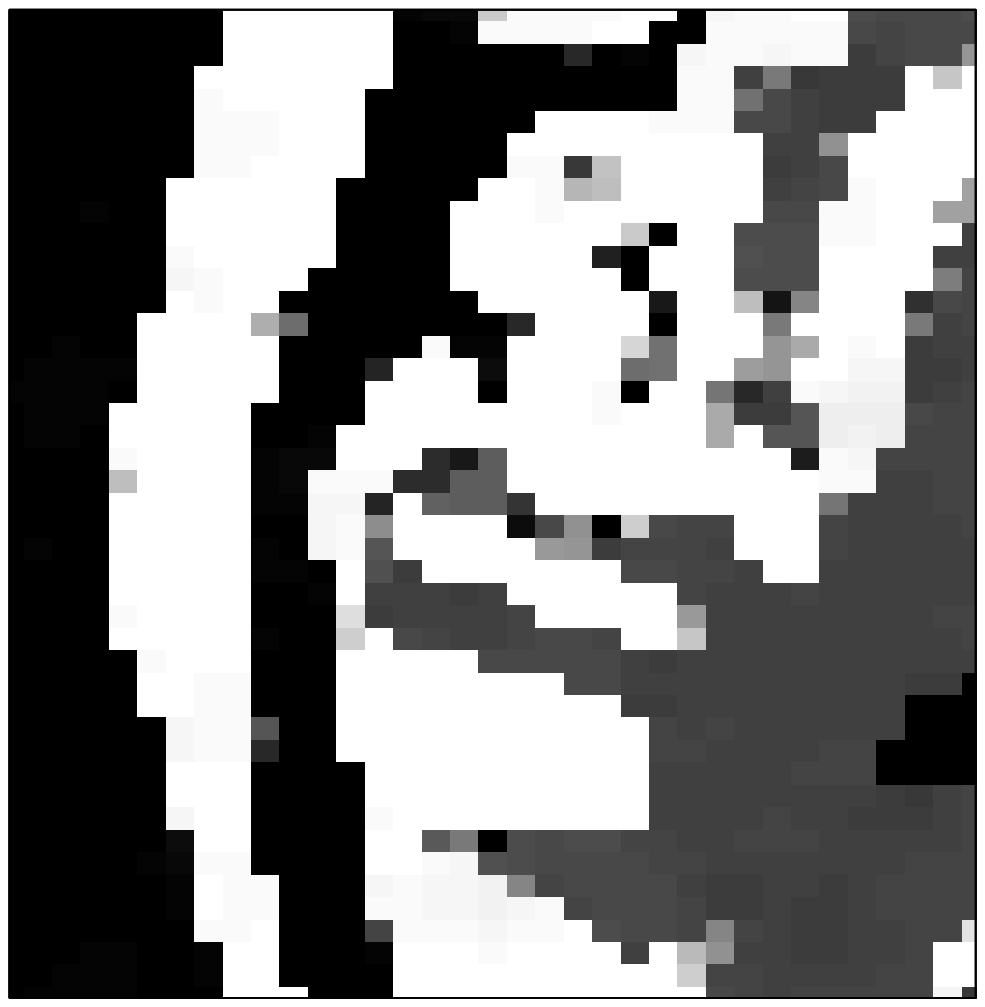}\\
$(a)$ & $(b)$ & $(c)$
\end{tabular}
\caption{$(a)$ Detail of the original image and corresponding reconstructions with $(b)$ convex penalty function and $(c)$ nonconvex penalty function.}
\label{Fig:SebTomoZoom}
\end{figure}

Figure~\ref{Fig:SebTomo} shows the results obtained for penalization strategies SC
and SNC\ref{ex:psi1}, with $(\lambda,\delta,\rho)$ tuned to maximize the SNR of
the restored image. We note that the SNC penalty leads to
better results in terms of reconstruction quality. In particular, it appears to
be well-suited to the reconstruction of the boundaries of the image, as demonstrated in
Figure~\ref{Fig:SebTomoZoom}. Tab.~\ref{Tab:Sebal} illustrates the performance of
the 3MG algorithm, in comparison with the three tested descent algorithms, when
either the SC or the SNC\ref{ex:psi1} penalty function is used. In
this example, the proposed algorithm outperforms the others, in terms of
both iteration number and computational time. In the nonconvex case, because of the
presence of local minimizers, the four algorithms do not lead to the same final
SNR value. It can be noticed that the smallest final criterion value is obtained
with the 3MG algorithm.

 \begin{table}[t]
\centering
\renewcommand{\arraystretch}{1.2}
\begin{tabular}{|c|c|c|c|c|c| }
\hline
Penalty function$(\lambda,\delta,\rho)$ & Algorithm &  Iteration & Time & $F_{\delta}$ &SNR\\
\hline
\hline
SC $(0.06, 2.9 , 1.6 )$ 
& 3MG & $253$ & \underline{$59.3$} & $1.1 \cdot 10^6$ & $18.05$\\ 
& \newtext[NLCG-HS] &  $358$ & $84.1$ & $1.1 \cdot 10^6$ & $18.05$\\
& \newtext[NLCG-PRP+] &  \newtext[$410$] & \newtext[$96.4$] & \newtext[$1.1 \cdot 10^6$] & \newtext[$18.05$]\\
& \newtext[NLCG-LS] & \newtext[$507$] & \newtext[$141.3$]  & \newtext[$1.1 \cdot 10^6$] & \newtext[$18.05$]\\  
& L-BFGS & $349$ & $82.3$ &    $1.1 \cdot 10^6$ & $18.05$\\
& HQ  &  $728$ & $337$ & $1.1 \cdot 10^6$ & $18.05$\\
\hline
\hline
SNC\ref{ex:psi1} $(1.2, 11.1 , 2.2)$ 
& 3MG & $516$ &  \underline{$119.8$} & $8.6214 \cdot 10^6$ & $21.13$ \\ 
& \newtext[NLCG-HS]  & $618$ & $143$ & $8.6228 \cdot 10^6$ & $20.89$ \\ 
& \newtext[NLCG-PRP+]  & \newtext[$876$] & \newtext[$204$] & $8.6229 \cdot 10^6$ & $20.89$ \\
& \newtext[NLCG-LS] & \newtext[$1212$] & \newtext[$360$]   & \newtext[$8.6228 \cdot 10^6$] & \newtext[$20.89$] \\   
& L-BFGS & $870$ & $203$ & $8.6225 \cdot 10^6$ & $21.17 $ \\ 
& HQ & $1152$ & $530$ & $8.6236 \cdot 10^6$ & $20.85$\\
\hline
\end{tabular}
\caption{Results for the tomography problem.}
\label{Tab:Sebal}
\end{table}

\section{Conclusion}
In this work, we have considered a class of smooth nonconvex regularization
functions and we have proposed an efficient minimization strategy for solving the
associated variational problems in imaging applications. Connections with
$\ell_0$ penalized problems were given asymptotically. In addition, a novel
convergence proof of the proposed subspace MM algorithm relying on the
Kurdyka-\L{}ojasiewicz inequality was given. Numerical experiments were
carried out to compare the proposed approach with other state-of-the art
continuous optimization methods (both for nonconvex and convex penalizations)
and with discrete optimization approaches dealing with a truncated quadratic
penalization. In the four presented image processing examples, we argue that
the proposed approach constitutes an appealing alternative to the existing
methods in terms of recovered image quality and computational time.
 \label{sec:conc}

\bibliographystyle{siam}
\bibliography{biben,revueabr,chouzenouxbib}

 \end{document}